\documentclass[11pt, a4paper, oneside, UKenglish]{amsart}
\usepackage[totalwidth=16cm,totalheight=24cm]{geometry}
\usepackage[utf8]{inputenc}
\usepackage{csquotes}

\usepackage[UKenglish]{babel}

    \usepackage{amsthm} 
    
     \usepackage{amsfonts} 
    
    \usepackage{mathtools} 
    
    \usepackage{amssymb} 
    
    \usepackage{amsmath} 
    
    \usepackage{tikz-cd} 

    \usepackage{tikz}
    
    \usepackage{etoolbox}

    \usepackage{faktor}

    \usepackage[export]{adjustbox}

    \usepackage{appendix} 
    
    \usepackage{hyperref} 
    
    \usepackage{graphicx} 
    \usepackage{subcaption}
    
    \usepackage{xcolor} 
    
    \usepackage{tikz} 

    \usepackage{url} 
    
    \usepackage{nameref} 
    \usepackage{fancyhdr} 

    \usepackage{enumitem} 

    \usepackage{wasysym}

\theoremstyle{plain}

\newtheorem{theorem}{Theorem}[subsubsection]
\newtheorem{theoremIntro}{Theorem}
\newtheorem{lemma}[theorem]{Lemma}

\newtheorem{proposition}[theorem]{Proposition}
\newtheorem{corollary}[theorem]{Corollary}

\newtheorem*{conjecture*}{Conjecture}
\newtheorem*{theorem*}{Theorem}
\newtheorem*{question*}{Question}

\theoremstyle{definition}
\newtheorem{definition}[theorem]{Definition}
\newtheorem*{definition*}{Definition}

\theoremstyle{remark}
\newtheorem{remark}[theorem]{Remark}

\numberwithin{equation}{section}

\newcommand{\addQEDstyle}[2]{\AtBeginEnvironment{#1}{\pushQED{\qed}\renewcommand{\qedsymbol}{#2}}\AtEndEnvironment{#1}{\popQED}}

\addQEDstyle{example}{$\triangle$}
\addQEDstyle{remark}{$\triangle$}

\DeclareMathOperator{\Prim}{Prim}

\DeclareMathOperator{\id}{id}

\DeclareMathOperator{\proj}{proj}

\DeclareMathOperator{\vspan}{span}

\DeclareMathOperator{\del}{\partial}
\DeclareMathOperator{\1T}{\text{1T}}
\DeclareMathOperator{\4T}{\text{4T}}
\DeclareMathOperator{\STU}{\text{STU}}
\DeclareMathOperator{\IHX}{\text{IHX}}
\DeclareMathOperator{\FI}{\scriptscriptstyle{\text{FI}}}
\DeclareMathOperator{\AS}{\text{AS}}

\DeclareMathOperator{\PTangles}{\textbf{PTangles}}

\DeclareMathOperator{\Diag}{\textbf{Diag}}

\DeclareMathOperator{\Sy}{\mathfrak{S}}

\DeclareMathOperator{\Bary}{Bar}
\DeclareMathOperator{\avg}{avg}

\newcommand{\epic}{\twoheadrightarrow}
\newcommand{\monic}{\xhookrightarrow{}}

\newcommand{\N}{\mathbb{N}}
\newcommand{\R}{\mathbb{R}}
\newcommand{\Z}{\mathbb{Z}}
\newcommand{\Q}{\mathbb{Q}}

\newcommand{\D}{\mathbb{D}}
\newcommand{\Ds}{\mathcal{D}}
\newcommand{\As}{\mathcal{A}}

\newcommand{\Fs}{\mathcal{F}}

\newcommand{\Ls}{\mathcal{L}}

\newcommand{\Ss}{\mathcal{S}}

\newcommand{\paren}[1]{\left( #1 \right)}

\newcommand{\abs}[1]{\lvert #1 \rvert}

\usepackage[style=alphabetic]{biblatex}
\author{Bruno Dular}
\address{Bruno Dular:
University of Luxembourg, FSTM, Department of Mathematics, 
Maison du nombre, 6 avenue de la Fonte,
L-4364 Esch-sur-Alzette, Luxembourg}
\email{bruno.dular@uni.lu}
\date{\today}

\title[Primitive Lie algebra of Feynman diagrams]{Primitive Feynman diagrams and the rational Goussarov--Habiro Lie algebra of string links}

\begin{document}
\maketitle

\begin{abstract}
    Goussarov--Habiro's theory of clasper surgeries defines a filtration of the monoid of string links $L(m)$ on $m$ strands, in a way that geometrically realizes the Feynman diagrams appearing in low-dimensional and quantum topology. Concretely, $L(m)$ is filtered by $C_n$-equivalence, for $n\geq 1$, which is defined via local moves that can be seen as higher crossing changes. The graded object associated to the Goussarov--Habiro filtration is the \emph{Goussarov--Habiro Lie algebra of string links} $\Ls L(m)$. We give a concrete presentation, in terms of primitive Feynman (tree) diagrams and relations ($\1T$, $\AS$, $\IHX$, $\STU^2$), of the rational Goussarov--Habiro Lie algebra $\Ls L(m)_{\Q}$. To that end, we investigate cycles in \emph{graphs of forests}: \emph{flip graphs} associated to forest diagrams and their $\STU$ relations. As an application, we give an alternative \emph{diagrammatic} proof of Massuyeau's rational version of the Goussarov--Habiro conjecture for string links, which relates indistinguishability under finite type invariants of degree $<n$ and $C_n$-equivalence.
\end{abstract}

\tableofcontents

\section{Introduction}
\subsection{Background}
A \emph{string link on $m$ strands in the cylinder} is a smooth embedding of a disjoint union of $m$ intervals into the cylinder
$$\gamma\colon\{1,\dots,m\}\times [0,1]\monic\D^2\times [0,1]$$
such that $\gamma(i,e)=(x_i,e)$ for all $i\in\{1,\dots,m\}$ and $e\in\{0,1\}$, where $x_1,\dots,x_m\in\D^2$ are fixed points\footnote{The points $x_i$ are usually ordered along the $x$-axis.}, and $\gamma$ is perpendicular to the boundary near those points\footnote{This last condition ensures that concatenation of string links is well-defined, i.e.\ the concatenation remains a smooth embedding.}. String links on $m$ strands are considered up to smooth isotopy\footnote{One can replace isotopies by \emph{ambient isotopies} by the \emph{smooth isotopy extension theorem}.}, and $L(m)$ denotes the corresponding set of isotopy classes. When there is no ambiguity, elements of $L(m)$ are also called \emph{string links}.

The set $L(m)$ is a \emph{monoid} with multiplication given by vertical concatenation (Figure \ref{fig:concatenation}), with the trivial string link as unit (Figure \ref{fig:trivial_string_link}).

\begin{figure}[!ht]
     \centering
     \begin{subfigure}[b]{0.3\textwidth}
         \centering
         \includegraphics[scale=1]{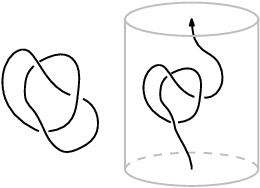}
         \caption{A knot and a long knot}
         \label{fig:knot_and_long_knot}
     \end{subfigure}
     \hfill
     \begin{subfigure}[b]{0.5\textwidth}
         \centering
         \includegraphics[scale=1]{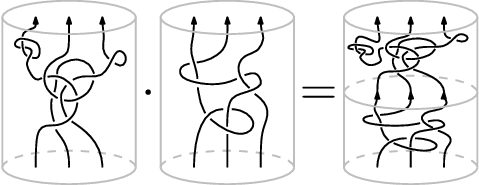}
         \caption{Two string links and their concatenation in $L(3)$}
         \label{fig:concatenation}
     \end{subfigure}
     \hfill
     \begin{subfigure}[b]{0.15\textwidth}
         \centering
         \includegraphics[scale=0.8]{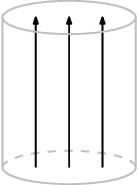}
         \caption{Trivial string link in $L(3)$}
         \label{fig:trivial_string_link}
     \end{subfigure}
        \caption{Long knots and string links}
        \label{fig:knot_long_knot_and_string_links}
\end{figure}

When $m=1$, elements of $L(1)$ are called \emph{long knots} and \emph{closing a long knot} yields an isomorphism between $L(1)$ and the monoid of knots under connected sum (Figure \ref{fig:knot_and_long_knot}). Thus, the study and classification of string links naturally belong to knot theory, and one of the main tools for probing them is the use of \emph{invariants}.

A string link invariant $V\colon L(m)\rightarrow A$ valued in an abelian group $A$ is a \emph{Vassiliev invariant} or \emph{finite type invariant} of \emph{degree $n$} if it satisfies specific \emph{skein relations} \cite{vassiliev1990cohomology,birman1993knot,bar1995vassiliev}, which can be read as the condition that $V$ vanishes on the $(n+1)$-st power of the augmentation ideal in the monoid ring $\Z L(m)$. Vassiliev invariants dominate\footnote{Here, \emph{dominate} means that the collection of all Vassiliev invariants determines the mentioned known invariants.} polynomial invariants (Conway, HOMFLY, etc.), quantum invariants \cite{Turaev+2016}, Milnor invariants \cite{meilhan2017surveyMilnor}, etc.\ but it is an open question whether they separate string links (or even knots, in the case $m=1$). One of the main results in that theory, due to Kontsevich, is that any $\Q$-valued Vassiliev invariant $V$ factors through the Kontsevich integral
\begin{equation}\label{eq:Kontsevich_integral}
    Z\colon L(m)\rightarrow\widehat{\As^{\FI}(m)_{\Q}}
\end{equation}
which is valued in the graded completion of the Hopf algebra $\As^{\FI}(m)_{\Q}$ of \emph{Feynman diagrams} on $m$ strands \cite{kontsevich1993vassiliev,bar1995vassiliev}, generated by uni-trivalent diagrams (with cyclic orientation at the trivalent vertices\footnote{The cyclic orientation is usually not indicated, in which case it is assumed to be counterclockwise.}, and univalent vertices attached to $m$ vertical strands, see Figure \ref{fig:examples_of_diagram}), modulo $\1T$ and $\STU$ relations (which imply the relations $\4T,\AS,\IHX$, see Figures \ref{fig:4T,STU,1T-relations}, \ref{fig:AS_and_IHX_relations}). The presence of the $\1T$ relation translates the fact that we are working with \emph{unframed string links}. This justifies the notation $\FI$ for \emph{framing independence}, used throughout this paper.

For the reason above, the Kontsevich integral is called a \emph{universal Vassiliev invariant over $\Q$}. The theory is much less understood over the integers.

\begin{figure}[!ht]
     \centering
     \begin{subfigure}[b]{0.4\textwidth}
         \centering
         \includegraphics[width=\textwidth]{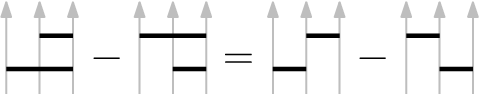}
         \caption{A $\4T$, \emph{four-terms}, relation}
         \label{fig:4T_relation}
     \end{subfigure}
     \hfill
     \begin{subfigure}[b]{0.25\textwidth}
         \centering
         \includegraphics[width=\textwidth]{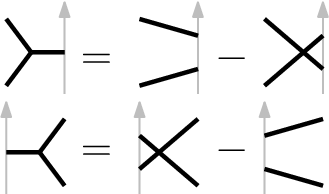}
         \caption{$\STU$ relations}
         \label{fig:STU_relation}
     \end{subfigure}
     \hfill
     \begin{subfigure}[b]{0.2\textwidth}
         \centering
         \includegraphics[width=0.7\textwidth]{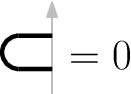}
         \caption{A $\1T$, \emph{one-term}, relation}
         \label{fig:1T_relation}
     \end{subfigure}
        \caption{The $\4T$, $\STU$ and $\1T$ relations}
        \label{fig:4T,STU,1T-relations}
\end{figure}

\begin{figure}[!ht]
     \centering
     \begin{subfigure}[b]{0.34\textwidth}
         \centering
         \includegraphics[width=0.78\textwidth]{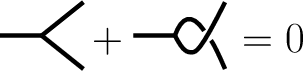}
         \caption{An $\AS$ relation}
         \label{fig:y equals x}
     \end{subfigure}
     \hfill
     \begin{subfigure}[b]{0.64\textwidth}
         \centering
         \includegraphics[width=0.60\textwidth]{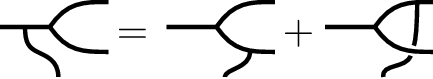}
         \caption{An $\IHX$ relation in the \emph{Kirchhoff form} \cite[Section 5.2.7]{chmutov2012introduction}}
         \label{fig:IHX_relation_Kirchhoff}
     \end{subfigure}
        \caption{The $\AS$ and $\IHX$ relations}
        \label{fig:AS_and_IHX_relations}
\end{figure}

A natural question about Vassiliev invariants is: what do they detect? In the case of knots ($m=1$), Goussarov \cite{goussarov1998interdependent} and Habiro \cite{habiro2000claspers} independently answered this question by introducing \emph{$C_n$-moves}, defined using \emph{clasper surgeries}. Those are local moves on string links that are modelled on \emph{tree claspers}, ribbon trees with leaves attached to the strands, that geometrically realize the Feynman diagrams from above. See Figure \ref{fig:C1_C2_moves} for examples of $C_1$- and $C_2$-moves. More precisely, $C_n$-moves generate the $C_n$-equivalence relation on $L(m)$, and they show that any two $C_n$-equivalent string links are $V_n$-equivalent, i.e.\ they cannot be distinguished by Vassiliev invariants of degree $<n$. They conjecture that the converse holds as well: this is the \emph{Goussarov--Habiro conjecture for string links in the cylinder} (which they prove for knots, i.e.\ the case $m=1$). This is still wide open, with only progress made in small degrees \cite{meilhanYasuhara2010stringlinks5}.

\begin{figure}[!ht]
    \centering
    \begin{subfigure}[b]{0.48\textwidth}
        \centering
        \includegraphics[width=0.8\textwidth]{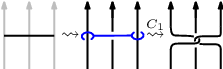}
        \caption{A $C_1$-move}
        \label{fig:C1_move}
    \end{subfigure}
    \hfill
    \begin{subfigure}[b]{0.5\textwidth}
        \centering
        \includegraphics[width=0.8\textwidth]{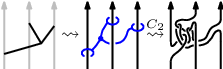}
        \caption{A $C_2$-move}
        \label{fig:C2_move}
    \end{subfigure}
       \caption{Diagrams, claspers realizing them and $C_n$-moves}
       \label{fig:C1_C2_moves}
\end{figure}

For each $n\geq 1$ the quotient monoid $L(m)/C_{n+1}$ turns out to be a nilpotent group, and $C_{n+1}$-equivalence classes of $C_n$-trivial string links form a finitely generated abelian group $\Ls_n L(m)\coloneqq L_n(m)/C_{n+1}$ \cite{goussarov2000finite,habiro2000claspers}. Those combine into a graded Lie $\Z$-algebra $\Ls L(m)\coloneqq\bigoplus_{n\geq 1}\Ls_n L(m)$, the \emph{Goussarov--Habiro Lie algebra of string links on $m$ strands}. From their work, it follows that there is a surjective $\Z$-linear graded \emph{realization map} $R\colon\Z\Ds^T(m)\epic\Ls L(m)$ that sends a tree diagram to the result of surgery along a tree clasper realizing it. Thus there is a presentation of the Goussarov--Habiro Lie algebra by tree diagrams on $m$ strands modulo specific relations $$\Ls L(m)\cong \Z\Ds^T(m)/\ker(R),$$
whose understanding would be a significant step towards the Goussarov--Habiro conjecture\footnote{The Goussarov--Habiro conjecture is equivalent to injectivity of the \emph{comparison map} $\chi\colon\Ls L(m)\rightarrow\As L(m)\colon [\gamma]\mapsto [\gamma-1]$, where $\As L(m)$ is the graded algebra associated with the Vassiliev filtration of the monoid ring $\Z L(m)$.}, while also being of its own interest. For example, it would provide a ``\emph{universal Goussarov--Habiro invariant}'', valued in a Lie algebra of tree diagrams. This is related to the study of $\pi_0$ of the embedding calculus tower \cite{conant2008stu2,budney2017embedding,kosanovic2020thesis}. The concordance analogous of this question, where $C_n$-equivalence is replaced by \emph{$C_n$-concordance} appears in works of Conant, Schneiderman and Teichner \cite{jimrobpeter2012whitney,jimrobpeter2016geometricfiltrations}, while the homotopy analogous has been settled by Habegger, Lin and Masbaum \cite{habeggerlin1990classification,habegger2000kontsevich}. Habiro and Massuyeau studied a similar question in the context of homology cylinders in \cite{habiro2009symplectic}. Nozaki, Sato and Suzuki also studied $\ker(R)$ in \cite{nozaki2022kernel}.

\newpage
\subsection{Results and strategy}
\subsubsection{Main result}
The main contribution of the present paper is an explicit presentation of the \emph{rational} Goussarov--Habiro Lie algebra $\Ls L(m)_{\Q}\coloneqq \Ls L(m)\otimes\Q$. In other words, we identify $\ker(R)_{\Q}$ as the subgroup generated by specific relations in $\Q\Ds^T(m)$: the usual $\1T,\AS,\IHX$ relations, completed with Conant's $\STU^2$ relation, see (\ref{eq:stu2_intro}). We also prove that this algebra of tree diagrams is the primitive Lie algebra of the Hopf algebra of Feynman diagrams $\As^{\FI}(m)_{\Q}$.

\begin{theoremIntro}[Presentation of the rational Goussarov--Habiro Lie algebra of string links]\label{thmI:main_result}
    There is an isomorphism of Lie $\Q$-algebras
    \begin{equation*}
        \Prim(\As^{\FI}(m)_{\Q})=\faktor{\Q\Ds^T(m)}{\langle\1T,\AS,\IHX,\STU^2\rangle}\xrightarrow{\cong}\Ls L(m)_{\Q}
    \end{equation*}
    defined by realizing tree diagrams using clasper surgeries, whose inverse is induced by the Kontsevich integral.
\end{theoremIntro}
The strategy of the proof is outlined below. It follows from Theorems \ref{thm:presentation_primitive_Lie_algebra}, \ref{prop:realizing_trees} and \ref{thm:tree_preservation_theorem} in the text.

\subsubsection{Strategy and side results}
For identifying the kernel of the realization map $R$, which takes tree diagrams to classes of string links, it is natural to look for a map in the other direction. This is exactly what the \emph{Kontsevich integral} (\ref{eq:Kontsevich_integral}) does \cite{kontsevich1993vassiliev,bar1995vassiliev}. Moreoever, it detects tree clasper surgeries, in a sense made precise by the following result\footnote{See \cite[Proposition E.24]{ohtsuki2002quantum} for a similar result in the context of the \emph{LMO invariant}.}.
\begin{theoremIntro}[Tree preservation theorem (Theorem \ref{thm:tree_preservation_theorem})]\label{thmI:Tree_preservation_theorem}
    Let $n\geq 1$ and $T\in\Ds^T_n(m)$ a degree $n$ tree diagram on $m$ strands. Then\footnote{The use of the big-$O$ notation indicates that this is an equality modulo terms of degree $\geq n+1$.}
    \begin{equation}
        Z(\sigma(C_T))=1+T+O(n+1),
    \end{equation}
    where $\sigma(C_T)$ denotes the string link obtained from the trivial one by surgery along a tree clasper $C_T$ realizing $T$, and $O(n+1)$ means that the equality holds modulo terms of degree $\geq n+1$.
\end{theoremIntro}
Consequently, the Kontsevich integral induces a map $Z^{GH}\colon\Ls L(m)_{\Q}\rightarrow\As^{\FI}(m)_{\Q}$, defined for $[\gamma]\in\Ls_n L(m)$ as $Z^{GH}([\gamma])\coloneqq Z_n(\gamma)$, whose image lies in the subspace generated by tree diagrams $\vspan_{\Q}(\text{trees})\subset\As^{\FI}(m)_{\Q}$. As in the case of knots, i.e.\ $m=1$, the subspace generated by tree diagams is precisely the primitive Lie algebra of $\As^{\FI}(m)_{\Q}$ (Theorem \ref{thm:comparison_primitive_and_size_filtration} in the text, see also (\ref{eq:primitive_filtration}) and Definition \ref{def:size_filtration}):
\begin{theoremIntro}[Primitive versus size filtration (Theorem \ref{thm:comparison_primitive_and_size_filtration})]\label{thmI:primitive_versus_size_filtration}
    The reduced primitive filtration and the size filtration of $\As^{\FI}(m)_{\Q}$ coincide. In particular, $\Prim(\As^{\FI}(m)_{\Q})$ is the submodule of $\As^{\FI}(m)_{\Q}$ generated by tree diagrams.
\end{theoremIntro}
As its name suggests, the \emph{size filtration} $F^{\bullet}\As^{\FI}(m)_{\Q}$ filters $\As^{\FI}(m)_{\Q}$ by the number of connected components of diagrams, i.e.\ $F^k\As^{\FI}(m)_{\Q}$ is generated by forest diagrams with $\leq k$ trees. The main ingredient of the proof of Theorem \ref{thmI:primitive_versus_size_filtration} requires averaging over permutations, which is the reason why we cannot extend Theorem \ref{thmI:primitive_versus_size_filtration} over $\Z$ with the current methods.

So far, the situation is summarized by the following commutative diagram:
\begin{equation}\label{diag:summary_1}\begin{tikzcd}
	{\Q\Ds^T(m)} && {\Prim(\As^{\FI}(m)_{\Q})} & {\As^{\FI}(m)_{\Q}} \\
	{\Ls L(m)_{\Q}}
	\arrow["R"', two heads, from=1-1, to=2-1]
	\arrow["{Z^{GH}}"', from=2-1, to=1-3]
	\arrow["\subseteq"{description}, draw=none, from=1-3, to=1-4]
	\arrow["\tilde{\iota}", two heads, from=1-1, to=1-3]
\end{tikzcd}\end{equation}
and the kernel of $\tilde{\iota}$ contains the $\1T,\AS,\IHX$ relations. It is a standard application of clasper calculus that the realization map $R$ satisfies those relations. But since $\STU$ does not preserve the size (or the primitive) filtration, it is not clear what other relations hold in $\Prim(\As^{\FI}(m)_{\Q})$, hence what other relations $R$ may satisfy.

Actually, over $\Q$ a single additional type of relation is needed: Conant's $\STU^2$ relation \cite{conant2008stu2}
\begin{equation}\label{eq:stu2_intro}
        \vcenter{\hbox{\includegraphics[scale=0.8]{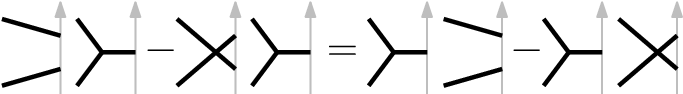}}}
    \end{equation}
where the four terms are tree (or forest) diagrams identical outside of the shown parts. As for the $\STU$ relation, we sometimes write $\STU^2$ as $F^{=Y}-F^{\times Y}=F^{Y=}-F^{Y\times}$. This relation preserves the size filtration.

Through the introduction of \emph{graph of forests}, we describe a concrete presentation by generators and relations of the primitive Lie algebra $\Prim(\As^{\FI}(m)_{\Q})$ (Theorem \ref{thm:presentation_primitive_Lie_algebra} in the text). The main ideas of the proof are outlined at the beginning of Section \ref{chap:lie_algebras_of_trees}. This constitutes the main technical result of this paper.
\begin{theoremIntro}[Primitive Lie algebra of trees (Theorem \ref{thm:presentation_primitive_Lie_algebra})]\label{thmI:primitive_Lie_algebra_of_trees}
The primitive Lie algebra $\Prim(\As^{\FI}(m)_{\Q})$ is naturally\footnote{Here, \emph{naturally} means that the isomorphism is induced by the map sending a diagram to itself.} isomorphic to the graded \emph{Lie algebra of trees}
\begin{equation}
	\Ls^{\FI}(m)_{\Q}\coloneqq\faktor{\Q\Ds^T(m)}{\langle\1T,\AS,\IHX,\STU^2\rangle}
\end{equation}
endowed with the bracket $[T,T']\coloneqq \overrightarrow{(T'T)(TT')}$ for any $T,T'\in\Ds^T(m)$ and extended linearly\footnote{See Definition \ref{def:vector_from_F_to_F'} for the $\overrightarrow{(\cdot)}$ notation.}.
\end{theoremIntro}

In the course of proving the above theorem, we actually obtain a concrete presentation of the primitive filtration of $\As^{\FI}(m)_{\Q}$ as well. For $k\geq 1$, let $\Ds^k(m)\subset\Ds^F(m)$ be the subset of size $k$ forest diagrams. A new relation is necessary:
\begin{definition*}
    A \emph{$\hexagon$-relation} is a sum of forest diagrams of the form
    \begin{equation}\label{eq:hexagon_intro}
        \vcenter{\hbox{\includegraphics[scale=0.8]{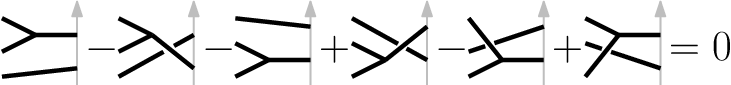}}}
    \end{equation}
    where all terms are identical forest diagrams outside of the part shown.
\end{definition*}
Like $\STU^2$, the $\hexagon$-relation preserves the size filtration and $\hexagon R\subset\Q\Ds(m)$ denotes the submodule generated by all $\hexagon$-relations. We obtain the following result, which is Theorem \ref{thm:presentation_primitive_filtration}:
\begin{theoremIntro}[Presentation of the primitive filtration (Theorem \ref{thm:presentation_primitive_filtration})]\label{thmI:presentation_of_primitive_filtration}
    For all $k\geq 1$, the \emph{graded $\Q$-module of size $k$ forests}
    \begin{equation}\label{eq:presentation_filtration}
        \Fs^{\FI,k}(m)_{\Q}\coloneqq\faktor{\Q\Ds^k(m)}{\langle\1T,\AS,\IHX,\STU^2,\hexagon R\rangle}
    \end{equation}
    is naturally isomorphic to the $k$-th step of the reduced primitive filtration of $\As_{\geq 1}^{\FI}(m)_{\Q}$. In particular, there is an isomorphism of filtrations
\[\begin{tikzcd}
	{\Ls^{\FI}(m)_{\Q}} & {\Fs^{\FI,1}(m)_{\Q}} & {\Fs^{\FI,2}(m)_{\Q}} & \dots & {\subset\As^{\FI}_{\geq 1}(m)_{\Q}} \\
	{\Prim(\As^{\FI}(m)_{\Q})} & {P^1\As^{\FI}_{\geq 1}(m)_{\Q}} & {P^2\As^{\FI}_{\geq 1}(m)_{\Q}} & \dots & {\subset\As^{\FI}_{\geq 1}(m)_{\Q}}
	\arrow[equal, from=1-5, to=2-5]
	\arrow[equal, from=1-1, to=1-2]
	\arrow[equal, from=2-1, to=2-2]
	\arrow["\iota"',"\cong", from=1-1, to=2-1]
	\arrow["\cong", from=1-2, to=2-2]
	\arrow["\cong", from=1-3, to=2-3]
	\arrow[hook, from=1-2, to=1-3]
	\arrow[hook, from=2-2, to=2-3]
	\arrow[hook, from=1-3, to=1-4]
	\arrow[hook, from=2-3, to=2-4]
\end{tikzcd}\]
where the bottom line is the reduced primitive filtration of $\As^{\FI}_{\geq 1}(m)_{\Q}$. Moreover, the vertical isomorphisms preserve the grading. 
\end{theoremIntro}

\begin{remark}
    In degree $n$ and size $n$, $\STU^2$ becomes $\4T$ and $\hexagon R$ is trivial. In degree $n$ and size $1$, $\hexagon R$ is trivial. In degree $n$ and size $<n-1$, $\hexagon R$ follows from $\STU^2$ and $\IHX$.
\end{remark}

Lastly, the above results combine into a concrete presentation of the rational Goussarov--Habiro Lie algebra, as mentioned at the beginning of this section:
\begin{theoremIntro}[Presentation of the rational Goussarov--Habiro Lie algebra]\label{thmI:presentation_rational_Habiro_Lie_algebra}
    There is an isomorphism of graded Lie $\Q$-algebras
    \begin{equation}
        R\colon\Ls^{\FI}(m)_{\Q}\longrightarrow\Ls L(m)_{\Q}
    \end{equation}
    that sends a tree diagram $T$ to the class $R(T)$ of the string link obtained by performing a clasper surgery along a tree clasper $C_T$ whose underlying diagram is $T$. By the tree preservation theorem, the inverse of $R$ is induced by the Kontsevich integral and denoted $Z^{GH}$.
\end{theoremIntro}

The commutative diagram \ref{diag:summary_1} can then be completed into
\begin{equation}
    \begin{tikzcd}
	{\Ls^{\FI}(m)_{\Q}} & {\Prim(\As^{\FI}(m)_{\Q})} & {\As^{\FI}(m)_{\Q}} \\
	{\Ls L(m)_{\Q}} && {\As L(m)_{\Q}}
	\arrow["{R}"',"\cong", from=1-1, to=2-1]
	\arrow["{R^v}"',"\cong",bend right=25, from=1-3, to=2-3]
	\arrow["\chi"', from=2-1, to=2-3]
	\arrow["\iota","\cong"', from=1-1, to=1-2]
	\arrow["\subset",hook, from=1-2, to=1-3]
	\arrow["{Z^{GH}}"', from=2-1, to=1-2]
	\arrow["{Z^v}"', bend right=25, from=2-3, to=1-3]
\end{tikzcd},
\end{equation}
where: $\As L(m)_{\Q}$ is the Hopf algebra of string links associated to finite-type invariants, $R$ is an isomorphism by Theorem \ref{thmI:presentation_rational_Habiro_Lie_algebra}, $R^v$ is an isomorphism by Kontsevich's theorem, and $\iota$ is an isomorphism by Theorem \ref{thmI:primitive_Lie_algebra_of_trees}. Consequently, this implies that the comparison $\chi$ is injective over $\Q$. This yields an alternative proof of the rational Goussarov--Habiro conjecture, which had already been shown by Massuyeau \cite{massuyeau2006finitetype} using completely different methods. This is explicited in Section \ref{sec:Application_to_rational_GHC}.

\subsection{Further directions}
\subsubsection{Towards the integral story}
In this work, we had to work over the rationals for two main reasons: the averaging over permutations and the use of the Kontsevich integral. However, many things still work over $\Z$: one can define an adequate Lie $\Z$-algebra of tree diagrams and it surjects onto the Goussarov--Habiro Lie algebra $\Ls L(m)$ as in Theorem \ref{prop:realizing_trees}. We plan on investigating new relations in this Lie $\Z$-algebra, and its comparison with the concordance version of the theory \cite{jimrobpeter2016geometricfiltrations} in future work.

\subsubsection{Framed string links}
In this paper, we restricted our attention to unframed string links in order to make the presentation clearer. However, the main algebraic ingredient (the identification of the primitive Lie algebra of $\As(m)_{\Q}$) holds with and without the $\1T$ relation. By considering the Kontsevich integral for \emph{framed string links} \cite{le1996universal}, one should obtain a presentation of the framed version of the Goussarov--Habiro Lie algebra.

\subsubsection{Embedding calculus approach}
We use the Kontsevich integral, a universal Vassiliev invariant over $\Q$, in an essential way to reduce the Goussarov--Habiro conjecture to a combinatorial question. Similarly, the existence of a universal Vassiliev invariant over $\Z$ would reduce the Goussarov--Habiro conjecture (over $\Z$!) to a purely combinatorial problem. A current candidate for such a universal invariant over $\Z$ is the \emph{Goodwillie--Weiss embedding calculus tower} \cite{budney2017embedding,kosanovic2020thesis,Boavida2021galois}. Actually, Conant's work \cite{conant2008stu2} on the $\STU^2$ relation was an important ingredient for identifying modules in the second page of the spectral sequence arising in the study of the Goodwillie--Weiss embedding calculus tower \cite{budney2017embedding,shi2023goodwillie}. The present work motivates the investigation of the string link analogous of the embedding calculus tower.

At the moment, the best known result about universality of the Goodwillie--Weiss embedding calculus tower is that it is a universal Vassiliev invariant over $\Z_{(p)}$ in degree $\leq p+1$ \cite[Theorem A]{Boavida2021galois}. Combining a \emph{string link} version of this result with the present work could lead to a proof of the Goussarov--Habiro conjecture over $\Z_{(p)}$ in degree $n<p$.

\subsection{Notations} Lower indices usually refer to \emph{degrees} of diagrams and claspers, and upper indices usually refer to filtrations of diagrams by \emph{size} (number of connected components). The notation $\FI$, for \emph{framing independence}, refers to the presence of the relation $\1T$. Sets of diagrams are denoted by $\Ds$ aggremented with various indices, Hopf algebras by $\As$ and Lie algebras by $\Ls$. 

\addtocontents{toc}{\protect\setcounter{tocdepth}{0}}
\subsection*{Aknowledgements}
This work was done under the supervision of Peter Teichner, whom I wish to thank for his guidance. I also wish to thank James Conant for helpful discussions and comments.
\addtocontents{toc}{\protect\setcounter{tocdepth}{2}}

\section{Primitive Lie algebra of Feynman diagrams}\label{chap:lie_algebras_of_trees}
The goal of this section is to identify the primitive Lie algebra $\Prim(\As^{\FI}(m)_{\Q})$ of the Hopf algebra $\As^{\FI}(m)_{\Q}$ Feynman diagrams on $m$ strands.

In the case of knots, i.e.\ $m=1$, tree diagrams generate the primitive Lie algebra\footnote{Since $\As^{\FI}(1)_{\Q}$ is commutative, this Lie algebra has a trivial bracket.} $\Prim(\As^{\FI}(1)_{\Q})$ of $\As^{\FI}(1)_{\Q}$ \cite{bar1995vassiliev}, which, by \cite{conant2008stu2}, is isomorphic to the space of trees $$\Q\Ds^T(1)/\langle\1T,\AS,\IHX,\STU^2\rangle.$$ See also \cite[Theorem G3]{kosanovic2020thesis} for a concise presentation of these results. Commutativity of $\As^{\FI}(1)_{\Q}$ implies, by the Milnor--Moore theorem, that there is an isomorphism between the primitive part $\Prim(\As^{\FI}(1)_{\Q})$ and the inseparable quotient $(\As^{\FI}(1)_{\Q})^{I}$ (where any diagram which is a product of non-trivial diagrams is set to zero). This fact is essential in Conant's result \cite{conant2008stu2}, which actually holds over $\Z$.

In the case of string links, tree diagrams also generate the primitive part $\Prim(\As^{\FI}(m)_{\Q})$. However, $\As^{\FI}(m)_{\Q}$ is not commutative when $m>1$ and the argument of Conant does not translate a priori. In this chapter, we prove that the conclusion still holds: $\Prim(\As^{\FI}(m)_{\Q})$ is isomorphic to $\Ls^{\FI}(m)_{\Q}\coloneqq\Q\Ds^T(m)/\langle\1T,\AS,\IHX,\STU^2\rangle$.

In Section \ref{sec:reminder_algebra_of_diagrams}, we recall the definitions of the algebra of diagrams $\As^{\FI}(m)$ \cite{bar1995vassiliev} and its bialgebra structure. In particular, we use a slightly different presentation of $\As^{\FI}(m)$, by forest diagrams instead of more general Feynman diagrams. This is more convenient for two reasons: the primitive part will be presented as a space of tree diagrams, and tree diagrams are more convenient\footnote{It makes the \emph{zip construction} easier to work with.} in Habiro's clasper calculus \cite{habiro2000claspers}, which is the topic of Chapter \ref{chap:claspers}.

In Section \ref{sec:Primitive filtration of forest diagrams}, we show that the primitive part of $\As^{\FI}(m)_{\Q}$ is indeed generated by tree diagrams. To that end, we introduce the \emph{(restricted) primitive filtration} $P^{\bullet}\As^{\FI}(m)_{\Q}$ \cite{cartier2007primer} and compare it with the \emph{size filtration} $F^{\bullet}\As^{\FI}(m)_{\Q}$, defined via the size\footnote{The \emph{size} of a forest diagram is its number of connected components, i.e.\ its number of trees.} of forest diagrams. The main ingredient is the definition of sections $s^{k+1}\colon F^{k+1}\As^{\FI}(m)_{\Q}\epic F^{k}\As^{\FI}(m)_{\Q}$ to the inclusions $i^k\colon F^{k}\As^{\FI}(m)_{\Q}\monic F^{k+1}\As^{\FI}(m)_{\Q}$. For a forest diagram $F$ of size $k+1$ composed of trees $T_1,\dots,T_{k+1}$, it is given by
\begin{equation*}
    s^{k+1}(F)=\frac{1}{(k+1)!}\sum_{\sigma\in\Sy_{k+1}}(F-(F)_{\sigma})
\end{equation*}
where $(F)_{\sigma}\coloneqq T_{\sigma(1)}\cdot\ldots\cdot T_{\sigma(k+1)}$ is a \emph{stacked version} of $F$. For $\sigma\in\Sy_{k+1}$, the forest diagram $(F)_{\sigma}$ contains the same trees as in $F$, but they are stacked above each other (in the order prescribed by $\sigma$) whereas the trees in $F$ may be shuffled together. One can obtain $F$ from $(F)_{\sigma}$ by a sequence of \emph{slide moves}: exchanges of adjacent legs belonging to distinct trees. By $\STU$, this allows us to write $F-(F)_{\sigma}$ as a sum of forest diagrams with $k$ trees: for each slide move, $\STU$ yields a term where the two involved trees are attached together. In degree $n$, the composition $$s^{2}\circ\dots\circ s^n\colon\As^{\FI}_n(m)_{\Q}\epic P^{1}\As_n^{\FI}(m)_{\Q}=Prim(\As^{\FI}(m)_{\Q})_n$$ yields a projector onto the submodule of primitive degree $n$ diagrams, which is thus generated by tree diagrams.

Section \ref{sec:permutographs_and_graphs_of_forests} develops a combinatorial framework aimed at finding a concrete presentation, by generators and relations, of the primitive part of $\As^{\FI}(m)_{\Q}$, together with a projection from $\As^{\FI}(m)_{\Q}$ onto it. The idea is simple: to investigate the minimal set of relations required for defining analogous of the inclusions $i^{k}$ and sections $s^{k+1}$, but now between concrete\footnote{By \emph{concrete} modules, we mean modules presented by explicit generators and relations, not submodules of larger modules.} modules of forest diagrams. In degree $n$, we start with the modules $\Q\Ds_n^k(m)/\langle\1T,\AS,\IHX\rangle$ generated by forest diagrams of size $k$ and degree $n$, for $1\leq k\leq n$. The $\STU$ relation is not \emph{size-homogeneous} and we cannot add it as it is. This is the source of the problem.

First, what are the minimal relations required for defining maps $$\tilde{\iota}^k\colon\frac{\Q\Ds_n^k(m)}{\langle\1T,\AS,\IHX\rangle}\dashrightarrow\frac{\Q\Ds_n^{k+1}(m)}{\langle\1T,\AS,\IHX\rangle}\;\text{ ?}$$ By $\STU$, a size $k$ forest should be sent to the difference obtained by breaking apart one trivalent vertex adjacent to a leg. However, the result may depend on the chosen vertex. Modding out by $\STU^2$ makes it well-defined \cite{conant2008stu2}. We obtain the sequence of maps
\begin{equation}\label{eq:sequence_of_inclusions_intro}
    \frac{\Q\Ds_n^1(m)}{\langle\1T,\AS,\IHX\rangle}\xrightarrow{\tilde{\iota}^1_n}\frac{\Q\Ds_n^1(m)}{\langle\1T,\AS,\IHX,\STU^2\rangle}\xrightarrow{\tilde{\iota}^2_n}\dots\xrightarrow{\tilde{\iota}^{n-1}_n}\frac{\Q\Ds_n^1(m)}{\langle\1T,\AS,\IHX,\STU^2\rangle}\cong\As^{\FI}_n(m)_{\Q}
\end{equation}
where the rightmost isomorphism follows from the fact that $\STU^2$ becomes $\4T$ for forests of size $n$ and degree $n$, i.e.\ chord diagrams.

Next, what additional relations are required for defining maps from right to left, i.e.\ sending size $k$ forests to sums of smaller forests? Starting from the \emph{ansatz} that the formula
\begin{equation}\label{eq:ansatz}
    \text{``}\pi^{k}\colon F\mapsto\frac{1}{k!}\sum_{\sigma\in\Sy_{k}}(F-(F)_{\sigma})\text{''}
\end{equation}
should work, we need to make sense of the expression ``$F-(F)_{\sigma}$''. More generally, we need to make sense of $F-F'$ for any two forests $F,F'$ made up of the same collection of trees $T_1,\dots,T_k$, but shuffled in different ways. As when working inside $\As^{\FI}(m)_{\Q}$ where $\STU$ holds, the expression $F-F'$ should equal a sum of smaller forest diagrams obtained by going from $F'$ to $F$ through a sequence of slide moves. Such a sequence of slide moves is interpreted as a \emph{path} in the \emph{graph of forests} $\Fs(T_1,\dots,T_k)$, whose vertices are the forests made up of the trees $T_1,\dots,T_k$ and whose (undirected) edges are given by \emph{slide moves}: exchanges of adjacent legs belonging to distinct trees. Such an edge is represented by ``$\vcenter{\hbox{\includegraphics[height=15pt]{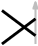}}}\leftrightarrow\vcenter{\hbox{\includegraphics[height=15pt]{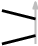}}}$'', where it is understood that both sides represent forest diagrams that are identical outside of the part shown, which involves legs of two distinct trees. To an oriented edge $F_{\times}=\vcenter{\hbox{\includegraphics[height=15pt]{Ressources/permutograph/STU2-satisfied-by-s/term-x.eps}}}\rightarrow F_{=}=\vcenter{\hbox{\includegraphics[height=15pt]{Ressources/permutograph/STU2-satisfied-by-s/term-=.eps}}}$ one can associate the forest diagram $$\overrightarrow{F_{\times}F_{=}}=\overrightarrow{\paren{\vcenter{\hbox{\includegraphics[height=25pt]{Ressources/permutograph/STU2-satisfied-by-s/term-x.eps}}}}\paren{\vcenter{\hbox{\includegraphics[height=25pt]{Ressources/permutograph/STU2-satisfied-by-s/term-=.eps}}}}}\coloneqq\vcenter{\hbox{\includegraphics[height=25pt]{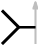}}},$$
in such a way that $F_{=}-F_{\times}=\overrightarrow{F_{\times}F_{=}}$ is an $\STU$ relation, when considering all the terms as elements of $\As^{\FI}(m)_{\Q}$. Then, to a path $P$ from $F$ to $F'$ one can associate an element $$\overrightarrow{P}\in\frac{\Q\Ds^{k-1}(m)}{\langle\1T,\AS,\IHX,\STU^2\rangle},$$ equal to the sum of forests associated to the edges traversed by $P$. In order to be able to write ``$F-F'=\overrightarrow{P}$'', we need to ensure that the expression $\overrightarrow{P}$ does not depend on the chosen path. Two different paths $P,P'$ from $F$ to $F'$ differ by a cycle $C$, and $\overrightarrow{P'}=\overrightarrow{P}+\overrightarrow{C}$. Therefore, we need to mod out the target module of forests by all $\overrightarrow{C}$, where $C$ runs over all cycles in graphs of forests of size $k$. However, this generates a large and non-explicit set of relations.

In order to obtain a smaller set of relations, we need to consider \emph{graphs of labelled forests}\footnote{The graph of labelled forests $\widetilde{\Fs}(T_1,\dots,T_k)$ differs from the graph of forests $\Fs(T_1,\dots,T_k)$ only when some trees among $T_1,\dots,T_k$ are identical, in which case a given forest has multiple possible labelings.}, denoted $\widetilde{\Fs}(T_1,\dots,T_k)$, whose vertices are forest diagrams made of $T_1,\dots,T_k$, with the difference that we remember the labeling of their trees by $1,\dots,k$. The edges are also given by slide moves. If the tree diagram $T_i$ has $n^j_i$ legs on the $j$-th strand, for $1\leq i\leq k$ and $1\leq j\leq m$, then edges of $\widetilde{\Fs}(T_1,\dots,T_k)$ correspond to basic transpositions in the product of permutation groups $\Sy_{n^1}\times\dots\times\Sy_{n^m}$ where $n^j\coloneqq\sum_{i}n^j_i$. More precisely, the edge (slide move) permuting the $i$-th and $(i+1)$-th legs (belonging to distinct trees) on the $j$-th strand corresponds to the transposition $\tau_i^j=(i, i+1)$ in the factor $\Sy_{n^j}$. Thus, a cycle $C$ in $\widetilde{\Fs}(T_1,\dots,T_k)$ corresponds to a permutation $\tau_{k_{l-1}}\dots\tau_{k_0}$ written as a product of basic transpositions. Since $C$ is a cycle, the product $\tau_{k_{l-1}}\dots\tau_{k_0}$ must equal the identity. Using the presentation
\begin{equation}\label{eq:presentation_Sn_intro}
            \Sy_{n^1}\times\dots\times\Sy_{n^m}=\left\langle\begin{matrix}\tau^{(1)}_1\dots\tau^{(1)}_{n^1-1}\\\cdots\\ \tau^{(m)}_1\dots\tau^{(m)}_{n^m-1}\end{matrix}\left|\begin{matrix}
        \tau^{(j)}_i\tau^{(j)}_i=1 & \text{for each }i,j\\
        \tau^{(j)}_i\tau^{(j)}_{i'}=\tau^{(j)}_{i'}\tau^{(j)}_i & \text{if }\abs{i-i'}>1\\
        \tau^{(j)}_{i}\tau^{(j)}_{i+1}\tau^{(j)}_{i}=\tau^{(j)}_{i+1}\tau^{(j)}_{i}\tau^{(j)}_{i+1} & \text{for each }i,j\\
        \tau^{(j)}_i\tau^{(j')}_{i'}=\tau^{(j)}_{i}\tau^{(j')}_{i'} & \text{if }j\neq j'
    \end{matrix}\right.\right\rangle
\end{equation}
of the product of groups of permutations $\Sy_{n^1}\times\dots\times\Sy_{n^m}$, we can turn $\tau_{k_{l-1}}\dots\tau_{k_0}$ into the empty word by a sequence of relations in (\ref{eq:presentation_Sn_intro}). Translating this sequence of relations turns the cycle $C$ into a composition of \emph{back-and-forth paths} ($\tau^{(j)}_i\tau^{(j)}_i=1$), squares for each \emph{far-commutativity relation} ($\tau^{(j)}_i\tau^{(j)}_{i'}=\tau^{(j)}_{i'}\tau^{(j)}_i$ when $\abs{i-i'}>1$ and $\tau^{(j)}_i\tau^{(j')}_{i'}=\tau^{(j)}_{i}\tau^{(j')}_{i'}$ when $j\neq j'$) and hexagons for each \emph{braiding relation} ($\tau^{(j)}_{i}\tau^{(j)}_{i+1}\tau^{(j)}_{i}=\tau^{(j)}_{i+1}\tau^{(j)}_{i}\tau_{i+1}$). Therefore, the element $\overrightarrow{C}\in\Q\Ds^{k-1}(m)$ is a sum of $\AS$ (Figure \ref{fig:AS_and_IHX_relations}), $\STU^2$ (\ref{eq:stu2_intro}) and $\hexagon$ relations (\ref{eq:hexagon_intro}).

The above paragraph concerns cycles in graphs of labelled forests $\widetilde{\Fs}(T_1,\dots,T_k)$, while we actually need to consider cycles in graphs of forests $\Fs(T_1,\dots,T_k)$. However, there is a surjection $\widetilde{\Fs}(T_1,\dots,T_k)\epic\Fs(T_1,\dots,T_k)$ and any path $P$ below lifts to a path $\tilde{P}$ above, after a choice of labelling $\tilde{F_0}$ of the starting point $F_0$ of $P$. Moreover, $\overrightarrow{\tilde{P}}=\overrightarrow{P}$. If $P$ is a cycle, i.e.\ $F_l=F_0$, then $\tilde{P}$ is not necessarily a cycle: even though its endpoints $\tilde{F_l}$ and $\tilde{F_0}$ have the same underlying \emph{unlabelled} forest $F_l=F_0$, their labeling may differ by a permutation $\sigma$. Lifting $P$ again, but now starting from $\tilde{F_l}$, and repeating this $N$ times, where $N$ equals the order of $\sigma$ in $\Sy_{n}$, yields a cycle in $\widetilde{\Fs}(T_1,\dots,T_k)$ which lifts the $N$-fold composition of the cycle $P$ in $\Fs(T_1,\dots,T_k)$. Thus $N\cdot\overrightarrow{C}$ is a sum of $\AS,\STU^2$ and $\hexagon$ relations as well, hence so is $\overrightarrow{C}$, since we are working over $\Q$.

Therefore, for two size $k$ forests $F$ and $F'$ made up of the same trees, we can finally define
$$\overrightarrow{F'F}\coloneqq\overrightarrow{P}\in\Fs^{\FI,k-1}(m)_{\Q}\coloneqq\frac{\Q\Ds^{k-1}(m)}{\langle\1T,\AS,\IHX,\STU^2,\hexagon R\rangle},$$
and in Section \ref{sec:merging trees} we can make sense of the \emph{ansatz} (\ref{eq:ansatz}) by defining
\begin{equation}
    \tilde{\pi}^k\colon\Q\Ds^{k}(m)\rightarrow\Fs^{\FI,k-1}(m)_{\Q}\colon F\mapsto\frac{1}{k!}\sum_{\sigma\in\Sy_{k}}\overrightarrow{(F)_{\sigma}F}.
\end{equation}
After checking that the maps $\tilde{\iota}$'s and $\tilde{\pi}$'s satisfy all the relations, the sequence of maps (\ref{eq:sequence_of_inclusions_intro}) becomes a sequence of inclusions, with the $\pi$'s as sections:
\begin{equation}\label{eq:chain_of_inclusion_modules_of_forests}
    \begin{tikzcd}
	{\Fs^{\FI,1}_n(m)_{\Q}} & {\Fs^{\FI,2}_n(m)_{\Q}} & \dots & {\Fs^{\FI,n}_n(m)_{\Q}}
	\arrow["{\iota^1_n}", hook, from=1-1, to=1-2]
	\arrow["{\iota^{2}_n}", hook, from=1-2, to=1-3]
	\arrow["{\iota^{n-1}_s}", hook, from=1-3, to=1-4]
\end{tikzcd}.
\end{equation}

In Section \ref{sec:presentation_of_primitive_filtration}, we gather the above results and identify the filtration (\ref{eq:chain_of_inclusion_modules_of_forests}) with the primitive or size filtration $F^{\bullet}\As^{\FI}(m)_{\Q}$. The leftmost module in (\ref{eq:chain_of_inclusion_modules_of_forests}) is denoted $\Ls^{\FI}_n(m)_{\Q}$, it is the degree $n$ part of the \emph{Lie algebra of trees}. The rightmost module is $\As^{\FI}_n(m)_{\Q}$, it is the degree $n$ part of the \emph{Hopf algebra of forests}\footnote{It is the well-known Hopf algebra of chord diagrams or, equivalenlty, Feynman diagrams.}. All in all, this shows that $\Ls^{\FI}(m)_{\Q}\coloneqq\bigoplus_{n\geq 1}\Ls^{\FI}_n(m)_{\Q}$ injects into $\As^{\FI}(m)_{\Q}$ and is isomorphic to the primitive Lie algebra $\Prim(\As^{\FI}(m)_{\Q})$. Since the $\hexagon$ relation requires at least two trees, it is trivial in $\Ls^{\FI}(m)_{\Q}$. We finally obtain, as desired, $$\Prim(\As^{\FI}(m)_{\Q})\cong\frac{\Q\Ds^T(m)}{\langle\1T,\AS,\IHX,\STU^2\rangle}.$$

\subsection{Algebras of diagrams}\label{sec:reminder_algebra_of_diagrams}
This section recalls the definition of chord and Feynman diagrams \cite{bar1995vassiliev} on tangle skeletons, as well as their bialgebra structure.

\subsubsection{Uni-trivalent graphs.}
A \emph{uni-trivalent graph}, also called \emph{Feynman} or \emph{Jacobi graph}, is an undirected graph $G=(V=U\sqcup T,E)$ consisting of univalent and trivalent vertices (forming the sets $U$ and $T$, respectively). The edges are defined as pairs of distinct vertices\footnote{We could allow diagrams with edges attached to a single vertex, i.e.\ \emph{tadpoles}, but those will vanish by $\AS$.}. The trivalent vertices are endowed with a cyclic orientation of their three adjacent edges. When a uni-trivalent graph is drawn, the cyclic orientations are the counterclockwise ones unless indicated otherwise.

The univalent and trivalent vertices are often called \emph{leaves} and \emph{nodes}, respectively. An edge adjacent to a leaf is called a \emph{leg}.

The \emph{degree} of a uni-trivalent graph $G$ is $\deg(G)\coloneqq\frac{\abs{V}}{2}$.

A uni-trivalent graph $G$ is a \emph{forest} if it has no cycles and a \emph{tree} if in addition it has a single component. Note that, if $G$ is a tree, then $\deg(G)=\abs{U}-1=\abs{T}+1$.

\subsubsection{Uni-trivalent forest and chord diagrams.}\label{sec:Uni-trivalent, forest and chord diagrams}
Let $S$ be an oriented $1$-manifold with boundary. Later $S$ will be the underlying manifold of what will be referred to as a \emph{tangle skeleton}. Since this work focuses on string links, $S$ will usually be a disjoint union of intervals $\{1,\dots,m\}\times [0,1]$.

A \emph{uni-trivalent} or \emph{Feynman} or \emph{Jacobi diagram} on $S$ is a pair $D=(G,[u])$ where $G=(U\sqcup T,E)$ is a uni-trivalent graph and $[u]$ is the isotopy class of an embedding $u\colon U\monic S-\del S$ of the univalent vertices into the interior of $S$. The set of uni-trivalent diagrams on $S$ is denoted $\Ds(S)$.

The \emph{degree} of a uni-trivalent diagram is the degree of its underlying graph. Denote by $\Ds_n(S)\subset\Ds(S)$ the subset of degree $n$ uni-trivalent diagrams. This endows $\Ds(S)$ with a graded set structure.

A uni-trivalent diagram is a \emph{forest}, resp. \emph{tree}, diagram if its underlying graph is a forest, resp. a tree. A forest diagram is a \emph{chord diagram} if it has no nodes. Denote by $\Ds^F(S),\Ds^T(S),\Ds^c(S)$ the corresponding graded subsets of $\Ds(S)$. The \emph{size} of a forest diagram is its number of trees, i.e.\ the number $b_0(G)$ of connected components of the underlying uni-trivalent graph $G$. For example, the forest diagram on Figure \ref{fig:Forest_diagram} has size $3$. For $s\geq 0$, the set of forest diagrams of size $s$ is denoted $\Ds^s(S)$. Note that $\Ds^{n}_n(S)=\Ds^c_n(S)$ and $\Ds^T(m)\coloneqq\Ds^1(m)$ is the set of tree diagrams.

\begin{figure}[!ht]
    \centering
     \begin{subfigure}[b]{0.3\textwidth}
         \centering
         \includegraphics[scale=1]{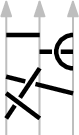}
         \caption{Chord diagram}
         \label{fig:chord_diagram}
     \end{subfigure}
     \hfill
     \begin{subfigure}[b]{0.3\textwidth}
         \centering
         \includegraphics[scale=1]{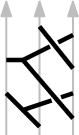}
         \caption{Forest diagram}
         \label{fig:Forest_diagram}
     \end{subfigure}
     \hfill
     \begin{subfigure}[b]{0.3\textwidth}
         \centering
         \includegraphics[scale=1]{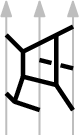}
         \caption{Feynman diagram}
         \label{fig:Feynman_diagram}
     \end{subfigure}
    \caption{Examples of diagrams}
    \label{fig:examples_of_diagram}
\end{figure}

By \cite{bar1995vassiliev}, the inclusions $\Ds^c(S)\subset\Ds^F(S)\subset\Ds(S)$ induce isomorphisms\footnote{In \cite{bar1995vassiliev}, it is shown that $\Ds^c(S)\monic\Ds(S)$ induces an isomorphism $\Z\Ds^c(S)/\langle\4T\rangle\cong\Z\Ds(S)/\langle\STU\rangle$. Thus, we get $\Z\Ds^c(S)/\langle\4T\rangle\monic\Z\Ds^F(S)/\langle\STU\rangle\epic\Z\Ds(S)/\langle\STU\rangle$ and the first map is surjective since any forest can be written as a sum of chord diagrams by applying $\STU$ until it has no node left. Hence the first map is an isomorphism, so is the second since their composition is an isomorphism. The same holds with $\FI$ added everywhere.} of graded $\Z$-modules
\begin{equation}\label{eq:definition_algebra_diagrams}
    \As(S)\coloneqq\frac{\Z\Ds^c(S)}{\langle\4T\rangle}\cong\frac{\Z\Ds^F(S)}{\langle\STU\rangle}\cong\frac{\Z\Ds(S)}{\langle\STU\rangle}
\end{equation}
and the relations $\AS,\IHX$ are consequences of $\STU$ (see Figures \ref{fig:4T,STU,1T-relations} and \ref{fig:AS_and_IHX_relations} for descriptions of all the involved relations). Unless stated otherwise, we use the presentation of $\As(S)$ using forest diagram. The degree $n$ part of $\As(S)$ is denoted $\As_n(S)$. The only degree $0$ diagram is the empty one, hence $\As_0(S)\cong\Z$.

\subsubsection{Framed versus unframed.}
There is an \emph{unframed} or \emph{framing independent} ($\FI$) version of the module of Feynman diagrams
\begin{equation}\label{eq:unframed_algebra}
    \As^{\FI}(S)\coloneqq\As(S)/\langle\1T\rangle,
\end{equation}
where $\1T$ is the submodule generated by all diagrams containing an \emph{isolated chord}, i.e.\ a chord whose endpoints lie on the same strand with no other endpoint in between them. 

Writing $A^{(FI)}(S)$ indicates that the statement holds for both $A(S)$ and $A^{FI}(S)$.

\subsubsection{Coalgebra structure.}
The graded $\Z$-module $\As^{(\FI)}(S)$ has a structure of coalgebra, with comultiplication \cite[Definition 3.7]{bar1995vassiliev}
\begin{equation}
    \Delta(D)\coloneqq\sum_{J\subset\pi_0(D)} D_{J}\otimes D_{\pi_0(D)-J},
\end{equation}
where the sum runs over all the subsets $J$ of the set $\pi_0(D)\coloneqq\pi_0(G)$ of connected components of $D$ and $D_{J}$ denotes the subdiagram of $D$ consisting of those components contained in $J$. The counit $\epsilon$ is given by $\epsilon(D)=1$ if $D$ is the empty diagram, and it is zero otherwise.

For example,
\begin{align*}
    \Delta\paren{\vcenter{\hbox{\includegraphics[height=25pt]{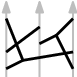}}}}=&\vcenter{\hbox{\includegraphics[height=25pt]{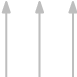}}}\otimes\vcenter{\hbox{\includegraphics[height=25pt]{Ressources/Delta/123.eps}}}\;+\;\vcenter{\hbox{\includegraphics[height=25pt]{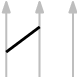}}}\otimes\vcenter{\hbox{\includegraphics[height=25pt]{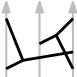}}}\;+\;\vcenter{\hbox{\includegraphics[height=25pt]{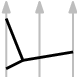}}}\otimes\vcenter{\hbox{\includegraphics[height=25pt]{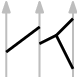}}}\;+\;\vcenter{\hbox{\includegraphics[height=25pt]{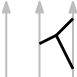}}}\otimes\vcenter{\hbox{\includegraphics[height=25pt]{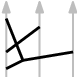}}}\\ &\;+\;\vcenter{\hbox{\includegraphics[height=25pt]{Ressources/Delta/12.eps}}}\otimes\vcenter{\hbox{\includegraphics[height=25pt]{Ressources/Delta/3.eps}}}\;+\;\vcenter{\hbox{\includegraphics[height=25pt]{Ressources/Delta/13.eps}}}\otimes\vcenter{\hbox{\includegraphics[height=25pt]{Ressources/Delta/2.eps}}}\;+\;\vcenter{\hbox{\includegraphics[height=25pt]{Ressources/Delta/23.eps}}}\otimes\vcenter{\hbox{\includegraphics[height=25pt]{Ressources/Delta/1.eps}}}\;+\;\vcenter{\hbox{\includegraphics[height=25pt]{Ressources/Delta/123.eps}}}\otimes\vcenter{\hbox{\includegraphics[height=25pt]{Ressources/Delta/empty.eps}}}.
\end{align*}

Note that the comultiplication is co-associative and co-commutative.

\subsubsection{Hopf algebra of diagrams on string links.}
When $S=\{1,\dots,m\}\times [0,1]$ is a disjoint union of intervals, we write $\Ds(m)$ for $\Ds(S)$, $\As(m)$ for $\As(S)$, and so on. In that case, vertical stacking of diagrams additionaly endows $\As(m)$ and $\As^{\FI}(m)$ with an algebra structure. More precisely, $D\cdot D'$ denotes the vertical concatenation of $D$ and $D'$, with $D$ above $D'$. The unit $1$ is the empty diagram.

\begin{proposition}[Bar-Natan {\cite[Proposition 3.9]{bar1995vassiliev}}]
    The comultiplication and multiplication defined above endow $\As^{(\FI)}(m)$ with the structure of a connected co-commutative Hopf algebra\footnote{Bar-Natan does not explicit the existence of an antipodal over $\Z$, but one can check that we can define it similarly to the antipodal in the \emph{Connes-Kreimer} Hopf algebra.}. It is commutative when $m=1$.
\end{proposition}

\subsection{Primitive filtration of forest diagrams}\label{sec:Primitive filtration of forest diagrams}
\subsubsection{Primitive filtration of a Hopf algebra.}
Let us recall some basic notions about primitive elements in Hopf algebras. See \cite{milnor1965structure}, \cite[Chapter 5]{montgomery1993hopf}, \cite{cartier2007primer} and \cite{ronco2011shuffle} for more details.

Fix a Hopf algebra $A$ over $\Q$ with multiplication $\mu$ or $\cdot$, unit $1$, comultiplication $\Delta$ and counit $\epsilon$. An element $x\in A$ is \emph{primitive} if $\Delta(x)=x\otimes 1+1\otimes x$. A direct calculation shows that the commutator of two primitive elements is primitive as well. In particular, the set $\Prim(A)$ of primitive elements in $A$ is a Lie algebra, called the \emph{primitive Lie algebra of $A$}. Conversely, to any Lie algebra $L$ one can associate its \emph{universal enveloping algebra} $U(L)$, which is a Hopf algebra. The Milnor--Moore theorem states conditions under which the functors $\Prim$ and $U$ are inverses of each other.

\begin{theorem}[Milnor--Moore theorem {\cite[Theorem 5.6.5]{montgomery1993hopf}}]
    Let $A$ be a connected\footnote{A Hopf algebra $A$ is \emph{connected} if $A_0\coloneqq\Q\cdot 1$ is its unique simple subcoalgebra.} and co-commutative Hopf algebra over a characteristic zero field, then the inclusion $\Prim(A)\monic A$ induces an isomorphism of Hopf algebras $U(\Prim(A))\cong A$.
\end{theorem}

Now, let us assume that $A$ is a connected and co-commutative Hopf algebra over $\Q$, so that $U(\Prim(A))\cong A$. The \emph{primitive filtration} of $A$ is the filtration
\begin{equation}\label{eq:primitive_filtration}
    \Q\cdot 1=P^0A\monic P^1A\monic P^2A\monic\dots
\end{equation}
where $P^kA$ is the $\Q$-submodule of $A$ additively generated by $\Q\cdot 1$ and all products of $\leq k$-many primitive elements. In particular $P^1A=\Q\oplus\Prim(A)$ and any element $x\in A$ is in some $P^k A$, for some $k$, since $A$ is primitively generated. Under the isomorphism $U(\Prim(A))\cong A$, the primitive filtration coincides with the \emph{Lie filtration} of $U(\Prim(A))$ \cite[Prop.\ 5.17]{milnor1965structure}.

For all $k\geq 0$, we have $P^kA=\Q\oplus P^kA_{\geq 1}$ where $A_{\geq 1}\coloneqq A_+$ is the kernel of the counit. The filtration
\begin{equation}\label{eq:reduced_primitive_filtration}
    0=P^0A_{\geq 1}\monic P^1A_{\geq 1}\monic P^2A_{\geq 1}\monic\dots
\end{equation}
is called the \emph{reduced primitive filtration} of $A_{\geq 1}$.

\subsubsection{Primitive filtration of forests.}
The focus returns to the Hopf algebra of forest diagrams on $m$ strands, $\As^{(\FI)}(m)_{\Q}$, which is indeed connected and co-commutative, hence the Milnor--Moore theorem applies, i.e.\ $U(\Prim(\As(m)_{\Q}))\cong\As(m)_{\Q}$, and it is filtered by the primitive filtration.

In the case $m=1$, Bar-Natan showed that $\Prim(\As(1)_{\Q})$ is (additively) generated by all connected diagrams \cite{bar1995vassiliev} and his argument translates well for $m\geq 1$. Actually, this statement is also true over $\Z$, as shown by Lando in \cite{lando2primitive} via a beautiful counting argument. Unfortunately, the argument uses commutativity in $\As(1)$, hence it does not translate directly to the general case $m\geq 1$. Instead, we give an argument that compares the primitive filtration to the size filtration defined below, in a way parallel to the ideas of the subsequent subsections in this section. First, let us introduce some notations.

\begin{definition}\label{def:size_filtration}
    The \emph{size filtration} of $\As(m)$ is the filtration $\{F^k\As(m)\}_{k\geq 0}$ where $F^k\As(m)$ is the $\Z$-submodule generated by forest diagrams of size $\leq k$. In degree $n\geq 0$, it induces a filtration\footnote{The only forest diagram of size $0$ is the empty diagram $1$, hence $F^0\As(m)=\Z\cdot 1$.}
    \begin{equation}
        F^0\As_n(m)\monic F^1\As_n(m)\monic\dots\monic F^{n-1}\As_n(m)\monic F^{n}\As_n(m)=\As_n(m)
    \end{equation}
    of $\As_n(m)$. The \emph{size filtration} of $\As^{\FI}(m)$ is defined in the same way, with $\FI$ added everywhere. Tensoring everything with $\Q$ yields the size filtration of $\As(m)_{\Q}$.
\end{definition}

For $F\in\Ds^F(m)$ a forest diagram, define the \emph{reduced co-multiplication} $\overline{\Delta}(F)$ as
\begin{equation}
    \overline{\Delta}(F)\coloneqq\sum_{\emptyset\neq J\subsetneq\pi_0(F)}\paren{F_J}\otimes\paren{F_{\pi_0(F)-J}}
\end{equation}
where $F_J$ denotes the subdiagram of $F$ that consists of the components contained in $J$. In other words, $\overline{\Delta}(F)=\Delta(F)-F\otimes 1-1\otimes F$. Thus, $\overline{\Delta}$ satisfies $\AS,\IHX,\STU$ (and $\1T$) and descends to a well-defined map
\begin{equation}
    \overline{\Delta}\colon\As_{\geq 1}(m)\longrightarrow\As_{\geq 1}(m)\otimes\As_{\geq 1}(m).
\end{equation}
Define recursively $\overline{\Delta}^{k}\coloneqq (\overline{\Delta}^{k-1}\otimes\id)\circ\overline{\Delta}\colon \As_{\geq 1}(m)\longrightarrow(\As_{\geq 1}(m))^{\otimes k+1}$, for $k>1$. Concretely, we have
\begin{equation}\label{eq:description_k_fold_reduced_comultiplication}
    \overline{\Delta}^k(F)=\sum_{\substack{J_0\sqcup\dots\sqcup J_{k}=\pi_0(F)\\ J_i\neq \emptyset}}\paren{F_{J_0}}\otimes\dots\otimes\paren{F_{J_{k}}},
\end{equation}
where the sum is over all possible partitions of the components of $F$ into $k+1$ nonempty subforests.

For example, we have
\begin{align*}
    \overline{\Delta}\paren{\vcenter{\hbox{\includegraphics[height=25pt]{Ressources/Delta/123.eps}}}}=&\vcenter{\hbox{\includegraphics[height=25pt]{Ressources/Delta/1.eps}}}\otimes\vcenter{\hbox{\includegraphics[height=25pt]{Ressources/Delta/23.eps}}}\;+\;\vcenter{\hbox{\includegraphics[height=25pt]{Ressources/Delta/2.eps}}}\otimes\vcenter{\hbox{\includegraphics[height=25pt]{Ressources/Delta/13.eps}}}\;+\;\vcenter{\hbox{\includegraphics[height=25pt]{Ressources/Delta/3.eps}}}\otimes\vcenter{\hbox{\includegraphics[height=25pt]{Ressources/Delta/12.eps}}}\\ &\;+\;\vcenter{\hbox{\includegraphics[height=25pt]{Ressources/Delta/12.eps}}}\otimes\vcenter{\hbox{\includegraphics[height=25pt]{Ressources/Delta/3.eps}}}\;+\;\vcenter{\hbox{\includegraphics[height=25pt]{Ressources/Delta/13.eps}}}\otimes\vcenter{\hbox{\includegraphics[height=25pt]{Ressources/Delta/2.eps}}}\;+\;\vcenter{\hbox{\includegraphics[height=25pt]{Ressources/Delta/23.eps}}}\otimes\vcenter{\hbox{\includegraphics[height=25pt]{Ressources/Delta/1.eps}}}
\end{align*}
and
\begin{align*}
    \overline{\Delta}^2\paren{\vcenter{\hbox{\includegraphics[height=25pt]{Ressources/Delta/123.eps}}}}=&\vcenter{\hbox{\includegraphics[height=25pt]{Ressources/Delta/1.eps}}}\otimes\vcenter{\hbox{\includegraphics[height=25pt]{Ressources/Delta/2.eps}}}\otimes\vcenter{\hbox{\includegraphics[height=25pt]{Ressources/Delta/3.eps}}}\;+\;\vcenter{\hbox{\includegraphics[height=25pt]{Ressources/Delta/1.eps}}}\otimes\vcenter{\hbox{\includegraphics[height=25pt]{Ressources/Delta/3.eps}}}\otimes\vcenter{\hbox{\includegraphics[height=25pt]{Ressources/Delta/2.eps}}}\;+\;\vcenter{\hbox{\includegraphics[height=25pt]{Ressources/Delta/2.eps}}}\otimes\vcenter{\hbox{\includegraphics[height=25pt]{Ressources/Delta/1.eps}}}\otimes\vcenter{\hbox{\includegraphics[height=25pt]{Ressources/Delta/3.eps}}}\\
    &+\;\vcenter{\hbox{\includegraphics[height=25pt]{Ressources/Delta/3.eps}}}\otimes\vcenter{\hbox{\includegraphics[height=25pt]{Ressources/Delta/2.eps}}}\otimes\vcenter{\hbox{\includegraphics[height=25pt]{Ressources/Delta/1.eps}}}\;+\;\vcenter{\hbox{\includegraphics[height=25pt]{Ressources/Delta/2.eps}}}\otimes\vcenter{\hbox{\includegraphics[height=25pt]{Ressources/Delta/3.eps}}}\otimes\vcenter{\hbox{\includegraphics[height=25pt]{Ressources/Delta/1.eps}}}\;+\;\vcenter{\hbox{\includegraphics[height=25pt]{Ressources/Delta/3.eps}}}\otimes\vcenter{\hbox{\includegraphics[height=25pt]{Ressources/Delta/1.eps}}}\otimes\vcenter{\hbox{\includegraphics[height=25pt]{Ressources/Delta/2.eps}}}.
\end{align*}

Note also that the reduced primitive filtration of $\As_{\geq 1}(m)_{\Q}$ can already be defined over $\Z$: for $k\geq 0$, $P^k\As_{\geq 1}(m)$ is generated by all products of $\leq k$-many primitive elements.

\begin{lemma}\label{lem:properties_of_reduced_comultiplication}
    We have the following properties:
    \begin{enumerate}[label={\upshape(\roman*)}]
    \item For all $x,y\in\As_{\geq 1}(m)$,
    \begin{equation}\label{eq:reduced_comult_of_product}
        \overline{\Delta}(x\cdot y)=\overline{\Delta}(x)\cdot\overline{\Delta}(y)+\overline{\Delta}(x)\cdot(1\otimes y+y\otimes 1)+(1\otimes x+x\otimes 1)\cdot \overline{\Delta}(y)+x\otimes y+y\otimes x;
    \end{equation}
    \item $\Prim(\As(m))=\ker(\overline{\Delta})$;
    \item $\overline{\Delta}(P^{k}\As_{\geq 1}(m))\subset\sum_{i=1}^{k-1}P^{i}\As_{\geq 1}(m)\otimes P^{k-i}\As_{\geq 1}(m)$;
    \item $P^k\As_{\geq_1}(m)\subset\ker(\overline{\Delta}^k)$;
    \item $F^k\As_{\geq_1}(m)\subset\ker(\overline{\Delta}^k)$, i.e.\ $\overline{\Delta}^k$ vanishes on forests of size $\leq k$.
    \end{enumerate}
    All of the above properties remain true after passing to coefficients in $\Q$.
\end{lemma}
\begin{proof}
    Part (i) follows from the definition of $\overline{\Delta}$ and compatibility $\Delta(x\cdot y)=\Delta(x)\cdot\Delta(y)$. Part (ii) is immediate. Part (v) follows from the description (\ref{eq:description_k_fold_reduced_comultiplication}), since for a forest $F$ of size $k$ the sum is empty. It remains to show (iii) and (iv).

    We prove (iii) by induction on $k$. To lighten notation, we write $P^i$ for $P^i\As_{\geq 1}(m)$. The base case, $k=1$, is settled by (ii). Let $1\leq k<n$ and assume that $\overline{\Delta}(P^{k})\subset\sum_{i=1}^{k-1}P^{i}\otimes P^{k-i}$. The filtration stage $P^{k+1}$ is generated by $P^k$ together with all products of exactly $k+1$ primitive elements. Consider an element $z\in P^{k+1}$ of the latter sort, i.e.\ $z=xy$ for $x\in P^k$ and $y\in P^1=\Prim(\As(m))$. Using $\overline{\Delta}(y)=0$ (by (ii)), (\ref{eq:reduced_comult_of_product}) becomes
    \begin{align*}
        \overline{\Delta}(z)=\overline{\Delta}(x)\cdot(1\otimes y+y\otimes 1)+x\otimes y+y\otimes x.
    \end{align*}
    Now, by induction hypothesis, we have $\overline{\Delta}(x)\in\sum_{i=1}^{k-1}P^{i}\otimes P^{k-i}$. Using $P^i\cdot y\subset P^{i+1}$, we finally obtain
     \begin{align*}
        \overline{\Delta}(z)\in \sum_{i=1}^{k-1}P^{i}\otimes P^{k+1-i}+\sum_{i=1}^{k-1}P^{i+1}\otimes P^{k-i}+P^{k}\otimes P^1+P^1\otimes P^k=\sum_{i=1}^{k}P^{i}\otimes P^{k+1-i},
    \end{align*}
    which proves (iv).

    The proof of (iv) also proceeds by induction on $k$, with the base case being (ii) again. Assume that $P^{k-1}\subset\ker(\overline{\Delta}^{k-1})$. Let $z\in P^{k}$. By (iv) we can write
    $\overline{\Delta}(z)=\sum_{i=1}^{k-1}\sum_j x_{ij}\otimes y_{ij}$
    where $x_{ij}\in P^i$ and $y_{ij}\in P^{k-i}$. Then
    $$\overline{\Delta}^{k}(z)\coloneqq (\overline{\Delta}^{k-1}\otimes\id)\circ\overline{\Delta}(z)=\sum_{i=1}^{k-1}\sum_j \overline{\Delta}^{k-1}(x_{ij})\otimes y_{ij}=0$$
    since all $\overline{\Delta}^{k-1}(x_{ij})=0$ by induction hypothesis. This concludes the proof of (iv).
\end{proof}

\begin{theorem}\label{thm:comparison_primitive_and_size_filtration}
    Over $\Q$, the reduced primitive filtration and the size filtration of $\As^{(\FI)}(m)_{\Q}$ coincide. In particular, $\Prim(\As^{(\FI)}(m)_{\Q})$ is the submodule of $\As^{(\FI)}(m)_{\Q}$ generated by tree diagrams.

    When $m=1$, this already holds over $\Z$.
\end{theorem}
\begin{proof}
    We omit $\FI$ in the notation since the proof for that case is the same.
    
    In degree $n=0$, there is nothing to prove. Fix a degree $n>0$. First, any forest diagram of size $1$ is primitive by definition of the co-multiplication, i.e.\ $F^1\As_n(m)_{\Q}\subset P^1\As_n(m)_{\Q}$.

    In fact, $F^1\As_n(m)_{\Q}\subset P^1\As_n(m)_{\Q}$ implies that $F^k\As_n(m)_{\Q}\subset P^k\As_n(m)_{\Q}$ for each $k\geq 1$. We show this by induction\footnote{This argument is similar to the proof of Proposition 6.10 in \cite{habiro2000claspers}, which deals with forest claspers instead of forest diagrams.}. The base case is the above paragraph. Assume that $F^k\subset P^k$ and consider a a forest diagram $F$ of size $k+1$. Pick any tree $T$ in $F$ and apply a sequence $F=F_0\rightsquigarrow\dots\rightsquigarrow F_r$ of $\STU$-relations to slide the legs of $T$ above the rest of the trees. Thus, $F_r=T\cdot (F\setminus T)\subset P^1\cdot F^k\subset P^1\cdot P^k\subset P^{k+1}$. For each $i$, $F_{i}-F_{i+1}$ is a forest of size $k$ by $\STU$, i.e.\ $F_{i}-F_{i+1}\in F^k\subset P^k$. Therefore,
    $$F=\sum_{i=0}^{r-1}(F_i-F_{i+1})+F_r\subset P^{k+1}.$$
    
    Consequently, the size filtration injects into the primitive filtration:
\[\begin{tikzcd}
	{F^1\As_n(m)_{\Q}} & {F^2\As_n(m)_{\Q}} & \dots & {F^{n-1}\As_n(m)_{\Q}} & {F^n\As_n(m)_{\Q}} \\
	{P^1\As_n(m)_{\Q}} & {P^2\As_n(m)_{\Q}} & \dots & {P^{n-1}\As_n(m)_{\Q}} & {P^n\As_n(m)_{\Q}}
	\arrow[hook, from=2-1, to=2-2]
	\arrow[hook, from=2-2, to=2-3]
	\arrow[hook, from=1-1, to=1-2]
	\arrow[hook, from=1-2, to=1-3]
	\arrow[hook, from=1-1, to=2-1]
	\arrow[hook, from=1-2, to=2-2]
	\arrow[equal, from=1-5, to=2-5]
	\arrow[hook, from=1-3, to=1-4]
	\arrow[hook, from=1-4, to=1-5]
	\arrow[hook, from=2-3, to=2-4]
	\arrow[hook, from=2-4, to=2-5]
	\arrow[hook, from=1-4, to=2-4]
\end{tikzcd}.\]
We show that the vertical arrows are all equalities, inductively from right to left. Let $1\leq k< n$ and assume that $F^{k+1}=P^{k+1}$. We seek for a section to the inclusion $i^{k}_n$ as in the diagram below
\begin{equation}\label{eq:com_diag_proof_comparison_filtrations_size_primitive}\begin{tikzcd}
	{F^{k}\As_n(m)_{\Q}} & {F^{k+1}\As_n(m)_{\Q}} \\
	{P^{k}\As_n(m)_{\Q}} & {P^{k+1}\As_n(m)_{\Q}}
	\arrow[hook, from=2-1, to=2-2]
	\arrow["{i^{k}_n}"', hook, from=1-1, to=1-2]
	\arrow[hook, from=1-1, to=2-1]
	\arrow[equal, from=1-2, to=2-2]
	\arrow["{s^{k+1}_n}"',two heads, bend right=10, from=1-2, to=1-1]
\end{tikzcd}.\end{equation}
Consider the linear map
\begin{equation}\label{eq:section_primitive=size_filtration}
\begin{split}
    s^{k+1}_n\colon F^{k+1}\As_n(m)_{\Q}&\longrightarrow\As_n(m)_{\Q}\\
    F&\longmapsto F-\tfrac{1}{(k+1)!}\mu^{k}\circ\overline{\Delta}^{k}(F),
\end{split}
\end{equation}
where $\mu^{k}\colon\As(m)^{\otimes k+1}\rightarrow\As(m)$ is the $m$-fold multiplication. We have:
\begin{itemize}
    \item \textbf{The image of $s^{k+1}_n$ lies in $F^k\As_n(m)_{\Q}$.} Indeed, for a forest diagram $F$ of size $k+1$, say $F=\bigcup_{i=1}^{k+1}T_i$, we have\footnote{There is no natural ordering of the trees of $F$, hence a choice of identification $\pi_0(F)\cong\{1,\dots,s\}$ is required. However, the sum in (\ref{eq:description_s_section}) is over all permutations, hence this choice is irrelevant.}
    \begin{equation}\label{eq:description_s_section}
        s^{k+1}_n(F)=\frac{1}{(k+1)!}\sum_{\sigma\in\Sy_{k+1}}(F-(F)_{\sigma})
    \end{equation}
    where $(F)_{\sigma}\coloneqq T_{\sigma(1)}\dots T_{\sigma(s)}$ is obtained by stacking the trees of $F$ in the order prescribed by $\sigma$. Each difference $(F-(F)_{\sigma})$ can be obtained by a sequence of leg exchanges (a path in $\Fs(F)$, see Definition \ref{def:graph_of_forests}) from $F_{\sigma}$ to $F$, hence we obtain $F-F_{\sigma}$ as a sum of forests of size $k$ (each leg exchange yields a term where the two legs are attached together, by $\STU$). Thus $s^{k+1}_n(F)\in F^k\As_n(m)_{\Q}$.

    For example,
    \end{itemize}
    \begin{align*}
        s^{3}_3\paren{\vcenter{\hbox{\includegraphics[height=30pt]{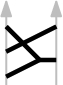}}}}=\frac{1}{2}\paren{\vcenter{\hbox{\includegraphics[height=30pt]{Ressources/section_primitive_filtration/section_primitive_filtration-term-F}}}-\vcenter{\hbox{\includegraphics[height=30pt]{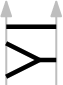}}}}+\frac{1}{2}\paren{\vcenter{\hbox{\includegraphics[height=30pt]{Ressources/section_primitive_filtration/section_primitive_filtration-term-F}}}-\vcenter{\hbox{\includegraphics[height=30pt]{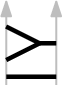}}}}
        =\frac{1}{2}\paren{\vcenter{\hbox{\includegraphics[height=30pt]{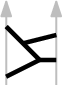}}}}+\frac{1}{2}\paren{-\vcenter{\hbox{\includegraphics[height=30pt]{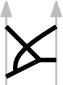}}}-\vcenter{\hbox{\includegraphics[height=30pt]{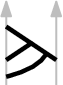}}}}.
    \end{align*}
    \begin{itemize}
    \item \textbf{The map $s^{k+1}_n$ is a section to the inclusion $i^k_n$.} Indeed, for $F\in F^k\As_n(m)_{\Q}$, we have $\overline{\Delta}^k(F)=0$ by Lemma \ref{lem:properties_of_reduced_comultiplication}(vi). Thus $s^{k+1}_n$ restricts to the identity on $F^k\As_n(m)_{\Q}$.
\end{itemize}
By Lemma \ref{lem:properties_of_reduced_comultiplication}(v), $s^{k+1}_n$ also restricts to the identity on $P^k\As_n(m)_{\Q}$. Thus, starting with $x\in P^k\As_n(m)_{\Q}$ and going around the diagram (\ref{eq:com_diag_proof_comparison_filtrations_size_primitive}) in the counterclockwise direction, we end up with $s(x)=x$, which shows that the vertical inclusion $F^k\monic P^k$ is also surjective. This concludes the proof that the two filtrations coincide.

It follows that
$$\Prim(\As(m)_{\Q})=P^1\As_{\geq 1}(m)_{\Q}=F^1\As_{\geq 1}(m)_{\Q},$$
i.e.\ the primitive part of $\As(m)_{\Q}$ is additively generated by forests of size $1$, that is, trees.

When $m=1$, the bialgebra $\As(m)$ is commutative. Thus, if $F$ is a forest of size $k+1$, then all $(F)_{\sigma}$, $\sigma\in\Sy_{k+1}$, are equal and averaging is not needed. In particular, the section $s^{k+1}_n$ defined in (\ref{eq:section_primitive=size_filtration}) is already defined over $\Z$. This yields an alternative proof of the fact that tree diagrams generate the primitive Lie algebra of $\As(1)$ over $\Z$. See \cite{lando2primitive} for a different argument.
\end{proof}

In order to better understand that filtration, and more particularly the primitive Lie algebra $\Prim(\As^{(\FI)}(m)_{\Q})$, a concrete description of the filtration steps, i.e.\ by generators and relations, is required. This is the topic of the next two sections.

\subsection{Permutographs and graphs of forests}\label{sec:permutographs_and_graphs_of_forests}
\subsubsection{Permutographs.}
A \emph{$n$-permutohedron} is a polytope defined as the convex hull of a set of the form
\begin{equation}\label{eq:permutohedron}
    \{(a_{\sigma(1)},\dots,a_{\sigma(n)})\mid\sigma\in\Sy_n\}\subset\R^n
\end{equation}
where $\textbf{a}=(a_1,\dots,a_n)\in\R$. When $\textbf{a}=(1,2,\dots,n)$, we call (\ref{eq:permutohedron}) the \emph{standard $n$-permutohedron}. They are rich combinatorial objects that appear naturally in many places \cite{ziegler2012lectures_on_polytopes,postnikov2009permutohedra,lambrechtsTurchinVolic2010permutohedron}. For example, their faces classify the labelled partitions of finite sets.

To study graphs of forests, we will actually focus on $2$-skeleta of permutohedra. For a tuple of non-negative integers $\textbf{n}\coloneqq (n_1,\dots,n_s)\in\N^{s}$, consider the set $V_{\textbf{n}}$ of all possible words made of $n_i$-many $i$, for each $i=1,\dots,s$. In particular, we write $\underline{w}_{\textbf{n}}\coloneqq \underbrace{1\dots 1}_{n_1}\dots\underbrace{s\dots s}_{n_s}$ and $\overline{w}_{\textbf{n}}\coloneqq \underbrace{s\dots s}_{n_s}\dots\underbrace{1\dots 1}_{n_1}$. The symmetric group $\Sy_n$, where $n\coloneqq\sum_i n_i$, acts on $V_{\textbf{n}}$ by permutation, denoted $\sigma\cdot w$ for $\sigma\in\Sy_n,w\in V_{\textbf{n}}$. The action is transitive but has many fixed points, unless all $n_i=1$.

\begin{definition}\label{def:permutohedron_graph}
    The \emph{\textbf{n}-permutohedron graph}, or \emph{\textbf{n}-permutograph}, is the undirected graph $P_{\textbf{n}}\coloneqq (V_{\textbf{n}},E_{\textbf{n}})$, whose set of vertices is $V_{\textbf{n}}$ and two distinct vertices $v,w\in V_{\textbf{n}}$ are connected by an undirected edge if there is an basic transposition $\tau_i=(i,i+1)\in\Sy_n$ such that $w=\tau\cdot v$ (hence also $v=\tau\cdot w$), i.e.\ $v=\dots jk\dots$ and $w=\dots kj\dots$ where $1\leq j\neq k\leq s$ are the $i$-th and $(i+1)$-st entries of $v,w$.    
\end{definition}
There is at most one edge joining any pair of vertices and at most $n-1$ edges adjacent to a given vertex, since there are only $n-1$ adjacent transpositions $\tau_1,\dots,\tau_{n-1}$. Note that permutographs are connected, since the $\Sy_{\textbf{n}}$-action on $V_{\textbf{n}}$ is transitive.

\begin{figure}[!ht]
     \centering
        \begin{subfigure}[b]{0.45\textwidth}
              \centering
        \includegraphics[scale=0.7]{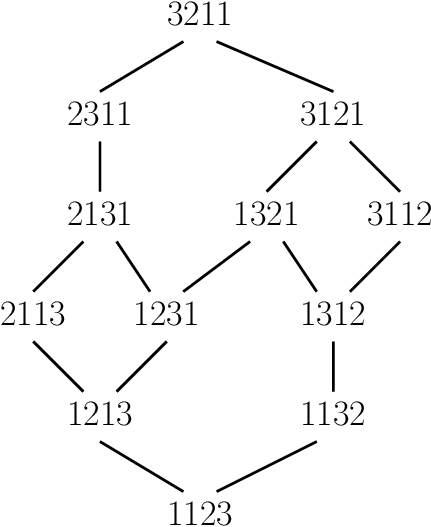}
        \caption{The permutograph $P_{(2,1,1)}$}
        \label{fig:permutograph_1123}
    \end{subfigure}
     \hfill
     \begin{subfigure}[b]{0.45\textwidth}
              \centering
        \includegraphics[scale=0.7]{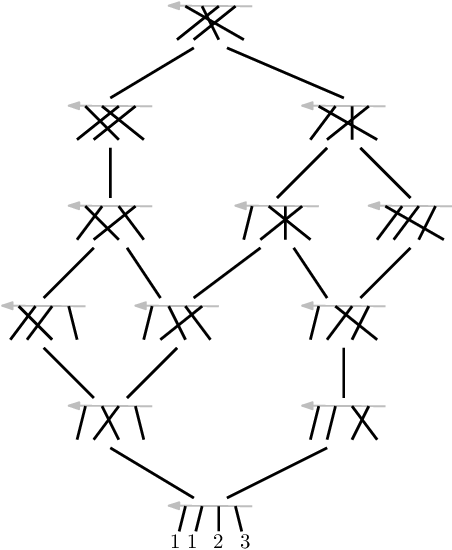}
        \caption{Graph of labelled legs on one strands}
        \label{fig:graph_labelled_legs_on_one_strand}
    \end{subfigure}
        \caption{One slice of a graph of labelled forests}
        \label{fig:permutograph_and_graph_of_labelled_forests}
\end{figure}

\begin{proposition}\label{prop:1-homology_of_P_n_generated_by_squares_and_hexagons}
    The first homology group of the permutograph $P_{\textbf{n}}$ is generated by length $4$ and $6$ cycles. In other words, any $1$-cycle in $P_{\textbf{n}}$ is the boundary of a union of squares and hexagons.
\end{proposition}
\begin{proof}
    To prove this, we scan $P_{\textbf{n}}$ using the \emph{height function}
    \begin{equation}
            h\colon V_{\textbf{n}}\longrightarrow\N\\
            \colon v\longmapsto\sum_{i=1}^{n}\#\{j>i\mid v_j<v_i\}
    \end{equation}
    which assigns to a word $v$ the number $h(v)$ of pairs of letters $\dots v_i\dots v_j\dots$ where $v_j< v_i$ appearing in it, i.e.\ the number of inversions in $v$. For example, $h(\underline{w}_{\textbf{n}})=0$ is minimal and $h(\overline{w}_{\textbf{n}})$ is maximal in $P_{\textbf{n}}$. The function $h$ is a Morse function in the sense that, for any edge $\{v,w\}$, we have $h(w)=h(v)\pm 1$. Pictorially, when drawing $P_{\textbf{n}}$ with all vertices of height $i$ at height $i$, no edge is horizontal (see Figure \ref{fig:permutograph_1123}).

    Now we show that any $1$-cycle $C$ in $P_{\textbf{n}}$ is a sum of length $4$ and $6$ cycles, by double induction on $\max_{v\in C} h(v)$ and on the number of vertices reaching this maximum. Up to decomposing $C$ into smaller subcycles and removing paths going back and forth, we can assume that $C=\overrightarrow{v^0v^1}+\dots+\overrightarrow{v^{l-1}v^l}$ where $(v^0,\dots,v^{l}=v_0)$ is a sequence without repetitions (except ($v^l=v^0$). If $\max_{v\in C} h(v)=0$ then there is nothing to prove since $C=0$.
    
    Assume the statement to be true for all cycles of maximal height $<\max_{v\in C} h(v)$ and with less maxima than $C$. Pick any vertex $v^M\in C$ that reaches the maximum, then we must have $h(v^{M\pm 1})=h(v^M)-1$. By definition of the permutograph, $v^{M-1}=\tau_i\cdot v^M$ and $v^{M+1}=\tau_{j}\cdot v^M$, where $i\neq j$. We assume $i<j$ to simplify notation. There are two cases to consider:
    \begin{enumerate}
        \item If $\abs{j-i}>1$, then $\tau_i$ and $\tau_j$ have disjoint support and the vertex $\tilde{v}^M\coloneqq \tau_i\tau_j\cdot v^M=\tau_i\cdot v^{M+1}=\tau_j\cdot v^{M-1}$ forms a square together with $v^M,v^{M+1},v^{M-1}$ (we use the relation $\tau_i\tau_j=\tau_j\tau_i$ in $\Sy_n$). In particular, we have
        \begin{equation}\label{eq:extracting_square_from_cycle}
        \begin{split}
            C&=\overrightarrow{v^0v^1}+\dots+\overrightarrow{v^{M-1}v^M}+\overrightarrow{v^Mv^{M+1}}+\dots+\overrightarrow{v^{l-1}v^l}\\
            &=\underbrace{(\overrightarrow{v^0v^1}+\dots+\overrightarrow{v^{M-1}\tilde{v}^M}+\overrightarrow{\tilde{v}^Mv^{M+1}}+\dots+\overrightarrow{v^{l-1}v^l})}_{\eqqcolon C'}+c=C'+c
        \end{split}
        \end{equation}
        where $c\coloneqq \overrightarrow{v^{M-1}v^M}+\overrightarrow{v^Mv^{M+1}}-\overrightarrow{v^{M-1}\tilde{v}^M}-\overrightarrow{\tilde{v}^Mv^{M+1}}$ is a length $4$ cycle, and $C'$ has one less maximum. We are done by induction.
        \item If $\abs{j-i}=1$, say $j=i+1$, then the above does not work since the vertices $\tau_i \tau_{i+1}\cdot v^M$ and $\tau_{i+1}\tau_i\cdot v^M$ are distinct. In this case we need to dive lower in the permutograph and use instead the relation $\tau_i\tau_{i+1}\tau_i=\tau_{i+1}\tau_i\tau_{i+1}$ in $\Sy_n$. Note that $v_i^M>v_{i+1}^M>v_{i+2}^M$ since both $\tau_i$ and $\tau_{i+1}$ need to reduce the number of inversions. We find an hexagon in $P_{\textbf{n}}$ as the one at the bottom of Figure \ref{fig:permutograph_1123}, where $(v^M_i,v^M_{i+1},v^M_{i+2})=(3,2,1)$. In particular, $C=C'+c$ with $C'\coloneqq C-c$ where $c$ is the hexagon cycle. Again, $C'$ has less maxima and we are done by induction.
    \end{enumerate}
\end{proof}

\begin{remark}\label{rem:symmetric_group_relations_cycles}
    Recall that the symmetric group $\Sy_n$ has the following presentation
    \begin{equation}\label{eq:presentation_Sn}
            \Sy_n=\left\langle\left.\tau_1\dots\tau_{n-1}\right|\begin{matrix}
        \tau_i\tau_i=1 & \text{for each }i\\
        \tau_i\tau_j=\tau_j\tau_i & \text{if }\abs{i-j}>1\\
        \tau_{i}\tau_{i+1}\tau_{i}=\tau_{i+1}\tau_{i}\tau_{i+1} & \text{for each }i
    \end{matrix}\right\rangle
    \end{equation}
    where $\tau_i=(i,i+1)$, $i=1,\dots,n-1$, are the standard transpositions. The second and third relations are called \emph{far-commutativity} and \emph{braiding} relations.
    
    A cycle in the permutograph $P_{\textbf{n}}$ can be written as a presentation of the identity permutation $\id=\tau_{k_{l-1}}\dots\tau_{k_0}$ as a product of transpositions $\tau_i$. In the proof, the assumption that the cycle has no repetitions implies that $k_i\neq k_{i+1}$ for each $i$, which we can assume since each $\tau_i$ has order two and cancelling paths going back and forth is already taken into account in the first homology group.

    The rest of the proof follows from the realization that, up to removing $\tau_i^2$ when necessary, we can make the word $\tau_{k_{l-1}}\dots\tau_{k_0}$ trivial by exclusively using far-commutativity relations and braiding relations. The former are responsible for length $4$ cycles, while the latter are responsible for length $6$ cycles.

    Since the action of $\Sy_n$ on the vertices of $P_{\textbf{n}}$ is not free, more details are needed to turn this remark into a rigorous proof.
\end{remark}

\begin{definition}
    For two graphs $G_1,G_2$, their \emph{cartesian} or \emph{box product}, denoted $G_1\Box G_2$, is the $1$-skeleton of $G_1\times G_2$. The box product of graphs is associative.
\end{definition}

\begin{corollary}\label{cor:1-homology_of_box_P_n_generated_by_squares_and_hexagons}
    The first homology of a box product of permutographs $\Box_{j=1}^m P_{\textbf{n}^j}$ is generated by length $4$ and $6$ cycles.
\end{corollary}
\begin{proof}
    The exact same proof as Proposition \ref{prop:1-homology_of_P_n_generated_by_squares_and_hexagons} works, where now the two transpositions in case 1 can happen in different factors. Consequently, there are two types of length $4$ cycles: product of edges from different factors and squares contained in one factor.
\end{proof}
\begin{remark}\label{rem:product_symmetric_group_relations_cycles}
    Following Remark \ref{rem:symmetric_group_relations_cycles}, the above Proposition follows from the fact that the product $\Sy_{n^1}\times\dots\times\Sy_{n^m}$ is generated by standard transpositions $\tau^{(j)}_1,\dots,\tau^{(j)}_{n^j}$, for $j=1,\dots,m$, with the same relations as (\ref{eq:presentation_Sn}) for each fixed $j$, with additional commutativity relations $\tau^{(j)}_{i}\tau^{(j')}_{i'}=\tau^{(j')}_{i'}\tau^{(j)}_{i}$ whenever $j\neq j'$. These additional relations are responsible for the new type of length $4$ cycles.
\end{remark}

\subsubsection{Graph of labelled forests.}\label{sec:Graph of labelled forests}
A \emph{labelled forest diagram} is simply a forest diagram in $\Ds^F(m)$ together with a labeling of its trees by $1,\dots,s$, where $s$ is the size of the forest, i.e.\ a choice of identification $\pi_0(F)\cong\{1,\dots,m\}$. We say that two labelled forests $F,F'$ are related by a \emph{slide move} if $F'$ is obtained from $F$ by permuting adjacent legs belonging to distinct trees:
\begin{equation}
    F=\vcenter{\hbox{\includegraphics[height=25pt]{Ressources/permutograph/STU2-satisfied-by-s/term-x.eps}}}\;\xleftrightarrow{\text{slide move}}F'=\vcenter{\hbox{\includegraphics[height=25pt]{Ressources/permutograph/STU2-satisfied-by-s/term-=.eps}}}\;,
\end{equation}
where it is understood, as usual, that $F$ and $F'$ are identical outside of the part shown.

\begin{definition}
    Given $s$-many tree diagrams $T_1,\dots,T_s\in\Ds^T(m)$ on $m$ strands ($s>0$), the \emph{graph of labelled forests on $T_1,\dots,T_s$} is the undirected graph whose set of vertices is the set
$$\widetilde{\Fs}(T_1,\dots,T_s)$$
of all labelled forests on $m$ strands, whose $i$-th tree is $T_i$, and where two labelled forests are connected by an edge if they are related by a slide move.
\end{definition}

A part of a graph of labelled forests, corresponding to what happens on one strand, is shown on Figure \ref{fig:graph_labelled_legs_on_one_strand}. On this figure, the two leftmost legs (e.g.\ in the bottom diagram) belong to $T_1$, while the next ones belong to $T_2$ and $T_3$, respectively\footnote{For simplicity, we have labelled only the forest at the bottom, but all forests on the figure should be labelled accordingly.}. The legs from $T_1$ are not allowed to slide across one another, but they can slide across the two other legs. Note that any two vertices in $\tilde{\Fs}(T_1,\dots,T_s)$ are connected by \emph{at most} one edge. There are no edges when $s=1$.

\begin{proposition}\label{prop:graph_of_labelled_forests=product_of_permutographs}
    Suppose that the tree $T_i$ has $n^j_{i}$ legs on the $j$-th strand, for $1\leq i\leq s$ and $1\leq j\leq m$. Then there is a graph isomorphism defined on vertices as\footnote{Since there is at most one edge connecting two vertices, this definition determines the morphism on edges as well.}
    \begin{equation}\label{eq:graph_of_labelled_forests=product_of_permutographs}
        W\colon\widetilde{\Fs}(T_1,\dots,T_s)\xrightarrow{\cong}\Box_{j=1}^m P_{\textbf{n}^{j}}
    \end{equation}
    where $\textbf{n}^j=(n^j_1,\dots,n^j_s)$, sending a labelled forest to $(w^j)_j$ where $w^j$ encodes the ordering of the legs on the $j$-strand, from top to bottom.
\end{proposition}
\begin{proof}
    This is immediate, since a labelled forest is entirely determined by its trees $T_1,\dots,T_s$ and the relative positions of their legs on the strands.
\end{proof}

\begin{corollary}\label{cor:1-homology_of_labelled_forests_graph_generated_by_squares_and_hexagons}
    The first homology of the graph of labelled forests $\widetilde{\Fs}(T_1,\dots,T_s)$ is generated by length $4$ and $6$ cycles. In other words, any $1$-cycle in it is the boundary of a union of squares and hexagons.
\end{corollary}
\begin{proof}
    This follows immediately from Proposition \ref{prop:graph_of_labelled_forests=product_of_permutographs} and Corollary \ref{cor:1-homology_of_box_P_n_generated_by_squares_and_hexagons}.
\end{proof}

\begin{definition}
To a directed edge $\overrightarrow{FF'}$ in $\widetilde{\Fs}(T_1,\dots,T_s)$, $s\geq 2$, one can assign the third term appearing in the $\STU$ relation, which we denote $\overrightarrow{FF'}\in\Ds^{s-1}(m)$:
\begin{equation}
    \text{If }F=\vcenter{\hbox{\includegraphics[height=20pt]{Ressources/permutograph/STU2-satisfied-by-s/term-x.eps}}}\;\text{ and }F'=\vcenter{\hbox{\includegraphics[height=20pt]{Ressources/permutograph/STU2-satisfied-by-s/term-=.eps}}}\;\text{, then }\overrightarrow{FF'}=\overrightarrow{\vcenter{\hbox{\includegraphics[height=20pt]{Ressources/permutograph/STU2-satisfied-by-s/term-x.eps}}}\vcenter{\hbox{\includegraphics[height=20pt]{Ressources/permutograph/STU2-satisfied-by-s/term-=.eps}}}}\coloneqq\vcenter{\hbox{\includegraphics[height=20pt]{Ressources/permutograph/STU2-satisfied-by-s/term-Y.eps}}}.
\end{equation}
In particular, $\overrightarrow{FF'}=F'-F$ in $\As(m)$.
\end{definition}

The fact that the legs involved in a slide move belong to distinct trees is important here, otherwise $\overrightarrow{FF'}$ would have a cycle and would not be a forest diagram.

There are many paths joining forests $F,F'\in\widetilde{\Fs}(T_1,\dots,T_s)$. Hence, $\overrightarrow{FF'}$ is not well-defined for non-adjacent forests. However, any two paths differ by a cycle in $\widetilde{\Fs}(T_1,\dots,T_s)$, which by Corollary \ref{cor:1-homology_of_labelled_forests_graph_generated_by_squares_and_hexagons} can be written as a sum of back-and-forth paths, length $4$ and $6$ cycles. Thus, modding out the target $\Z\Ds^{s-1}(m)$ by those cycles will give us a well-defined map $\overrightarrow{\cdot}$.

\begin{definition}\label{def:square_and_hexagon_relations}
    A \emph{$\square$ relation} in $\Z\Ds^{s-1}(m)$ is a relation of the form
    \begin{equation}\label{eq:square_relation}
        \overrightarrow{F_1F_2}+\overrightarrow{F_2F_3}+\overrightarrow{F_3F_4}+\overrightarrow{F_4F_1}
    \end{equation}
    where $(F_1,F_2,F_3,F_4)$ is a length $4$ cycle in $\widetilde{\Fs}(T_1,\dots,T_s)$ for some trees $T_i$. See Figure \ref{fig:square_relation}.

    Similarly, a $\hexagon$ \emph{relation} in $\Z\Ds^{s-1}(m)$ is a relation of the form
    \begin{equation}\label{eq:hexagon_relation}
        \overrightarrow{F_1F_2}+\overrightarrow{F_2F_3}+\overrightarrow{F_3F_4}+\overrightarrow{F_4F_5}+\overrightarrow{F_5F_6}+\overrightarrow{F_6F_1}
    \end{equation}
    where $(F_1,\dots,F_6)$ is a length $6$ cycle in $\widetilde{\Fs}(T_1,\dots,T_s)$ for some trees $T_i$. See Figure \ref{fig:hexagon_relation}.

    Let $\square R$, resp. $\hexagon R$, denote the submodule of $\Z\Ds^{F}(m)$ generated by all $\square$, resp. $\hexagon$, relations.
\end{definition}

\begin{remark}\label{rem:HR_is_image_of_IHX}
    A $\hexagon$ relation can be obtained in the following way: start with an $\IHX$ relation that involves a leg, and break apart that leg in the three terms. The resulting signed sum of six terms is an $\hexagon$ relation, and all $\hexagon$ relations can be obtained in this way.
\end{remark}

\begin{figure}[!ht]
     \centering
     \begin{subfigure}[b]{0.4\textwidth}
         \centering
         \includegraphics[scale=1.2]{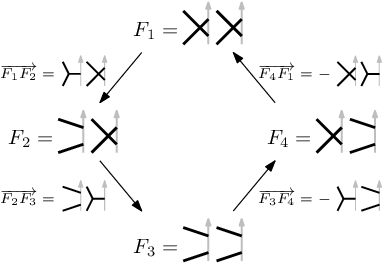}
         \caption{A $\square$ relation}
         \label{fig:square_relation}
     \end{subfigure}
     \hfill
     \begin{subfigure}[b]{0.4\textwidth}
         \centering
         \includegraphics[scale=1]{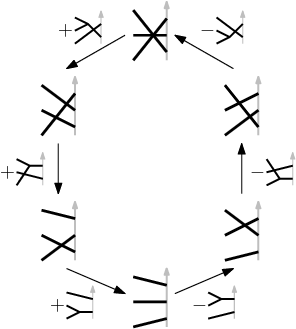}
         \caption{A $\hexagon$ relation}
         \label{fig:hexagon_relation}
     \end{subfigure}
        \caption{Relations obtained from cycles of length $4$ and $6$}
        \label{fig:square_and_hexagon_relations}
\end{figure}

By the above definition and Corollary \ref{cor:1-homology_of_labelled_forests_graph_generated_by_squares_and_hexagons}, the following definition is well-defined.

\begin{definition}\label{def:vector_from_F_to_F'_label}
    For two labelled forests $F,F'\in\widetilde{\Fs}(T_1,\dots,T_s)$, $s\geq 2$, define
    \begin{equation}\label{eq:vector_from_F_to_F'_label}
        \overrightarrow{FF'}\coloneqq\overrightarrow{FF_1}+\dots+\overrightarrow{F_{l-1}F'}\in\faktor{\Z\Ds^{s-1}(m)}{\langle\AS,\square R,\hexagon R\rangle}
    \end{equation}
    where $F\rightsquigarrow F_1\rightsquigarrow\dots\rightsquigarrow F'$ is a path from $F$ to $F'$ in $\widetilde{\Fs}(T_1,\dots,T_s)$.
\end{definition}

Any two choices of paths differ by a cycle, which can be decomposed into back-and-forth paths, squares and hexagons, by Corollary \ref{cor:1-homology_of_labelled_forests_graph_generated_by_squares_and_hexagons}. The relations $\AS$, $\square R$ and $\hexagon R$ respectively take care of these. This shows that the above definition is well-defined.

\begin{proposition}\label{prop:STU2_contains_square_relation}
    Square relations and $\STU^2$ relations coincide, i.e.\ $\square R=\STU^2$ as submodules of $\Z\Ds^{s}(m)$.
\end{proposition}
\begin{proof}
    Recall \cite[Definition 3.1]{conant2008stu2} that $\STU^2$ relations among size $s$ forests are defined starting with a \emph{template}, i.e.\ a size $s-1$ forest diagram or a size $s$ Feynman diagram with a single cycle, then choosing two legs (which must be adjacent to the cycle in the latter case) and breaking one or the other apart using $\STU$. Breaking both apart yields four possible size $s+1$ forest diagrams, which form the corners of a square in a graph of labelled forests and the corresponding $\square$ relation coincides with the $\STU^2$ relation.

    Conversely, consider a square as on Figure \ref{fig:square_relation} and the associated $\square$ relation. Then the template $\vcenter{\hbox{\includegraphics[height=15pt]{Ressources/permutograph/STU2-satisfied-by-s/term-Y.eps}}}\vcenter{\hbox{\includegraphics[height=15pt]{Ressources/permutograph/STU2-satisfied-by-s/term-Y.eps}}}$ yields a $\STU^2$ relation which coincides with that $\square$ relation.
\end{proof}

\subsubsection{Graph of forests.}\label{sec:Graph of forests}
In the previous section we investigated the graph of labelled forests and the relations that are required to define differences between labelled forests that are joined by a path. However, forests diagrams in the bialgebra $\As(m)$ are not labelled. The \emph{true} graph of forests is a quotient of the graph of labelled forests, where the quotient map forgets the labelling.

\begin{definition}\label{def:graph_of_forests}
    The \emph{graph of forests} on trees $T_1,\dots,T_s$ is the graph $\Fs(T_1,\dots,T_s)$ whose vertices are all forest diagrams in $\Ds^s(m)$ consisting of those trees, and to each slide move relating two forests there is an associated edge joining the corresponding vertices.

    Note that $\Fs(T_1,\dots,T_s)$ does not depend on the ordering of the trees $T_i$. If $F=\bigcup_{i=1}^s T_i$ is a forest diagram consisting of trees $T_1,\dots,T_s$, we write $\Fs(F)\coloneqq\Fs(T_1,\dots,T_s)$.
\end{definition}

There is a natural map
\begin{equation}\label{eq:forget_labeling}
    \widetilde{\Fs}(T_1,\dots,T_s)\epic\Fs(T_1,\dots,T_s)
\end{equation}
that forgets the labels of forests. It is surjective but not injective if there are identical trees among $T_1,\dots,T_s$.

\begin{remark}
    In $\Fs(T_1,\dots,T_s)$, it can happen that two vertices are joined by multiple edges. See Figure \ref{fig:length_2_cycles} for an example: the top, resp. bottom, edge corresponds to exchanging the legs on the right, resp. left, strand.
    \begin{figure}[!ht]
        \centering
        \includegraphics[scale=1.4]{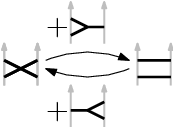}
        \caption{Two different leg moves joining two forests}
        \label{fig:length_2_cycles}
    \end{figure}

    Thus, the notation $\overrightarrow{FF'}$ is ambiguous. Given a directed edge $F\overset{e}{\rightsquigarrow}F'$ we write $\overrightarrow{e}$ to denote the third term of the corresponding $\STU$ relation, i.e.\ $\overrightarrow{e}=F'-F$ in $\As(m)$.
\end{remark}

\begin{lemma}\label{lem:path_lifting_in_graph_of_forests}
    For any labelled forest diagram $F$, the forgetful map $\widetilde{\Fs}(F)\epic\Fs(F)$ satisfies the following \emph{path-lifting property}: For any path $P=F_0\overset{e_0}{\rightsquigarrow}F_1\overset{e_1}{\rightsquigarrow}\dots\overset{e_{l-1}}{\rightsquigarrow}F_{l}$ in $\Fs(F)$ and any choice of labelled forest $\tilde{F_0}$ lifting $F_0$, there is a unique path $\tilde{P}=\tilde{F}_0\overset{\tilde{e}_0}{\rightsquigarrow}\tilde{F}_1\overset{\tilde{e}_1}{\rightsquigarrow}\dots\overset{\tilde{e}_{l-1}}{\rightsquigarrow}\tilde{F}_{l}$ in $\widetilde{\Fs}(F)$ that lifts $P$. In other words, the surjection $\tilde{\Fs}(F)\epic\Fs(F)$ is a finite covering.

    Moreover, $\overrightarrow{\tilde{e}_i}=\overrightarrow{e_i}$ for each $i=0,\dots,l-1$.
\end{lemma}
\begin{proof}
    Starting with $\tilde{F}_0$ and applying the sequence of slide moves prescribed by $P$ yields such a lift $\tilde{P}$, which is uniquely determined by this process.
\end{proof}

\begin{remark}
    The structure of the first homology $H_1(\Fs(T_1,\dots,T_s))$ over $\Z$ is not as simple as in the labelled case. The resulting new relations will be investigated in future work. The good news is that those problems disappear once we pass to rational coefficients.
\end{remark}

\begin{proposition}\label{prop:any_cycle_in_graph_forests_has_a_multiple_that_lifts}
    Any cycle in $\Fs(T_1,\dots,T_s)$ has a multiple that lifts along the covering 
    $$\widetilde{\Fs}(T_1,\dots,T_s)\epic\Fs(T_1,\dots,T_s).$$
    In particular, over $\Q$, for any cycle $C$, the element $\overrightarrow{C}$ belongs to the submodule $\langle\AS,\square R,\hexagon R\rangle\subset\Q\Ds^{s-1}(m)$.
\end{proposition}
\begin{proof}
    Let $G$ be the group of deck transformations of the finite covering $p\colon\tilde{\Fs}(T_1,\dots,T_s)\epic\Fs(T_1,\dots,T_s)$, and $\abs{G}$ its size. Consider the transfer homomorphism $\tau\colon H_1(\Fs(T_1,\dots,T_s))\rightarrow H_1(\tilde{\Fs}(T_1,\dots,T_s))$. Then $p_*\circ\tau$ equals multiplication by $\abs{G}$. In particular, the cokernel of $p_*$ is $\abs{G}$-torsion and $p_*$ becomes surjective when passing to coefficients $\Q$. This concludes.
    
    The second statement then follows from the first and Corollary \ref{cor:1-homology_of_labelled_forests_graph_generated_by_squares_and_hexagons}.
\end{proof}

With the above result in hands, we can adapt Definition \ref{def:vector_from_F_to_F'_label} to the graph of forests, but with rational coefficients.

\begin{definition}\label{def:vector_from_F_to_F'}
    For two forests $F,F'\in\Fs(T_1,\dots,T_s)$, define
    \begin{equation}\label{eq:vector_from_F_to_F'}
        \overrightarrow{FF'}\coloneqq\overrightarrow{P}\in\faktor{\Q\Ds^{s-1}(m)}{\langle\AS,\square R,\hexagon R\rangle}
    \end{equation}
    where $P$ is a path from $F$ to $F'$ in $\Fs(T_1,\dots,T_s)$.
\end{definition}

\subsection{Presentation of the primitive filtration of \texorpdfstring{$\As(m)_{\Q}$}{A(m)Q}}\label{sec:presentation_of_primitive_filtration}
\subsubsection{Splitting trees.}
\begin{definition}\label{}
    For $1\leq s\leq n-1$, the \emph{$\Q$-module of size $s$ forest diagrams} $\Fs^s(m)_{\Q}$ is defined by
\begin{equation}
    \Fs^{(\FI),s}(m)_{\Q}\coloneqq\faktor{\Q\Ds^s(m)}{\langle(\1T),\AS,\IHX,\STU^2,\hexagon R\rangle}.
\end{equation}
It is graded by the degree of diagrams, with degree $n$ graded part denoted by $\Fs^{(\FI),s}_n(m)_{\Q}$.
\end{definition}

\begin{remark}\label{rem:simpler_relations_in_many_cases}
    When $s=1$, the submodule $\hexagon R$ is actually trivial, since an $\hexagon$ relation requires at least two trees. Thus, $\Fs^1(m)_{\Q}=\Ls(m)_{\Q}$ as $\Q$-modules.

    In degree $n$ and with $s=n$, $\STU^2$ becomes $\4T$ while $\hexagon R$ is trivial, hence we have $\Fs^n_n(m)_{\Q}\cong\As_n(m)_{\Q}$.

    If $s<n-1$, the $\hexagon$ relations are actually implied by $\STU^2$ and $\IHX$. Indeed, by Remark \ref{rem:HR_is_image_of_IHX} any $\hexagon$ relation can be obtained by breaking apart a leg in a $\IHX$ relation. Since $s<n-1$, the terms of the $\hexagon$ relation must contain a node not involved in the relation. Apply $\STU^2$ to break that additional node and rebuild a sum of two $\IHX$ relations. Thus $\hexagon R\subset\langle\STU^2,\IHX\rangle$ as submodules of $\Q\Ds_n^s(m)$ when $s<n-1$.
\end{remark}

For each $1\leq s<n$, there is a linear map
\begin{equation}\label{eq:breaking_apart_a_vertex}
    \tilde{\iota}^s_n\colon\Q\Ds_n^s(m)\rightarrow\Fs_n^{s+1}(m)_{\Q}
\end{equation}
defined as follows: given a size $s$ forest diagram $F\in\Ds^s_n(m)$, pick a trivalent vertex (which must exist since $s<n$) adjacent to a strand and break it apart according to $\STU$, i.e.
\begin{align}\label{eq:definition_of_map_iota}
    \tilde{\iota}^s_n\paren{\vcenter{\hbox{\includegraphics[height=25pt]{Ressources/permutograph/STU2-satisfied-by-s/term-Y.eps}}}}\coloneqq \vcenter{\hbox{\includegraphics[height=25pt]{Ressources/permutograph/STU2-satisfied-by-s/term-=.eps}}}-\vcenter{\hbox{\includegraphics[height=25pt]{Ressources/permutograph/STU2-satisfied-by-s/term-x.eps}}}.
\end{align}
By $\STU^2$, this does not depend on the choice of trivalent vertex to break apart.

\begin{proposition}\label{prop:map_iota_descends}
    The map $\tilde{\iota}^s_n\coloneqq\Q\Ds_n^s(m)\rightarrow\Fs_n^{(\FI),s+1}(m)_{\Q}$ factors through the linear map $\iota^s_n\colon\Fs^{(\FI),s}_n(m)_{\Q}\rightarrow\Fs^{(\FI),n+1}_n(m)_{\Q}$.
\end{proposition}
\begin{proof}
    This proof is very similar to the proof of Claim 3.4 in \cite{conant2008stu2}.

    We need to show that $\tilde{\iota}^s_n$ vanishes on $\AS,\IHX,\STU^2,\hexagon R$. If $s<n-2$, then forest diagrams in $\Ds_n^s(m)$ have at least three nodes, while the relations $\AS,\IHX,\STU^2,\hexagon R$ involve at most $2$ nodes, hence it suffices to pick another node to break apart, hence the image under $\tilde{\iota}^s_n$ is a difference of sums that belong to the respective submodules.

    For $s<n-1$, the above argument also applies to $\STU^2,\hexagon R,\AS$. For $s=n-2$ and $\IHX$, observe that an $\IHX$ relation is sent to an $\hexagon$ relation.

    For $s=n-1$, there are no $\IHX$ relations, while a direct calculation shows that $\STU^2$ and $\AS$ relations are sent to zero.

    This concludes the proof that $\tilde{\iota}^s_n$ factors through $\Fs^s_n(m)_{\Q}$. The $\FI$ case holds as well: if $D$ has an isolated chord, then so do the two terms in $\tilde{\iota}^s_n(D)$.
\end{proof}

\begin{remark}
    In \cite[Claim 3.4]{conant2008stu2}, the proof restricts to the case $s\neq n-1$ because the $\hexagon$ relation is not taken into account there. This relation appears naturally here by considering graphs of forests and solves that issue.
\end{remark}

The linear maps $\iota^s_n$, $s=1,\dots,n-1$ provide a factorization of the map $\iota_n\colon\Fs^1_n(m)_{\Q}\rightarrow\As_n(m)_{\Q}$, where nodes are broken apart one after the other, instead of all at once
\[\begin{tikzcd}
	{\Fs^1_n(m)_{\Q}} & {\Fs^2_n(m)_{\Q}} & \dots & {\Fs^n_n(m)_{\Q}=\As_n(m)_{\Q}}
	\arrow["{\iota^1_n}", from=1-1, to=1-2]
	\arrow["{\iota^{2}_n}", from=1-2, to=1-3]
	\arrow["{\iota^{n-1}_s}", from=1-3, to=1-4]
	\arrow["\iota_n"',bend right=15, from=1-1, to=1-4]
\end{tikzcd}.\]

\subsubsection{Barycenters of forests.}\label{sec:Barycenters of forests}
Let $T_1,\dots,T_s\in\Ds^T(m)$ be tree diagrams. The set of vertices of the associated graph of forests $\Fs(T_1,\dots,T_s)$ shares some properties with affine spaces. Recall that an \emph{affine space} can be defined as a principal homogeneous space $A$ for the action of the additive group of a vector space $\overrightarrow{A}$, its \emph{direction}. In particular, there is a well-defined map
\begin{equation}
\overrightarrow{\cdot}\colon A\times A\longrightarrow \overrightarrow{A}\colon
(a,b)\longmapsto \overrightarrow{ab}.
\end{equation}
Although $\Fs^{s-1}(m)_{\Q}$ does not act on $\Fs(T_1,\dots,T_s)$, Definition \ref{def:vector_from_F_to_F'} defines an analogous map
\begin{equation}\label{eq:direction_map}
\overrightarrow{\cdot}\colon \Fs(T_1,\dots,T_s)\times \Fs(T_1,\dots,T_s)\longrightarrow\Fs^{(\FI),s-1}(m)_{\Q}\colon (F,F')\longmapsto \overrightarrow{FF'}.
\end{equation}
Consequently, $\Fs(T_1,\dots,T_s)$ inherits some properties of affine spaces. In particular, we can still talk about barycenters of vertices of $\Fs(T_1,\dots,T_s)$.

\begin{definition}
The \emph{barycenter} of forests $F_1,\dots,F_k\in\Fs(T_1,\dots,T_s)$ with coefficients $\lambda_1,\dots,\lambda_k\in\Q$ summing to $\sum_i\lambda_i=1$ is the formal sum $\sum_i\lambda_i F_i$.

Denote by $\Bary(\Fs(T_1,\dots,T_s))$ the set of all barycenters of forests in $\Fs(T_1,\dots,T_s)$. There is a natural inclusion $\Fs(T_1,\dots,T_s)\monic\Bary(\Fs(T_1,\dots,T_s))$ sending a forest $F$ to the formal sum $1\cdot F$.
\end{definition}

\begin{lemma}[Chasles relation]\label{lem:Chasles}
    If $F_0,F_1,F_2\in\Fs(T_1,\dots,T_s)$, then $\overrightarrow{F_0F_2}=\overrightarrow{F_0F_1}+\overrightarrow{F_1F_2}$.
\end{lemma}
\begin{proof}
    By Definition \ref{def:vector_from_F_to_F'}, we can choose any path from $F_0$ to $F_2$ in order to obtain $\overrightarrow{F_0F_2}$. In particular, choosing a path that passes through $F_1$, we get $\overrightarrow{F_0F_2}=\overrightarrow{F_0F_1}+\overrightarrow{F_1F_2}$.
\end{proof}

\begin{lemma}\label{lem:choice_origin_barycenter_irrelevant}
    Let $F_0,F'_0,F_1,\dots,F_k\in\Fs(T_1,\dots,T_s)$ and $\lambda_1,\dots,\lambda_k\in\Q$. If $\sum_i\lambda_i=0$, then $\sum_{i}\lambda_i \overrightarrow{F_0F_i}=\sum_{i}\lambda_i\overrightarrow{F'_0F_i}$ in $\Fs^{(\FI),s-1}(m)_{\Q}$.
\end{lemma}
\begin{proof}
    By the Chasles relation (Lemma \ref{lem:Chasles}), we immediately get
    \begin{align*}
        {\sum}_{i}\lambda_i\overrightarrow{F'_0F_i}=\underbrace{\paren{{\sum}_{i}\lambda_i}}_{=0}\overrightarrow{F'_0F_0}+{\sum}_{i}\lambda_i\overrightarrow{F_0F_i}.
    \end{align*}
\end{proof}

\begin{definition}\label{def:vector_barycenter}
    Given two barycenters $B=\sum_i\lambda_iF_i$, $B'=\sum_j\lambda'_iF'_i$ in $\Bary(\Fs(F))$, we define the \emph{vector} from $B$ to $B'$ as
    \begin{equation}\label{eq:definition_vector_barycenter}
        \overrightarrow{BB'}\coloneqq\sum_j\lambda'_j \overrightarrow{F_0F'_j}-\sum_i\lambda_i\overrightarrow{F_0F_i}
    \end{equation}
    where $F_0$ is any forest in $\Fs(F)$.
\end{definition}

\begin{proposition}\label{prop:direction_extends_to_barycenters}
    Definition \ref{def:vector_barycenter} is well-defined. More precisely, the map
    \begin{equation}\label{eq:direction_map_barycenters}
    \overrightarrow{\cdot}\colon\Bary(\Fs(T_1,\dots,T_s))\times\Bary(\Fs(T_1,\dots,T_s))\longrightarrow\Fs^{(\FI),s-1}(m)_{\Q}
    \end{equation}
    extends the map (\ref{eq:direction_map}) and the Chasles relation holds for barycenters of forests.
\end{proposition}
\begin{proof}
    By Lemma \ref{lem:choice_origin_barycenter_irrelevant}, the value of $\overrightarrow{BB'}$ does not depend on the choice of origin $F_0$.

    For barycenters of the form $1\cdot F$ and $1\cdot F'$, we have $\overrightarrow{(1\cdot F)(1\cdot F')}=\overrightarrow{FF'}$, e.g.\ by choosing $F_0=F$. Thus the map $\overrightarrow{\cdot}$ for barycenters indeed extends that for forests.
    
    For three barycenters $B=\sum_i\lambda_iF_i$, $B'=\sum_i\lambda'_iF'_i$ and $B''=\sum_i\lambda''_iF''_i$, the Chasles relation $\overrightarrow{BB''}=\overrightarrow{BB'}+\overrightarrow{B'B''}$ follows by writing $\lambda''_i-\lambda_i=\lambda_i''-\lambda_i'+\lambda_i'-\lambda_i$.
\end{proof}

It follows from Proposition \ref{prop:direction_extends_to_barycenters} that any choice of barycenter $B_0\in\Bary(\Fs(T_1,\dots,T_s))$ defines a map
\begin{equation}        \Fs(T_1,\dots,T_s)\longrightarrow\Fs^{(\FI),s-1}(m)_{\Q}\colon F\longmapsto \overrightarrow{B_0F}
\end{equation}
which sends a forest of size $s$ to a linear combination of forests of size $s-1$. Maps like the above will be crucial in defining sections to the maps $\iota^s_n$. Doing so requires picking a barycenter in each graph of forest, but random choices will not work in general. The barycenters must be consistently chosen.

\begin{definition}\label{def:average_barycenter}
    For $F$ a forest of size $s$ consisting of trees $T_1,\dots,T_s$, define the \emph{average barycenter}, or \emph{average basepoint}, associated to $F$ as
    \begin{equation}\label{eq:average_barycenter}
        (F)_{\avg}\coloneqq\frac{1}{s!}\sum_{\sigma\in\Sy_s}T_{\sigma(1)}\dots T_{\sigma(s)},
    \end{equation}
    i.e.\ the average of all possible ways of stacking the trees of $F$ one above the other.
\end{definition}

Note that $(F)_{\avg}=(F')_{\avg}$ whenever $F$ and $F'$ consist of the same trees, i.e.\ they belong to a common graph of forests.

\subsubsection{Merging trees.}\label{sec:merging trees}
Consider the linear map
\begin{equation}\label{eq:pi_tilde}
\tilde{\pi}^s_n\colon\Q\Ds^s_n(m)\longrightarrow\Fs_n^{(\FI),s-1}(m)_{\Q}\colon F\longmapsto \overrightarrow{(F)_{\avg}F}
\end{equation}
that assigns to a forest $F$ the vector from the average $(F)_{\avg}$ to $F$. It can be thought of as a sort of distance function from the average barycenters.

\begin{lemma}[Diagrammatic $\STU$]\label{lem:diagrammatic_STU}
    For any adjacent forests $F^==\vcenter{\hbox{\includegraphics[height=20pt]{Ressources/permutograph/STU2-satisfied-by-s/term-=.eps}}}$ and $F^{\times}=\vcenter{\hbox{\includegraphics[height=20pt]{Ressources/permutograph/STU2-satisfied-by-s/term-x.eps}}}$ in a graph of forests, we have
    \begin{align*}
        \tilde{\pi}^{s}_{n}\paren{\vcenter{\hbox{\includegraphics[height=25pt]{Ressources/permutograph/STU2-satisfied-by-s/term-=.eps}}}-\vcenter{\hbox{\includegraphics[height=25pt]{Ressources/permutograph/STU2-satisfied-by-s/term-x.eps}}}}=\vcenter{\hbox{\includegraphics[height=25pt]{Ressources/permutograph/STU2-satisfied-by-s/term-Y.eps}}}
    \end{align*}
    where $\vcenter{\hbox{\includegraphics[height=20pt]{Ressources/permutograph/STU2-satisfied-by-s/term-Y.eps}}}-\vcenter{\hbox{\includegraphics[height=20pt]{Ressources/permutograph/STU2-satisfied-by-s/term-=.eps}}}+\vcenter{\hbox{\includegraphics[height=20pt]{Ressources/permutograph/STU2-satisfied-by-s/term-x.eps}}}$ is an $\STU$ relation.
\end{lemma}
\begin{proof}
    Since $F^==\vcenter{\hbox{\includegraphics[height=20pt]{Ressources/permutograph/STU2-satisfied-by-s/term-=.eps}}}$ and $F^{\times}=\vcenter{\hbox{\includegraphics[height=20pt]{Ressources/permutograph/STU2-satisfied-by-s/term-x.eps}}}$ belong to the same graph of forests, we have $(F^{=})_{\avg}=(F^{\times})_{\avg}$ and
    \begin{align*}
    \tilde{\pi}^{s}_n(F^{=}-F^{\times})=\tilde{\pi}^{s}_n(F^{=})-\tilde{\pi}^{s}_n(F^{\times})=\overrightarrow{(F^{=})_{\avg}F^{=}}-\overrightarrow{(F^{\times})_{\avg}F^{\times}}=\overrightarrow{F^{\times}F^{=}}=F^Y
    \end{align*}
    by the Chasles relation.
\end{proof}

\begin{proposition}\label{prop:map_pi_descends}
    The map $\tilde{\pi}^s_n\colon\Q\Ds_n^s(m)\rightarrow\Fs_n^{(\FI),s-1}(m)_{\Q}$ factors through $\pi^s_n\colon\Fs^{(\FI),s}_n(m)_{\Q}\rightarrow\Fs^{(\FI),s-1}_n(m)_{\Q}$.
\end{proposition}
\begin{proof}
    We need to show that $\tilde{\pi}^s_n$ vanishes on $(\1T),\AS,\IHX,\STU^2$ and $\hexagon R$. The argument is similar to \cite[Claim 3.6]{conant2008stu2}.

    Start with an $\AS$ relation $F+F'$ in $\Ds^s_n(m)$. We can write $(F^{(')})_{\avg}=\sum_{\sigma}\lambda_{\sigma}F^{(')}_{\sigma}$ where $F_{\sigma}+F'_{\sigma}$ are $\AS$ relations (the notation $(')$ indicates that we can read it with or without~$'$). Pick a path $F_{\sigma}\overset{e_0}{\rightsquigarrow}\dots\overset{e_{l-1}}{\rightsquigarrow}F$ of leg slides in $\Fs(F)$. Then we can follow the parallel path $F'_{\sigma}\overset{e'_0}{\rightsquigarrow}\dots\overset{e'_{l-1}}{\rightsquigarrow}F'$ in $\Fs(F')$, obtained by performing the exact same leg slides. Notice that for each $j$, the origin and extremities of $e_j$ and $e'_j$ only differ by the cyclic ordering of a node, in the same trees as for $F$ and $F'$. Consequently, $\overrightarrow{e_j}+\overrightarrow{e'_j}$ is an $\AS$ relation for each $j$ and $\overrightarrow{F_{\sigma}F}+\overrightarrow{F'_{\sigma}F'}$ is a sum of those. In particular,
    \begin{align*}
        \overrightarrow{(F)_{\avg}F}+\overrightarrow{(F')_{\avg}F'}={\sum}_{\sigma}\lambda_{\sigma}(\overrightarrow{F_{\sigma}F}+\overrightarrow{F'_{\sigma}F'})=0.
    \end{align*}

    For an $\IHX$ relation $F^I+F^H+F^X$, the argument is the same. As above, picking parallel paths from $F^{\bullet}_{\sigma}$ to $F^{\bullet}$, for $\bullet\in\{I,H,X\}$, we obtain
    $$\overrightarrow{(F^I)_{\avg}F^I}+\overrightarrow{(F^H)_{\avg}F^H}+\overrightarrow{(F^X)_{\avg}F^X}=0$$
    since it is a sum of $\IHX$ relations.

    Unlike $\AS$ and $\IHX$ which happen far from the strands, $\STU^2$ happens close to the strands and we need to distinguish two cases. If the $\STU^2$ relation $F^{Y=}-F^{Y\times}-F^{=Y}+F^{\times Y}$ involves two or three trees, then the legs forming $=$ or $\times$ must belong to distinct trees\footnote{Otherwise, joining them in the term with $Y$ would form a cycle.} Then, two applications of the diagrammatic $\STU$ (Lemma \ref{lem:diagrammatic_STU}) yield
    \begin{align*}
        \tilde{\pi}^{s}_{n}(F^{Y=}-F^{Y\times})=F^{YY}=\tilde{\pi}^{s}_{n}(F^{=Y}-F^{\times Y}),
    \end{align*}
    hence the $\STU^2$ relation is satisfied by $\tilde{\pi}^s_n$ in this case. Let us turn to the second case, where the $\STU^2$ relation involves a single tree, i.e.\ all the legs involved belong to the same tree. We argue as for the $\AS$ and $\IHX$ relations, by picking four parallel paths of slide moves from $F^{\bullet}_{\sigma}$ to $F^{\bullet}$, for $\bullet\in\{Y\!=,Y\times,=\!Y,\times Y\}$. More precisely, call the places involved in the $\STU^2$ relation $\alpha$ and $\beta$
    \begin{align*}
        \underbrace{\vcenter{\hbox{\includegraphics[height=25pt]{Ressources/permutograph/STU2-satisfied-by-s/term-=.eps}}}}_{\alpha}\underbrace{\vcenter{\hbox{\includegraphics[height=25pt]{Ressources/permutograph/STU2-satisfied-by-s/term-Y.eps}}}}_{\beta}-\vcenter{\hbox{\includegraphics[height=25pt]{Ressources/permutograph/STU2-satisfied-by-s/term-x.eps}}}\;\vcenter{\hbox{\includegraphics[height=25pt]{Ressources/permutograph/STU2-satisfied-by-s/term-Y.eps}}}-\vcenter{\hbox{\includegraphics[height=25pt]{Ressources/permutograph/STU2-satisfied-by-s/term-Y.eps}}}\;\vcenter{\hbox{\includegraphics[height=25pt]{Ressources/permutograph/STU2-satisfied-by-s/term-=.eps}}}+\vcenter{\hbox{\includegraphics[height=25pt]{Ressources/permutograph/STU2-satisfied-by-s/term-Y.eps}}}\;\vcenter{\hbox{\includegraphics[height=25pt]{Ressources/permutograph/STU2-satisfied-by-s/term-x.eps}}}.
    \end{align*}
    A leg sliding across site $\alpha$ induces one or two edges in the four parallel paths, depending on whether site $\beta$ contains a ``$Y$'', ``$=$'' or ``$\times$'' configuration. The corresponding terms in $\sum_{\bullet}\overrightarrow{F^{\bullet}_{\sigma}F^{\bullet}}$ add up to zero by $\STU^2$ and $\hexagon R$:

    \begin{align*}
        \vcenter{\hbox{\includegraphics[height=25pt]{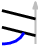}}}\vcenter{\hbox{\includegraphics[height=25pt]{Ressources/permutograph/STU2-satisfied-by-s/term-Y.eps}}}+\vcenter{\hbox{\includegraphics[height=25pt]{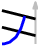}}}\vcenter{\hbox{\includegraphics[height=25pt]{Ressources/permutograph/STU2-satisfied-by-s/term-Y.eps}}}-\vcenter{\hbox{\includegraphics[height=25pt]{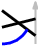}}}\vcenter{\hbox{\includegraphics[height=25pt]{Ressources/permutograph/STU2-satisfied-by-s/term-Y.eps}}}-\vcenter{\hbox{\includegraphics[height=25pt]{Ressources/permutograph/STU2-satisfied-by-s/term-xbottom.eps}}}\vcenter{\hbox{\includegraphics[height=25pt]{Ressources/permutograph/STU2-satisfied-by-s/term-Y.eps}}}\underbrace{-\vcenter{\hbox{\includegraphics[height=25pt]{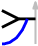}}}\vcenter{\hbox{\includegraphics[height=25pt]{Ressources/permutograph/STU2-satisfied-by-s/term-=.eps}}}+\vcenter{\hbox{\includegraphics[height=25pt]{Ressources/permutograph/STU2-satisfied-by-s/term-Y+.eps}}}\vcenter{\hbox{\includegraphics[height=25pt]{Ressources/permutograph/STU2-satisfied-by-s/term-x.eps}}}}_{\overset{\STU^2}{=}-\vcenter{\hbox{\includegraphics[height=25pt]{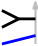}}}\vcenter{\hbox{\includegraphics[height=25pt]{Ressources/permutograph/STU2-satisfied-by-s/term-Y.eps}}}+\vcenter{\hbox{\includegraphics[height=25pt]{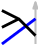}}}\vcenter{\hbox{\includegraphics[height=25pt]{Ressources/permutograph/STU2-satisfied-by-s/term-Y.eps}}}}\overset{\hexagon R}{=}0.
    \end{align*}

    When a leg slides across site $\beta$, the argument is the same. Thus $\tilde{\pi}^s_n$ satisfies $\STU^2$.

    It remains to show that $\hexagon$-relations are satisfied. This immediately follows from three applications of the diagrammatic $\STU$ Lemma (Lemma \ref{lem:diagrammatic_STU}) and the $\IHX$ relation.

    This concludes the proof that $\tilde{\pi}^s_n$ factors through $\Fs^s_n(m)_{\Q}$.

    In the $\FI$ case, we need to check that $\tilde{\pi}^s_n$ vanishes on $\1T$. Let $F\in\Ds^s_n(m)$ be a forest diagram with an isolated chord. For $\sigma\in\Sy_{s}$, the element $\overrightarrow{(F)_{\sigma}F}$ is a sum of diagrams in which this isolated chord remains, together with diagrams coming from sliding that isolated chord across another leg. Those terms of the latter type vanish in pairs by $\AS$. This also follows from the commutativity property, shown in the proof of \cite[Lemma 3.1]{bar1995vassiliev}.
\end{proof}

\begin{proposition}\label{prop:pi_section_of_iota}
For each $1\leq s<n$, the map $\pi^{s+1}_n$ is a section of $\iota^s_n$, i.e.\ $\pi^{s+1}_n\circ\iota^s_n=\id_{\Fs_n^{(\FI),s}(m)_{\Q}}$. In particular, $\iota^s_n$ is injective.
\end{proposition}
\begin{proof}
    Given a size $s$ and degree $n$ forest diagram $F=\vcenter{\hbox{\includegraphics[height=20pt]{Ressources/permutograph/STU2-satisfied-by-s/term-Y.eps}}}$ (represented by one of its nodes adjacent to a strand), we have $\iota^{s}_n\paren{\vcenter{\hbox{\includegraphics[height=20pt]{Ressources/permutograph/STU2-satisfied-by-s/term-Y.eps}}}}=\vcenter{\hbox{\includegraphics[height=20pt]{Ressources/permutograph/STU2-satisfied-by-s/term-=.eps}}}-\vcenter{\hbox{\includegraphics[height=20pt]{Ressources/permutograph/STU2-satisfied-by-s/term-x.eps}}}$ and then
    \begin{align*}
        \pi^{s+1}_n\circ\iota^{s}_n(F)=\pi^{s+1}_n\circ\iota^{s}_n\paren{\vcenter{\hbox{\includegraphics[height=25pt]{Ressources/permutograph/STU2-satisfied-by-s/term-Y.eps}}}}=\pi^{s+1}_n\paren{\vcenter{\hbox{\includegraphics[height=25pt]{Ressources/permutograph/STU2-satisfied-by-s/term-=.eps}}}-\vcenter{\hbox{\includegraphics[height=25pt]{Ressources/permutograph/STU2-satisfied-by-s/term-x.eps}}}}=\vcenter{\hbox{\includegraphics[height=25pt]{Ressources/permutograph/STU2-satisfied-by-s/term-Y.eps}}}=F
    \end{align*}
    by the diagrammatic $\STU$ Lemma (Lemma \ref{lem:diagrammatic_STU}).
\end{proof}

\subsubsection{From \texorpdfstring{$\Fs^k(m)_{\Q}$}{Fk(m)Q} to \texorpdfstring{$F^k\As(m)_{\Q}$}{FkA(m)Q}.}
Combining the results of Sections \ref{sec:Primitive filtration of forest diagrams} and \ref{sec:merging trees}, we obtain the following theorem.

\begin{theorem}\label{thm:presentation_primitive_filtration}
    For each $1\leq k\leq n$, there is a natural linear isomorphism
    \begin{equation}\label{eq:presentation_primitive_filtration}
        \Fs_n^{(\FI),k}(m)_{\Q}\overset{\cong}{\longrightarrow}F^k\As^{(\FI)}_n(m)_{\Q}.
    \end{equation}
    More precisely, we have the following isomorphism of filtrations (with $\FI$ as well):
    \begin{equation}\label{diag:comparison_filtrations_concrete}
\begin{tikzcd}
	{\Fs^1_n(m)_{\Q}} & {\Fs^2_n(m)_{\Q}} & \dots & {\Fs^{n-1}_n(m)_{\Q}} & {\Fs^n_n(m)_{\Q}=\As_n(m)_{\Q}} \\
	{F^1\As_n(m)_{\Q}} & {F^2\As_n(m)_{\Q}} & \dots & {F^{n-1}\As_n(m)_{\Q}} & {F^n\As_n(m)_{\Q}=\As_n(m)_{\Q}}
	\arrow["{i_n^1}"', hook, from=2-1, to=2-2]
	\arrow["{i^2_n}"', hook, from=2-2, to=2-3]
	\arrow["{\iota_n^1}", hook, from=1-1, to=1-2]
	\arrow["{\iota_n^2}", hook, from=1-2, to=1-3]
	\arrow["\cong"{description}, from=1-1, to=2-1]
	\arrow["\cong"{description}, from=1-2, to=2-2]
	\arrow[equal, from=1-5, to=2-5]
	\arrow["{\iota_n^{n-2}}", hook, from=1-3, to=1-4]
	\arrow["{\iota_n^{n-2}}", hook, from=1-4, to=1-5]
	\arrow["{i_n^{n-2}}"', hook, from=2-3, to=2-4]
	\arrow["{i_n^{n-1}}"', hook, from=2-4, to=2-5]
	\arrow["\cong"{description}, from=1-4, to=2-4]
\end{tikzcd},
    \end{equation}
    under which the sections $s^k_n$ and $\pi^k_n$ coincide.
\end{theorem}
\begin{proof}
    We omit $\FI$. By definition, $F^k\As_n(m)_{\Q}$ is generated by size $k$ forests in $\As_n(m)_{\Q}$, i.e.\ there is a linear surjection $\Q\Ds^k(m)\epic F^k\As_n(m)_{\Q}$. The relations $\AS,\IHX,\STU^2,\hexagon R$ in $\Q\Ds^k(m)$ are all consequences of the $\AS,\IHX,\STU$ relations that hold in $\As_n(m)_{\Q}$, hence this surjection descends to a surjection
    $$\Fs_n^k(m)_{\Q}\epic F^k\As_n(m)_{\Q}.$$
    Notice that those surjections, for $1\leq k\leq n$, combine into the vertical maps in (\ref{diag:comparison_filtrations_concrete}) and commute with the maps $i^k_n$ and $\iota^k_n$. By Proposition \ref{prop:pi_section_of_iota}, the maps $\iota$ are injective, hence so are the vertical maps.

    To see that the sections $s^k_n$ and $\pi^k_{n}$ coincide, compare (\ref{eq:description_s_section}) with (\ref{eq:pi_tilde}) and Definition \ref{def:average_barycenter} of $(\cdot)_{\avg}$.
\end{proof}

\subsubsection{The Lie algebra of trees.}\label{sec:The Lie algebra of trees}
From the presentation of the primitive filtration of $\As_n(m)_{\Q}$ in Theorem \ref{thm:presentation_primitive_filtration}, we directly obtain a concrete presentation of the primitive Lie algebra $\Prim(\As(m)_{\Q})$.
\begin{theorem}\label{thm:presentation_primitive_Lie_algebra}
The primitive Lie algebra $\Prim(\As^{(\FI)}(m)_{\Q})$ is naturally isomorphic to the \emph{Lie algebra of trees}
\begin{equation}
	\Ls^{(\FI)}(m)_{\Q}\coloneqq\Fs^{(\FI),1}(m)_{\Q}=\faktor{\Q\Ds^T(m)}{\langle(\1T),\AS,\IHX,\STU^2\rangle}
\end{equation}
endowed with the bracket
\begin{equation}\label{eq:bracket_Lie_algebra_of_trees}
\begin{split}
    [,]\colon\Ls^{(\FI)}(m)_{\Q}\times \Ls^{(\FI)}(m)_{\Q}&\longrightarrow\Ls^{(\FI)}(m)_{\Q}\\
    (T,T')&\longmapsto [T,T']\coloneqq \overrightarrow{(T'\cdot T)(T\cdot T')}
\end{split}
\end{equation}
for any $T,T'\in\Ds^T(m)$ and extended linearly.
\end{theorem}
\begin{proof}
Under the isomorphism of filtrations from Theorem \ref{thm:presentation_primitive_filtration} and by Theorem \ref{thm:comparison_primitive_and_size_filtration}, we have an isomorphism of $\Q$-modules (we omit $(\FI)$) $$\Ls(m)_{\Q}\cong F^1\As_{\geq 1}(m)_{\Q}\cong\Prim(\As(m)_{\Q})$$
and $\overrightarrow{(T'\cdot T)(T\cdot T')}$ in $\Ls(m)_{\Q}$ corresponds to $T\cdot T'-T'\cdot T$ in $Prim(\As(m)_{\Q})$ by definition of $\overrightarrow{(\cdot)}$, hence $\Ls(m)_{\Q}\cong\Prim(\As(m)_{\Q})$ as Lie algebras as well. This concludes the proof.
\end{proof}

\begin{remark}\label{rem:Lie_algebra_of_trees_is_graded}
    The Lie algebra of trees $\Ls^{(\FI)}(m)_{\Q}$ is graded by the degree of tree diagrams, with degree $n$ graded part $\Ls^{(\FI)}_n(m)_{\Q}=\Fs^{(\FI),1}_n(m)_{\Q}$. By construction, the Lie bracket is compatible with the grading:
    $$[\Ls^{(\FI)}_n(m)_{\Q},\Ls^{(\FI)}_{n'}(m)_{\Q}]\subset\Ls^{(\FI)}_{n+n'}(m)_{\Q}.$$
\end{remark}

\section{The rational Goussarov--Habiro Lie algebra of string links}
\subsection{Clasper calculus and realization of diagrams}\label{chap:claspers}
In the previous section, we identified the rational primitive Lie algebra of tree diagrams $\Ls^{\FI}(m)_{\Q}$. It is given by concrete relations on tree diagrams, hence we can check that the \emph{clasper surgery realization} $R\colon\Q\Ds(m)\rightarrow\Ls L(m)_{\Q}$ satisfies those relations, so that it factors through $\Prim(\As^{\FI}(m)_{\Q})\rightarrow\Ls L(m)_{\Q}$. In other words, we check that the diagrammatic relations $\1T,\AS,\IHX,\STU^2$ are satisfied geometrically in the Goussarov--Habiro Lie algebra. Most of the argument is adapted from known clasper manipulations from Habiro \cite{habiro2000claspers}, Goussarov \cite{goussarov1998interdependent}, Ohtsuki \cite{ohtsuki2002quantum}, Conant--Teichner \cite{conantteichner2004classicalknots,conantteichner2004gropefeynmandiagrams} and Kosanovic \cite{kosanovic2020thesis}. In particular, we provide a detailed proof of the $\IHX$ relation, different from the ones in the aforementioned references.

The definitions of claspers and clasper surgeries are not recalled in details. We refer the reader to the original papers of Habiro \cite{habiro2000claspers} and Goussarov \cite{goussarov1998interdependent,goussarov2000finite}.

\subsubsection{Claspers.}
A \emph{simple tree clasper} $C$ on a string link $\gamma$ is a ribbon uni-trivalent tree diagram embedded in the complement of $\gamma$, with \emph{leaves} attached to the strands. See Figure \ref{fig:C3_move}, on the left, where $\gamma$ is the trival string link on three strands. We use Habiro's drawing convention \cite[Figure 7]{habiro2000claspers}: nodes (tri-valent vertices) are dots, leaves (uni)valent vertices) are disks clasping around a strand and ribbon edges have the blackboard framing with half-twists indicated by small barred disks.

\begin{figure}[!ht]
    \centering
    \includegraphics[width=0.45\textwidth]{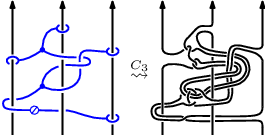}
    \caption{A degree $3$ simple tree clasper and the resulting clasper surgery.}
    \label{fig:C3_move}
\end{figure}

A tree clasper $C$ has a degree $\deg(C)=\#\text{leaves}-1$, which is equal to the degree of its underlying tree diagram. For example, the tree claspers in Figure \ref{fig:C1_C2_moves} have degree $1$ and $2$, respectively, while the one in Figure \ref{fig:C3_move} has degree $3$.

In \cite{habiro2000claspers}, Habiro defines more general claspers than just simple tree claspers. Those naturally appear when manipulating tree claspers. A clasper $C$ determines a link $L_C$. Dehn surgery along $L_C$ is called \emph{clasper surgery along $C$}. A clasper surgery does not modify the ambient manifold\footnote{In general, this is the case when the clasper is \emph{tame}, which is always the case in our context. See \cite{habiro2000claspers} for more details.} (here, the cylinder), but does modify the string link. The result of clasper surgery along a clasper $C$ on $\gamma$ is denoted $\gamma^C$. When $\gamma$ is the trivial string link $\gamma_0=1$, we write $\sigma(C)\coloneqq \gamma_0=1^C$. See Figures \ref{fig:C1_C2_moves} and \ref{fig:C3_move} for examples of clasper surgeries.

\emph{Habiro's clasper calculus} consists in a collection of \emph{moves}, i.e.\ modifications on claspers that do not alter the resulting surgery $\gamma^C$, those moves are pictured in \cite[Proposition 2.7]{habiro2000claspers}, from move 1 to move 12. We will refer to Habiro's moves by the same numbers and sometimes write, e.g.\ $H7$ for move 7.

Observe that clasper surgeries are entirely supported in a regular neighborhood of the clasper, hence they define \emph{local} operations on the set of string links and naturally lead to the study of equivalence classes of string links under those moves. Different filtrations or degrees on the set of claspers define various equivalence relations.

\subsubsection{The Goussarov--Habiro Lie algebra of string links.}
In this paper, we stick to the usual degree for trees, inducing local moves known as \emph{$C_n$-moves}, for $n\geq 1$. Let $\gamma$ be a string link. A \emph{$C_n$-move} on $\gamma$ is a surgery along a degree $n$ tree clasper on $\gamma$. Two string links $\gamma,\gamma'$ are said to be \emph{$C_n$-equivalent}, denoted $\gamma\overset{C_n}{\sim}\gamma'$, if one can be obtained from the other by applying a sequence of $C_n$-moves. The $C_n$-equivalence relation is symmetric by \cite[Proposition 3.23]{habiro2000claspers}. This defines the Goussarov--Habiro filtration of the monoid of string links $L(m)$
\begin{equation}\label{eq:GH_filtration}
    L(m)=L_1(m)\supset L_2(m)\supset L_3(m)\supset\dots
\end{equation}
where $L_n(m)\coloneqq\{\gamma\in L(m)\mid\gamma\overset{C_n}{\sim}1\}$ is the submonoid of $C_n$-trivial string links, i.e.\ string links that are $C_n$-equivalent to the trivial string link $1$. The quotient monoids $L_k(m)/C_{n+1}$ are finitely generated nilpotent groups \cite[Theorem 5.4]{habiro2000claspers}. They are abelian groups when $k=n$, in which case they are denoted
\begin{equation}\label{eq:graded_pieces_of_GH_Lie_algebra}
    \Ls_n L(m)\coloneqq \overline{L_n(m)}\coloneqq L_n(m)/C_{n+1},
\end{equation}
and more generally, for any $1\leq k,k'\leq n$, one has
\begin{equation}\label{eq:N_series_property}
    \left[L_k(m)/{C_{n+1}},L_{k'}(m)/{C_{n+1}}\right]\subset L_{k+k'}(m)/{C_{n+1}}.
\end{equation}
Thus the abelian groups (\ref{eq:graded_pieces_of_GH_Lie_algebra}) combine into the \emph{Goussarov--Habiro Lie algebra of string links on $m$ strands}
\begin{equation}\label{eq:GH_Lie_algebra}
    \Ls L(m)\coloneqq\bigoplus_{n\geq 1}\Ls_n L(m)\coloneqq\bigoplus_{n\geq 1}\overline{L_n(m)},
\end{equation}
whose Lie bracket is induced from (\ref{eq:N_series_property}). Alternatively, the truncation $\Ls_{\leq n}L(m)$ is the graded Lie algebra associated to the $N$-series $\{L_k(m)/C_{n+1
}\}_{1\leq k\leq n}$ of the group $L(m)/{C_{n+1}}$. See \cite[Theorem 2.1]{lazard1954groupes} and \cite{massuyeau2006finitetype}. By construction The Goussarov--Habiro Lie algebra is graded, and the Lie bracket satisfies the grading by (\ref{eq:N_series_property}).

\begin{remark}
    When $m=1$, the Goussarov--Habiro Lie algebra $\Ls L(1)$ is commutative, i.e.\ its bracket is trivial. This follows from the fact that $L(1)$ is a commutative monoid or, equivalently, that connected sum of knots is commutative.
\end{remark}

\subsubsection{Common moves and zip construction.}\label{sec:common_moves}
Combining Habiro's moves, one can obtain useful moves that we describe here: edge- and strand-crossing, sliding and twisting moves, and the zip construction.
\begin{figure}[!ht]
     \centering
     \begin{subfigure}[b]{0.3\textwidth}
         \centering
         \includegraphics[scale=1.4]{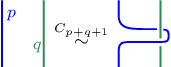}
         \caption{Crossing edges}
         \label{fig:crossing_edge}
     \end{subfigure}
     \hfill\quad
     \begin{subfigure}[b]{0.3\textwidth}
         \centering
         \includegraphics[scale=1.4]{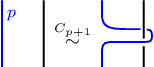}
         \caption{Crossing a strand}
         \label{fig:crossing_strand}
     \end{subfigure}
     \hfill
     \begin{subfigure}[b]{0.3\textwidth}
         \centering
         \includegraphics[scale=1.4]{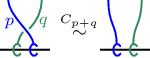}
         \caption{Sliding leaves}
         \label{fig:sliding_leaves}
     \end{subfigure}
    \begin{subfigure}[b]{0.4\textwidth}
         \centering
         \includegraphics[scale=1.45]{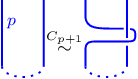}
         \caption{Self-crossing}
         \label{fig:crossing_self}
     \end{subfigure}
     \begin{subfigure}[b]{0.4\textwidth}
         \centering
         \includegraphics[scale=1.8]{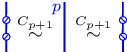}
         \caption{Twisting edges}
         \label{fig:twisting}
     \end{subfigure}
        \caption{Crossing edges and sliding leaves}
        \label{fig:sliding_and_crossing}
    \end{figure}
First, given two trees of respective degrees $p$ and $q$ in a clasper forests, one can perform a crossing change between edges of those two trees, up to a $C_{p+q+1}$-move, see Figure \ref{fig:crossing_edge}. Similarly, given a tree of degree $p$, one can perform a crossing change between one of its edges and a tangle strand, up to a $C_{p+1}$-move, see Figure \ref{fig:crossing_strand}. Lastly, one can slide two leaves, that belong to distinct trees of respective degrees $p$ and $q$, across one another, up to a $C_{p+q}$-move, see \ref{fig:sliding_leaves}. See \cite[\S 4]{habiro2000claspers} for proofs. From the moves depicted on Figure \ref{fig:sliding_and_crossing}(a,b,c) one can deduce two additional useful moves, depicted on Figure \ref{fig:sliding_and_crossing}(d,e), where the two blue lines in (d) belong to the same tree clasper.

The zip construction allows to get rid of boxes that may appear after some moves. One can imagine the box as a \emph{zipper} that will double everything when unzipped. See \cite[\S 3.3]{habiro2000claspers} for the construction. In Habiro's version of the zip construction, many new \emph{annoying} boxes appear during the process. If the starting tree has degree $n$, one can get rid of those annoying boxes modulo $C_{n+1}$-moves.

There is also a \emph{box-free} version of the zip construction\footnote{See also my master's thesis, available on my website, for more details and examples of the box-free zip construction.}, free of annoying boxes, detailed by Goussarov \cite{gusarov2001variations} and by Conant--Teichner \cite[\S 4.2]{conantteichner2004classicalknots}, based on the move described in \cite[Lemma 32]{conantteichner2004classicalknots}. Even though the box-free zip construction has no annoying boxes, it produces intricate linking of the resulting tree claspers. One can still get rid of those (modulo $C_{n+1}$) by applying the moves depicted on Figure \ref{fig:sliding_and_crossing}.

\subsubsection{Realizing tree diagrams in the Goussarov--Habiro Lie algebra}
We are now able to describe the realization map from tree diagrams to string links. Similar realization maps were considered by Ohtsuki in \cite{ohtsuki2002quantum}, Conant--Teichner \cite{conantteichner2004classicalknots,conantteichner2004gropefeynmandiagrams} and Kosanovic \cite{kosanovic2020thesis},.

\begin{theorem}\label{prop:realizing_trees}
    For each $n\geq 1$, there is a surjective $\Q$-linear morphism
    \begin{equation}\label{eq:realization_map_lie_algebras_over_Z}
        R_n\colon\Ls^{\FI}_n(m)=\faktor{\Q\Ds_n^T(m)}{\langle\1T,\AS,\IHX,\STU^2\rangle}\longrightarrow\Ls_nL(m)_{\Q}
    \end{equation}
    that sends a tree diagram $T$ to the class of the string link obtained by performing a clasper surgery along a clasper $C$ whose underlying diagram\footnote{What is meant by \emph{underlying diagram} is explicited in the proof.} is $T$. Those morphisms combine into a surjective morphism of Lie algebras over $\Q$
    \begin{equation}
        R\colon\Ls^{\FI}(m)_{\Q}\epic\Ls L(m)_{\Q}
    \end{equation}
    from the primitive Lie algebra of trees (Theorem \ref{thm:presentation_primitive_Lie_algebra}) onto the rational Goussarov--Habiro Lie algebra.
\end{theorem}    
\begin{proof}
    \textbf{Definition of the realization map:} First, we define the realization map $$\tilde{R}\colon\Ds_n^T(m)\rightarrow\overline{L_n}(m)$$
    as follows: start with a tree diagram $T\in\Ds_n^T(m)$ as on Figure \ref{fig:R-diagram}. Since $T$ is a tree, its underlying graph is planar and we can pick a planar projection where the cyclic orderings of the nodes are counterclockise. Place that planar projection of the tree to the left of the $m$ vertical strands, with the univalent vertices lying on a vertical line parallel to the strands, as on Figure \ref{fig:R-planar-clasper}. Attach the leaves of the tree to their respective positions on the strand, as given by the diagram $T$, with an additional \emph{positive} half-twist added to a single chosen leaf\footnote{The addition of that additional half-twist is related to the footnote on page 68 in \cite{habiro2000claspers}. Although it is not crucial in the case of realizing trees, it will be when realizing forests later on, in order for the $\STU$ relation to hold with the usual signs.} (see Figure \ref{fig:R-CT}). Outside of that additional half-twist, the attachment must be made without introducing half-twists, always going from left to right, and the disk-leaves must clasp to the tangle strands as on Figure \ref{fig:R-CT}. This yields a simple tree clasper $C_T$ on $m$ strands.
    \begin{figure}[!ht]
     \centering
        \begin{subfigure}[b]{0.155\textwidth}
              \centering
        \includegraphics[width=0.6\textwidth]{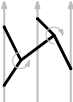}
        \caption{Diagram $T$}\label{fig:R-diagram}
    \end{subfigure}
     \hfill
     \begin{subfigure}[b]{0.23\textwidth}
              \centering
        \includegraphics[width=0.95\textwidth]{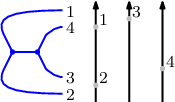}
        \caption{Planar tree}\label{fig:R-planar-clasper}
    \end{subfigure}
    \hfill
    \begin{subfigure}[b]{0.23\textwidth}
              \centering
        \includegraphics[width=0.95\textwidth]{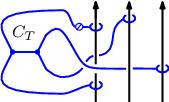}
        \caption{Clasper $C_T$}\label{fig:R-CT}
    \end{subfigure}
     \hfill
     \begin{subfigure}[b]{0.28\textwidth}
              \centering
        \includegraphics[width=0.95\textwidth]{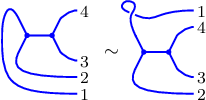}
        \caption{Changing the order}\label{fig:R-changing-order}
    \end{subfigure}
        \caption{Realizing a tree diagram}
        \label{fig:realizing_a_tree}
\end{figure}
    
    The only choices that are made in the above process are:
    \begin{enumerate}
        \item The choice of ordering of the univalent vertices on the vertical line in Figure \ref{fig:R-CT}. Perturbing this order cyclically introduces full twists, which do not matter modulo $C_{n+1}$-moves by (6) in the proof of \cite[Theorem 4.3]{habiro2000claspers}. See Figure \ref{fig:R-changing-order}.
        \item The choice of embedding of the edges joining the planar graph to the disk-leaves. Any two such embeddings are related by a sequence of crossing changes of pairs of edges or of edges with tangle strands, which do not matter modulo $C_{n+1}$-moves by Section \ref{sec:common_moves}.
        \item The choice of leaf to which we added the additional half-twist. A change of chosen leaf can be obtained by introducing two new half-twists: one to cancel out with the former one, and one as the new one. Since adding a half-twist corresponds to inversing modulo $C_{n+1}$ and inversing twice does not change anything, this choice does not matter modulo $C_{n+1}$.
    \end{enumerate}
    For a tree clasper $C_T$ in such a position, we say that its \emph{underlying diagram} is $D(C_T)\coloneqq T$.

    Define $\tilde{R}(T)\coloneqq \gamma_0^{C_T}=\sigma(C_T)$. Since $C_T$ is a degree $n$ tree clasper, the $C_{n+1}$-equivalence class of $\tilde{R}(T)$ is indeed an element of $\overline{L_n}(m)$.

    Since $\overline{L_n}(m)$ is an abelian group, the monoid morphism $\tilde{R}$ induces a morphism of abelian groups
    \begin{equation}\label{eq:group_morphism_free_abelian_group}
        \tilde{R}\colon\Z\Ds_n^T(m)\rightarrow\overline{L_n}(m).
    \end{equation}

    \textbf{The realization map $\tilde{R}$ is surjective:}
    By \cite[Lemma 5.5]{habiro2000claspers}, the abelian group $\overline{L_n}(m)$ is generated by string links of the form $\sigma(T)$ where $T$ is a simple tree clasper of degree $n$ on $\gamma_0$. Let $T$ be a simple tree clasper of degree $n$. Pull the nodes to the left of the strands, as on Figure \ref{fig:realizing_a_tree}(c). This may require strand-crossing and self-crossing moves, which do not change the class $[T]\in C_n(m)$. Flip the nodes so that their cyclic ordering is counterclockwise, and remove half-twists by pairs (again, this does not change the $C_{n+1}$-equivalence class of $T$), leaving either one or zero half-twists. The resulting representative of $[T]_{C_{n+1}}$ is then in the image of $\tilde{R}$. More precisely, it is given by $\tilde{R}(D(T))$ where $D(T)$ is its underlying tree diagram. This shows that $\tilde{R}$ is surjective.

    \textbf{The map $\tilde{R}$ satisfies $\1T$:} First, the $\1T$ relation in $\Z\Ds^T(m)$ is trivial except in degree $1$. In degree $1$, the $\1T$ relation is shown on Figure \ref{fig:1T_relation_for_claspers}. Note that choosing a negative half-twist instead of a positive half-twist results in a \emph{trefoil lonk knot}, which is $C_2$-equivalent to the trivial long knot.
    \begin{figure}[!ht]
        \centering
        \includegraphics[scale=1.6]{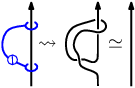}
        \caption{The $\1T$ relation for claspers}
        \label{fig:1T_relation_for_claspers}
    \end{figure}

    \textbf{The map $\tilde{R}$ satisfies $\AS$:} Flipping a node introduces three half-twists, hence we are done since each half-twist corresponds to inversing (see Theorem 4.7 in \cite{habiro2000claspers}):
    \begin{equation}
        \vcenter{\hbox{\includegraphics[scale=1.5]{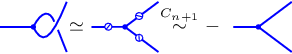}}}.
    \end{equation}
    
    \textbf{The map $\tilde{R}$ satisfies $\IHX$:} Proving that the $\IHX$ relation is satisfied requires a bit more care. A proof is given in \cite[Theorem 6.6]{gusarov2001variations} in the different but equivalent language of \emph{variations}. A sketch of proof is given in \cite[Lemma E.11]{ohtsuki2002quantum} and a proof in four dimensions is given in \cite{Rob2007Jacobi}. We give here a detailed argument.

    We will need the $C_{n+1}$-equivalence
    \begin{equation}\label{eq:Ohtsuki_IHX}
        \vcenter{\hbox{\includegraphics[scale=1.7]{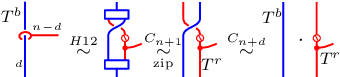}}}
    \end{equation}
    which starts with a blue tree clasper of degree $d\geq 1$ and a red clasper, which is a tree outside of the leaf depicted, of degree $n-d$. We apply Habiro's move 12, then the zip construction modulo $C_{n+1}$-moves to obtain two tree claspers $T^b$ and $T^r$, where $T^b$ (blue) is identical to the original blue tree, and $T^r$ (red) equals the union of those, with a node instead of the leaf and an additional half-twist. Note that $\deg(T^r)=n-d+d=n$. Then we can vertically separate $T^b$ and $T^r$ modulo $C_{n+d}$-moves, using sliding and edge-crossing moves (see Section \ref{sec:common_moves}). Since $n+d\geq n+1$, we obtain the desired $C_{n+1}$-equivalence. Note that $T^b$ and $T^r$ commute.

With this move in hands, we can prove that the $\IHX$ relation holds in $\overline{L_n}(m)$:
\begin{equation}\label{eq:IHX_proof}
        \vcenter{\hbox{\includegraphics[scale=1.65]{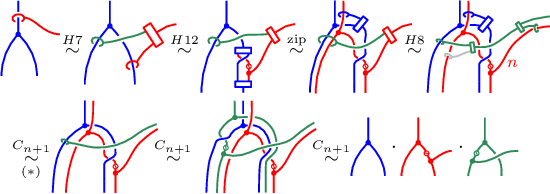}}}.
    \end{equation}
The first line of moves follows (\ref{eq:Ohtsuki_IHX}). The $C_{n+1}$-equivalence marked with $(\ast)$ is the subtler step. Before $(\ast)$, all the annoying boxes that have appeared during the zip construction are clasped to the red clasper, which is a tree clasper of degree $n$. Following the box-free zip construction (Section \ref{sec:common_moves}), we can get rid of those annoying boxes at the price of linking pieces of the red clasper in a complicated way around the blue and green claspers. Then apply the move from Figure (\ref{eq:Ohtsuki_IHX}) again to the green and blue claspers. Now that it is done, we can bring back the complicated pieces of the red clasper to their original positions, modulo $C_{n+1}$-equivalence. Finally, after vertically separating the trees, we obtain a \emph{product} of the blue, red and green trees. Again, this product is commutative since they have respective degrees $d,n$ and $n$.

At the beginning of (\ref{eq:IHX_proof}), we could have directly applied (\ref{eq:Ohtsuki_IHX}) instead of sliding the leaf down. Comparing the two $C_{n+1}$-equivalences, and cancelling the blue tree (by multiplying both sides by an inverse of that clasper modulo $C_{n+1}$), we obtain the $C_{n+1}$-equivalence
\begin{equation}\label{eq:IHX_conclusion}
        \vcenter{\hbox{\includegraphics[scale=1.7]{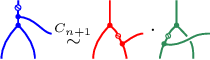}}}.
    \end{equation}
which yields the $\IHX$ relation modulo $C_{n+1}$, after removing the half-twists, i.e.\ taking the inverses of each side.

\textbf{The map $\tilde{R}$ satisfies $\STU^2$:} A proof in the case of knots is given in \cite[Sec.\ 10.3]{kosanovic2020thesis}. We detail here a very similar argument.

The first $C_{n+1}$-equivalence that we need is
\begin{equation}\label{eq:Ohtsuki_move}
    \vcenter{\hbox{\includegraphics[scale=1.5]{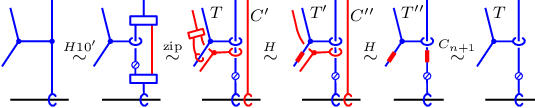}}}
\end{equation}
starting with a simple clasper $C$ with $n$ nodes and exactly one cycle. The cycle contains the depicted nodes. First, apply move 10' (which is move 10 in a slightly different form) and the zip construction to get a tree clasper $T$ and a red clasper $C'$. Remove the boxes, at the cost of replacing $T$ by $T'$, which has pieces that link in a neighborhood of the red clasper (those pieces are indicated with a red bar, there are situated where the annoying boxes used to clasp). Next, performing the surgery along the red clasper yields a new tree clasper $T''$, which gets an additional small red bar. Up to $C_{n+1}$-equivalence, we can bring back all those pieces of $T''$ to their original positions, to get $T$ again.

The $C_{n+1}$-equivalence depicted below is a form of $\STU$ relation for tree claspers \cite[Lemma E.11]{ohtsuki2002quantum}
\begin{equation}\label{eq:tree_STU}
    \vcenter{\hbox{\includegraphics[scale=1.5]{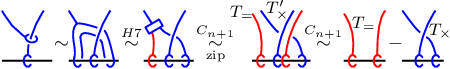}}}
\end{equation}
where the left-hand-side depicts a non-simple tree clasper of degree $n$, with one leaf clasping around one of its legs. First, slide down the leaf to replace it with two opposing leaves by move 7. Then the zip construction modulo $C_{n+1}$ yields two simple tree claspers $T_=$ and $T'_{\times}$. We can separate them, and remove the half-twist of $T_{\times}'$ up to a minus sign (we use the additive notation since we are working in the abelian group $\overline{L_n}(m)$).

Combining those $C_{n+1}$-equivalences, we can finally prove the $\STU^2$ relation for string links
\begin{equation}\label{eq:STU2_trees}
    \vcenter{\hbox{\includegraphics[scale=1.4]{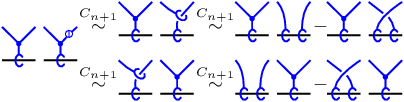}}}
\end{equation}
where the left-hand-side is a clasper (with $n$ nodes) with exactly one cycle which contains the two nodes depicted. Add a half-twist anywhere (this is just to avoid removing half-twists at the end). We can apply (\ref{eq:tree_STU}) to either one of the depicted nodes, hence replacing the node by a leaf together with a half-twist, which cancels out with the original half-twist modulo $C_{n+1}$. Then applying (\ref{eq:STU2_trees}) to the two different outputs yields the desired $\STU^2$ relation.

This concludes the proof that $\tilde{R}$ factors through $R_n$ as in the statement of this theorem. Those combine into a surjective $\Q$-linear morphism $R\colon\Ls^{\FI}(m)_{\Q}\rightarrow\Ls L(m)_{\Q}$.

\textbf{The map $R$ is compatible with the Lie brackets:} Pick two tree diagrams $T,T'$ of respective degrees $k,l$, let $n\coloneqq k+l$. Denote by the same letters tree claspers realizing them.
    \begin{figure}[!ht]
    \centering
    \includegraphics[width=0.55\textwidth]{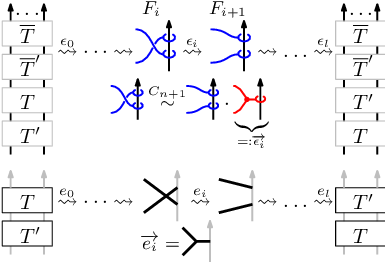}
    \caption{Commutator of string links (top) versus commutator of tree diagrams (bottom)}
    \label{fig:realization_map_preserves_bracket}
\end{figure}

    The commutator\footnote{We omit $\cdot$ when writing products: we write $TT'$ for $T\cdot T'$.} $[R(T),R(T')]$ is equal to the $C_{n+1}$-equivalence class of $\sigma(\overline{T}\overline{T'}TT')$ where $\overline{T}$ represents the inverse of $T$, idem for $T'$, the product is concatenation of tree claspers and $\sigma$ denotes surgery modulo $C_{n+1}$-equivalence. The strategy is to go from the clasper $\overline{T}\overline{T'}TT'$ to the clasper $\overline{T}\overline{T'}T'T$ by a sequence $\epsilon_0,\dots,\epsilon_{r-1}$ of slide moves (Section \ref{sec:common_moves}), see Figure \ref{fig:realization_map_preserves_bracket}(top). On one hand, $\overline{T}\overline{T'}T'T\overset{C_{n+1}}{\sim}1$. On the other hand, each slide move creates a red tree $\overrightarrow{\epsilon_i}$. We can separate those new red tree claspers since they have degree $n$, hence they commute with all the rest. All in all, we obtain
    \begin{equation}\label{eq:commutator_of_tree_claspers}
        [\sigma(T),\sigma(T')]=\sigma(\overline{T}\overline{T'}TT')={\sum}_{i=1}^{r-1}\sigma(\overrightarrow{\epsilon_i}).
    \end{equation}

    Let us compare (\ref{eq:commutator_of_tree_claspers}) with the commutator of the original tree diagrams $T,T'$. We can follow the exact same path of slide moves on the diagrammatic level to get
    \begin{equation}\label{eq:commutator_of_tree_claspers_diagrams}
        [T',T]=\overrightarrow{(TT')(T'T)}={\sum}_{i=1}^{r-1}\overrightarrow{e_i}.
    \end{equation}
    by Definition \ref{def:vector_from_F_to_F'_label}.

    At this point, one may think that there is a mistake. Indeed, in (\ref{eq:commutator_of_tree_claspers}) and (\ref{eq:commutator_of_tree_claspers_diagrams}) the roles of $T$ and $T'$ are reversed. The answer is given by the mysterious sign in the definition of the realization map. More precisely, the tree claspers realizing $T,T'$ both have an additional half-twist introduced somewhere. Thus, the red tree clasper $\overrightarrow{\epsilon_i}$ has two half-twists, or zero, modulo $C_{n+1}$. This means that $R(\overrightarrow{e_i})=-\sigma(\overrightarrow{\epsilon_i})$ and therefore
    \begin{equation}\label{eq:realization_map_preserves_bracket}
        R([T,T'])=-R([T',T])=-{\sum}_iR(\overrightarrow{e_i})={\sum}_i\sigma(\overrightarrow{\epsilon_i})=[R(T),R(T')],
    \end{equation}
    which concludes the proof that $R$ is a surjective morphism of Lie algebras over $\Q$.
\end{proof}

\begin{remark}
    The above proof already produces a surjective morphism of Lie algebras over $\Z$. We plan on investigating the story over $\Z$ in future work.
\end{remark}

\subsection{The Kontsevich integral and the tree preservation theorem}\label{chap:Kontsevich integral}
The \emph{Kontsevich integral} is the last ingredient of our identification of the rational Goussarov--Habiro Lie algebra. Discovered by Kontsevich \cite{kontsevich1993vassiliev} as a universal Vassiliev invariant over $\Q$, it is a morphism
$$Z\colon L(m)\longrightarrow\widehat{\As^{\FI}}(m)_{\Q},$$
which to a string link $\gamma\in L(m)$ associates an element $Z(\gamma)$ in the graded completion $\widehat{\As^{\FI}}(m)_{\Q}$ of $\As^{\FI}(m)_{\Q}$. The Kontsevich integral of a string link $Z(\gamma)$ can be thought of as a \emph{series expansion} of $\gamma$, such that degree $\leq n$ Vassiliev invariants only depend on the degree $\leq n$ part $Z_{\leq n}(\gamma)$. However, it is an open question whether $Z$ is injective on $L(m)$, i.e.\ whether the Kontsevich integral classifies string links.

Since the Kontsevich integral turns string links into diagrams, it is a good candidate for inducing an inverse to the realization map $R$.

\subsubsection{The combinatorial Kontsevich integral.}\label{sec:combinatorial_Kontsevich_integral}
This section gathers results from Bar-Natan \cite{bar1995vassiliev,bar1996tangles} and Kontsevich \cite{kontsevich1993vassiliev}, Le and Murakami \cite{le1996universal} as well as the book by Chmutov, Duzhin and Mostovoy \cite{chmutov2012introduction}. See also the nice expository papers \cite{bar1997fundamental,chmutov2005kontsevich}.

This section gathers the main properties of the combinatorial version of the Kontsevich integral \cite{bar1996tangles}, to which we stick from now on. The combinatorial Kontsevich integral is understood as a functor whose domain is the category $\PTangles$ of \emph{parenthesized tangles} up to isotopy (denoted $\textbf{\text{PT}}$ in \cite{bar1996tangles}), i.e.\ isotopy classes of unframed tangles with a choice of parenthesizing of their endpoints, also referred to as \emph{q-tangles} \cite{le1996universal}. Its target is the category $\widehat{\Diag}_{\Q}$ of uni-trivalent diagrams \cite[Def.\ 3.1]{bar1996tangles} on tangle skeletons, whose objects are simply $\updownarrow$-words and morphisms are formal series in $\widehat{\As^{\FI}}(S)_{\Q}$ for the adequate tangle skeleton $S$. Composition in $\PTangles$ is simply vertical concatenation. Composition in $\widehat{\Diag}_{\Q}$ is given by concatenation of the tangle skeleta and concatenation of the diagrams (extended linearly).

\begin{theorem}[\cite{bar1996tangles,chmutov2005kontsevich,drinfeld1991quasi}]\label{thm:properties_combinatorial_Kontsevich_integral}
    The Kontsevich integral of unframed parenthesized tangles
    \begin{equation}
        Z\colon\PTangles\longrightarrow\widehat{\Diag}_{\Q}
    \end{equation}
    has the following properties:
    \begin{enumerate}[label={\upshape(\roman*)}]
    \item\label{it:Z_is_a_functor} It is a \emph{functor}: If $\gamma=\id_{w}$ for some parenthesized $\updownarrow$-word $w$, i.e.\ it consists of parallel vertical strands, then $Z(\gamma)=1\in\widehat{\As^{\FI}}(m)$ is the empty diagram. If $\gamma_1,\gamma_2$ are two concatenable tangles (i.e.\ composable morphisms in $\PTangles$), then $Z(\gamma_1\circ \gamma_2)=Z(\gamma_1)\circ Z(\gamma_2)$.
    \item\label{it:Z_is_monoidal} It is a \emph{monoidal} functor, where the monoidal structure on both sides is given by juxtaposition, i.e.\ horizontal concatenation\footnote{In $\PTangles$, the parenthesizing of the tensor product is given by: $u \otimes v\coloneqq ((u)(v))$.}: $Z(\gamma_1\otimes \gamma_2)=Z(\gamma_1)\otimes Z(\gamma_2)$.
    \item\label{it:Z_on_braiding_and_associators} It takes the following values on the standard \emph{braiding}, \emph{cup} and \emph{cap} parenthesized tangles:
    \begin{align}
        \PTangles(\uparrow\uparrow,\uparrow\uparrow)\ni\vcenter{\hbox{\includegraphics{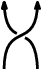}}}\quad&\overset{Z}{\longmapsto}\quad\vcenter{\hbox{\includegraphics{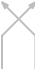}}}+\frac{1}{2}\vcenter{\hbox{\includegraphics{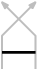}}}+\frac{1}{8}\vcenter{\hbox{\includegraphics{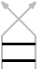}}}+\dots\nonumber\\&\quad\quad\quad\quad\quad=\vcenter{\hbox{\includegraphics{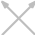}}}\circ\exp\paren{\frac{1}{2}\vcenter{\hbox{\includegraphics{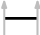}}}},\\
        \PTangles(\uparrow\uparrow,\uparrow\uparrow)\ni\vcenter{\hbox{\includegraphics{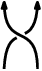}}}\quad&\overset{Z}{\longmapsto}\quad\vcenter{\hbox{\includegraphics{Ressources/Kontsevich_integral/basic_tangles/diagram_cross_empty.eps}}}\circ\exp\paren{\frac{-1}{2}\vcenter{\hbox{\includegraphics{Ressources/Kontsevich_integral/basic_tangles/diagram_1_chord.eps}}}},\\
        \PTangles(\uparrow\downarrow,\emptyset)\ni\vcenter{\hbox{\includegraphics{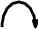}}}\quad&\overset{Z}{\longmapsto}\quad \vcenter{\hbox{\includegraphics{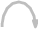}}},\\
        \PTangles(\emptyset,\downarrow\uparrow)\ni\vcenter{\hbox{\includegraphics{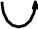}}}\quad&\overset{Z}{\longmapsto}\quad \vcenter{\hbox{\includegraphics{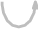}}}.
    \end{align}
    \item\label{it:Z_and_reversing} It is compatible with the operation of \emph{reversing a strand}. If $\textbf{R}_C\gamma$ is obtained from $\gamma$ by reversing the orientation of the component $C$, then
    $$Z(\textbf{R}_C\gamma)=(\textbf{R}_C)_*Z(\gamma)$$
    where $(R_C)_*\colon\widehat{\As^{\FI}}(S)_{\Q}\rightarrow\widehat{\As^{\FI}}(\textbf{R}_CS)_{\Q}$ is defined as follows. For a diagram $D$ on the tangle skeleton $S=S(\gamma)$, $$(\textbf{R}_C)_*(D)=(-1)^{\#\{\text{endpoints on }C\}}D'\in\widehat{\As^{\FI}}(\textbf{R}_CS)_{\Q}$$
    where $D'$ is the same diagram as $D$ on the tangle skeleton $\textbf{R}_CS=S(\textbf{R}_C\gamma)$.
    \item\label{it:Z_and_doubling} It is compatible with the operation of \emph{doubling a strand}. If $\textbf{D}_C\gamma$ is obtained from $\gamma$ by doubling\footnote{If $\uparrow$ is an endpoint of $C$ one of the end $\updownarrow$-words, then it is replaced by $(\uparrow\uparrow)$ in $\textbf{D}_C\gamma$, and similarly for $\downarrow$.} the strand $C$, then
    $$Z(\textbf{D}_C\gamma)=(\textbf{D}_C)_*Z(\gamma)$$
    where $(D_C)_*\colon\widehat{\As^{\FI}}(S)_{\Q}\rightarrow\widehat{\As^{\FI}}(\textbf{D}_CS)_{\Q}$ is defined as follows. For a diagram $D$ on the tangle skeleton $S=S(\gamma)$,
    $$(\textbf{D}_C)_*(D)=\sum_{D'\text{ lifts }D}D'\in\widehat{\As^{\FI}}(\textbf{R}_CS)$$
    is the sum of all the $2^{\#\{\text{endpoints on }C\}}$ ways of lifting $D$ to a diagram on $\textbf{D}_C\gamma$. For example:
    \begin{align*}
        (\textbf{D}_{3})_*\paren{\vcenter{\hbox{\includegraphics[height=30pt]{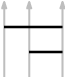}}}}=\vcenter{\hbox{\includegraphics[height=30pt]{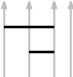}}}+\vcenter{\hbox{\includegraphics[height=30pt]{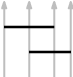}}}+\vcenter{\hbox{\includegraphics[height=30pt]{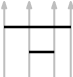}}}+\vcenter{\hbox{\includegraphics[height=30pt]{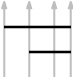}}},
    \end{align*}
    where $\textbf{D}_3$ is the operation of doubling the third strand.
    \end{enumerate}
\end{theorem}

As explained in \cite[Section 10.3.2]{chmutov2012introduction}, any parenthesized tangle can be decomposed into a product of tensor products of basic tangles. Those basic tangles are the standard tangles from Theorem \ref{thm:properties_combinatorial_Kontsevich_integral}(iii), together with \emph{associator tangles}\footnote{It should be mentioned that the construction of a combinatorial Kontsevich integral $Z$ actually depends on the choice of a \emph{Drinfeld's associator}, of which there is a deep theory (see \cite{chmutov2012introduction} for more details and references). We omit this here since this choice has no influence for our purpose.}, or images of them under a sequence of reversing and doubling operations. The value of the Kontsevich integral $Z$ on each basic tangle is determined by \ref{thm:properties_combinatorial_Kontsevich_integral}(iii,iv,v), and the value of $Z$ on the original parenthesized tangle is obtained by the multiplicativity properties from \ref{thm:properties_combinatorial_Kontsevich_integral}(i,ii).

Denote by $Z_{\leq n}$, $Z_{n}$ and $Z_{\geq n}$ the compositions $\proj\circ Z$ where $\proj$ is the projection onto the corresponding direct summands.

\subsubsection{Tree preservation theorem.}
\begin{theorem}\label{thm:tree_preservation_theorem}
    Let $n\geq 1$ and $T\in\Ds^T_n(m)$ a degree $n$ tree diagram on $m$ strands. Then
    \begin{equation}\label{eq:tree_preservation_theorem}
        Z(R(T))=1+T+O(n+1),
    \end{equation}
    where $R(T)=\sigma(C_T)$ is the result of surgery along a tree clasper $C_T$ realizing $T$ (see Theorem \ref{prop:realizing_trees}), and the notation $O(n+1)$ means that the equality holds modulo terms of degree $\geq n+1$. In particular, the Kontsevich integral induces an inverse to $R$, which then provides an isomorphism of graded Lie $\Q$-algebras $\Ls^{\FI}(m)_{\Q}\cong\Ls L(m)_{\Q}$.
\end{theorem}
\begin{proof}
    We proceed by induction on $n$. Moreover, at each step, we first make the simplifying assumption (\textbf{A}): $m=n+1$, $T$ has one leg on each strand and (if $n>1$) the legs attached to the rightmost two strands are adjacent to a common node. Then we show the result for the general case.

    For $n=1$ and assuming (\textbf{A}), there is only one possible diagram $T\in\Ds_1^T(2)$ with a leg on each strand. Its realization is a full positive braiding between the two strands, hence we obtain the desired result by Theorem \ref{thm:properties_combinatorial_Kontsevich_integral}(\ref{it:Z_on_braiding_and_associators}):
    \begin{align*}
        \vcenter{\hbox{\includegraphics[scale=1.2]{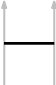}}}\overset{R}{\longmapsto}\sigma\paren{\vcenter{\hbox{\includegraphics[scale=1.2]{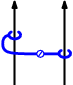}}}}=\vcenter{\hbox{\includegraphics[scale=1.2]{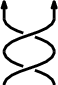}}}\overset{Z}{\longmapsto}&\paren{\vcenter{\hbox{\includegraphics[scale=1.2]{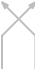}}}+\frac{1}{2}\vcenter{\hbox{\includegraphics[scale=1.2]{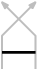}}}+O(2)}\circ\paren{\vcenter{\hbox{\includegraphics[scale=1.2]{Ressources/Kontsevich_integral/base-case-TPT/empty-diagram.eps}}}+\frac{1}{2}\vcenter{\hbox{\includegraphics[scale=1.2]{Ressources/Kontsevich_integral/base-case-TPT/chord-on-x.eps}}}+O(2)}\\
        &=\vcenter{\hbox{\includegraphics[scale=1.2]{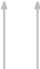}}}+\vcenter{\hbox{\includegraphics[scale=1.2]{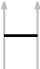}}}+O(2).
    \end{align*}

    Now we remove the assumption (\textbf{A}), hence the single chord of $T$ can have its legs attached anywhere on the $m$ strands. Consider a clasper $C_T$ realizing $T$, as in the construction of the realizing map $R$, see Figure \ref{fig:realizing_a_tree}. By an isotopy, we can pull the clasper out and bring it into a small box, such that the box contains $2$ parallel vertical strands and the clasper $C_T$ is attached to them satisfying (\textbf{A}):
    \begin{equation}\label{eq:TPT_degree_1_pull_out}
        \vcenter{\hbox{\includegraphics[scale=1.2]{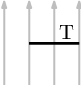}}}\;\rightsquigarrow\;\vcenter{\hbox{\includegraphics[scale=1.2]{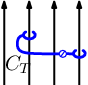}}}\simeq\vcenter{\hbox{\includegraphics[scale=1.2]{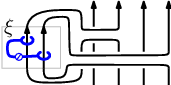}}}.
    \end{equation}
    The grey box consists of a string link $\xi$ with a degree $1$ tree clasper $C_T$ on it, satisfying (\textbf{A}). Thus we compute
    \begin{align*}
        Z&\paren{R\paren{\vcenter{\hbox{\includegraphics[scale=1.2]{Ressources/Kontsevich_integral/pulling-claspers-out/diagram-T.eps}}}}}=Z\paren{\sigma\paren{\vcenter{\hbox{\includegraphics[scale=1.2]{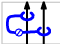}}}}}\circ Z\paren{\vcenter{\hbox{\includegraphics[scale=1.2]{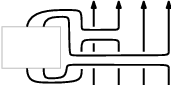}}}}\\
        &\overset{(\textbf{A})}{=}\paren{Z\paren{\vcenter{\hbox{\includegraphics[scale=1.2]{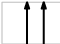}}}}+\vcenter{\hbox{\includegraphics[scale=1.2]{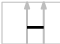}}}+O(2)}\circ Z\paren{\vcenter{\hbox{\includegraphics[scale=1.2]{Ressources/Kontsevich_integral/pulling-claspers-out/tangle-box.eps}}}}\\
        &=Z\paren{\vcenter{\hbox{\includegraphics[scale=1.2]{Ressources/Kontsevich_integral/pulling-claspers-out/box-straight.eps}}}}\circ Z\paren{\vcenter{\hbox{\includegraphics[scale=1.2]{Ressources/Kontsevich_integral/pulling-claspers-out/tangle-box.eps}}}}+\vcenter{\hbox{\includegraphics[scale=1.2]{Ressources/Kontsevich_integral/pulling-claspers-out/box-chord.eps}}}\circ\vcenter{\hbox{\includegraphics[scale=1.2]{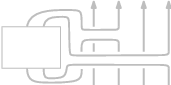}}}+O(2)\\
        &=Z\paren{\vcenter{\hbox{\includegraphics[scale=1.2]{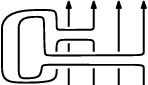}}}}+\vcenter{\hbox{\includegraphics[scale=1.2]{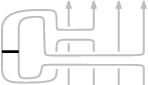}}}+O(2)=1+\vcenter{\hbox{\includegraphics[scale=1.2]{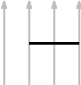}}}+O(2)
    \end{align*}
    where the first equality uses Figure \ref{eq:TPT_degree_1_pull_out} and multiplicativity. Then the left factor can be computed by the above since (\textbf{A}) is satisfied in the small grey box. Then we distribute the product over the sum on the left to obtain the second line, where everything that has degree $\geq 2$ is contained in $O(2)$. Use multiplicativity again to reconstruct the first term as the Kontsevich integral of a string link isotopic to the unlink, which becomes $1$, while the second term gives the diagram $T$ when attached to the outside. This concludes the proof of the case $n=1$.

    Now we show the induction step. Let $n\geq 2$ and assume that the statement holds in degree $n-1$.

    First, consider a degree $n$ tree diagram $T\in\Ds^T_n(n+1)$ on $n+1$ strands and satisfying (\textbf{A}). It looks like the diagram on the left below, where the black dashed rectangle contains the $n-1$ other legs. Realize it as a tree clasper $C_T$, assuming that the additional half-twist is in the blue blob.
    \begin{align*}
        \vcenter{\hbox{\includegraphics[scale=1.2]{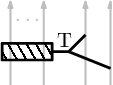}}}\:\rightsquigarrow\;\vcenter{\hbox{\includegraphics[scale=1.2]{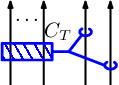}}}.
    \end{align*}
    The core of the induction step is to apply Habiro's move 2 and 10 \cite[pages 14,15]{habiro2000claspers} to the visible tripod:
    \begin{align*}
        \vcenter{\hbox{\includegraphics[scale=1.3]{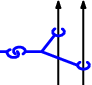}}}\;\overset{H10}{\sim}\;\vcenter{\hbox{\includegraphics[scale=1.3]{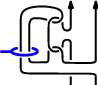}}}\;\simeq\:\vcenter{\hbox{\includegraphics[scale=1.3]{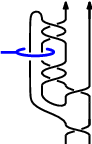}}}.
    \end{align*}
    The right-hand-side is obtained by isotopy. The new tree clasper $C'$ ($C_{T}$ minus the tripod) realizes a degree $n-1$ tree diagram $T'$.

    Thus, to compute $R(T)$ we can first apply moves $2$ and $10$ as above, then perform surgery on $C'$. This surgery is entirely contained in the grey rectangle shown in the first line below, and we obtain that first equality by multiplicativity of $Z$:
    \begin{align*}
        Z&(R(T))=Z\paren{\vcenter{\hbox{\includegraphics[scale=1.2]{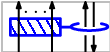}}}}\circ Z\paren{\vcenter{\hbox{\includegraphics[scale=1]{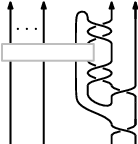}}}}\\
        &=\paren{1+(R_{n+1})_*(D_n)_*\paren{\vcenter{\hbox{\includegraphics[scale=1.2]{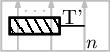}}}}+O(n)}\circ\paren{1+Z_1\paren{\vcenter{\hbox{\includegraphics[scale=1]{Ressources/Kontsevich_integral/induction-TPT/tangle-without-box.eps}}}}+O(2)}\\
        &=\underbrace{\paren{(R_{n+1})_*(D_n)_*\paren{\vcenter{\hbox{\includegraphics[scale=1.2]{Ressources/Kontsevich_integral/induction-TPT/Tprime.eps}}}}}\circ Z_1\paren{\vcenter{\hbox{\includegraphics[scale=1]{Ressources/Kontsevich_integral/induction-TPT/tangle-without-box.eps}}}}}_{(1)}\\ &\quad+\underbrace{\vcenter{\hbox{\includegraphics[scale=1.2]{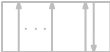}}}\circ Z\paren{\vcenter{\hbox{\includegraphics[scale=1]{Ressources/Kontsevich_integral/induction-TPT/tangle-without-box.eps}}}}}_{(2)}+\underbrace{Z_{\geq 1}\paren{\vcenter{\hbox{\includegraphics[scale=1.2]{Ressources/Kontsevich_integral/induction-TPT/clasper-in-box.eps}}}}\circ\vcenter{\hbox{\includegraphics[scale=1]{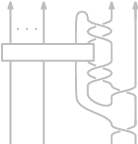}}}}_{(3)}\;+\; O(n+1).
    \end{align*}
    From the first to the second line, the induction hypothesis $Z(R(T'))=1+T'+O(n)$ is used, with additional strand doubling and reversing. Then the third line is obtained by distributing this product\footnote{Observe that the $1$'s are not all the same, they stand for the empty diagram on the adequate tangle skeleton.}, and putting all the terms of degree $\geq n+1$ inside $O(n+1)$. The resulting three terms are called (1), (2) and (3). Let us compute them:
    \begin{enumerate}
        \item The first term is a product of degree $(n-1)$ terms with degree $1$ terms. First, we compute those two factors individually. Since we are only computing the degree $1$ of $Z$, only the braidings matter. The first factor is obtained by doubling the $n$-th strand and reversing the new $(n+1)$-st strand
        \begin{align*}
            (R_{n+1})_*(D_n)_*\paren{\vcenter{\hbox{\includegraphics[scale=1.2]{Ressources/Kontsevich_integral/induction-TPT/Tprime.eps}}}}=(R_{n+1})_* &\paren{\vcenter{\hbox{\includegraphics[scale=1.2]{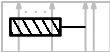}}}+\vcenter{\hbox{\includegraphics[scale=1.2]{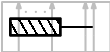}}}}\\
            &=\vcenter{\hbox{\includegraphics[scale=1.2]{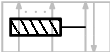}}}-\vcenter{\hbox{\includegraphics[scale=1.2]{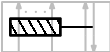}}}
        \end{align*}
        and the second factor is a sum of Kontsevich integrals of braidings
        \begin{align*}
            Z_1\paren{\vcenter{\hbox{\includegraphics[scale=1]{Ressources/Kontsevich_integral/induction-TPT/tangle-without-box.eps}}}}&=\vcenter{\hbox{\includegraphics[scale=1]{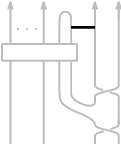}}}-\vcenter{\hbox{\includegraphics[scale=1]{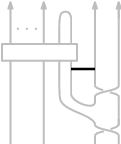}}}+\frac{1}{2}\vcenter{\hbox{\includegraphics[scale=1]{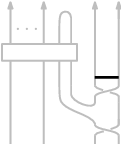}}}-\frac{1}{2}\vcenter{\hbox{\includegraphics[scale=1]{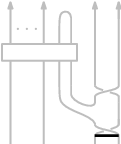}}}.
        \end{align*}
        Combining them yields
        \begin{align*}
            (1)&=\vcenter{\hbox{\includegraphics[scale=1]{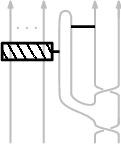}}}-\vcenter{\hbox{\includegraphics[scale=1]{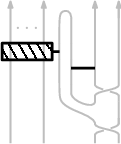}}}+\frac{1}{2}\vcenter{\hbox{\includegraphics[scale=1]{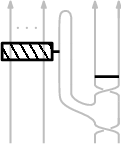}}}-\frac{1}{2}\vcenter{\hbox{\includegraphics[scale=1]{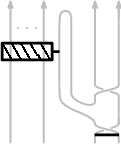}}}\\
            &-\vcenter{\hbox{\includegraphics[scale=1]{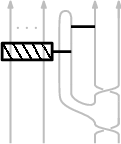}}}+\vcenter{\hbox{\includegraphics[scale=1]{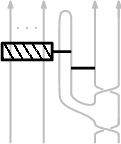}}}-\frac{1}{2}\vcenter{\hbox{\includegraphics[scale=1]{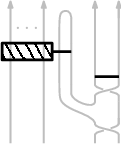}}}+\frac{1}{2}\vcenter{\hbox{\includegraphics[scale=1]{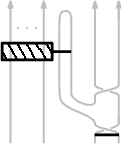}}}\\
            &=\vcenter{\hbox{\includegraphics[scale=1]{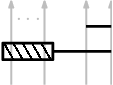}}}-\vcenter{\hbox{\includegraphics[scale=1]{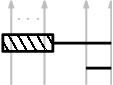}}}\overset{\STU}{=}\vcenter{\hbox{\includegraphics[scale=1]{Ressources/Kontsevich_integral/induction-TPT/tree_T.eps}}}
        \end{align*}
        where the last three columns vanish, since the leg attached to the blob is free to slide across the maximum. In the first column, the additional chord prevents that leg from sliding. This last column gives the tree diagram $T$ by $\STU$.
        \item This part is easier. Indeed, the left factor is $Z$ applied to a trivial string link (but whose rightmost strand is reversed). By mutliplicativity of $Z$, we can combine the two factors into
        \begin{equation*}
            Z\paren{\vcenter{\hbox{\includegraphics{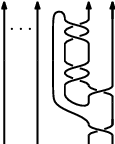}}}}=Z\paren{\vcenter{\hbox{\includegraphics{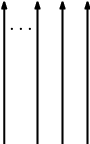}}}}=1.
        \end{equation*}
        \item Since terms (1) and (2) equal $T$ and $1$, respectively, we expect (3) to vanish. Although the first factor might be very complicated, all we know is that it is obtained from a sum of diagrams by doubling the last strand, then reversing the new last strand. This fact alone actually suffices to show that (3) vanishes.
        
        Indeed, suppose that $D$ is any nonempty diagram on $n$ strands, with $k\geq 1$ legs attached to the last strand. Doubling the last strand turns that diagram into a sum of all its $2^k$ lifts. Such a lift is indexed by a word $e_1\dots e_k\in\{0,1\}^{k}$ where a $e_i=0$, resp.\ $e_i=1$, means that the $i$-th leg is attached to the $n$-th, resp.\ $n+1$-st strand. Reversing the last strand has the effect of mutliplying by $\pm 1$, depending on the number of legs attached to the last strand. The resulting sum
        \begin{align*}
            (R_{n+1})_*(D_n)_*\paren{\vcenter{\hbox{\includegraphics[scale=1.2]{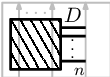}}}}=\sum_{e_2,\dots,e_k\in\{0,1\}}(-1)^{\sum_{j\geq 2}e_j}\paren{\vcenter{\hbox{\includegraphics[scale=1.2]{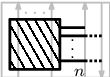}}}-\vcenter{\hbox{\includegraphics[scale=1.2]{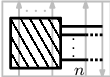}}}}
        \end{align*}
        is a sum of differences of the form "$e_1=0$" minus "$e_1=1$". Multiplying each such difference with the second factor of (3) gives $0$ since the first leg is now free to go from one position to the other:
        \begin{align*}
            \paren{\vcenter{\hbox{\includegraphics[scale=1.2]{Ressources/Kontsevich_integral/induction-TPT/D-lift-left.eps}}}-\vcenter{\hbox{\includegraphics[scale=1.2]{Ressources/Kontsevich_integral/induction-TPT/D-lift-right.eps}}}}\circ\vcenter{\hbox{\includegraphics{Ressources/Kontsevich_integral/induction-TPT/empty-without-box.eps}}}=0.
        \end{align*}

        The above shows that all terms in $Z(R(T'))$ that have at least one leg on the last strand vanish once we apply $(R_{n+1})_{*}(D_n)_*$ and multiply the result with the second factor.
        
        In fact, all the terms in $Z(R(T'))$ must have at least 1 leg on the last strand. Indeed, by definition of the Kontsevich integral, the sum of all terms with no leg on the last strand is equal to the Kontsevich integral of $R(T')$ minus the last strand. If the last strand is removed, then $T'$ gets a trivial leaf and the surgery is trivial. In other words, this comes from the fact that a string link obtained by surgery along a tree clasper with exactly one leg on each strand is a \emph{Brunnian string link} \cite{habiroMeilhan2006Brunnian}.

        This concludes the proof that (3) vanishes.
    \end{enumerate}
    By the above computations of (1,2,3), we obtain
    $$Z(R(T))=1+T+O(n+1)$$
    as desired. This concludes the case (\textbf{A}) in degree $n$.

    The general case is shown exactly as in degree $1$, by pulling the clasper out to extract a box containing the clasper (as in (\ref{eq:TPT_degree_1_pull_out})), in which assumption (\textbf{A}) is satisfied. This concludes the induction step, and the proof of (\ref{eq:tree_preservation_theorem}).

    It remains to show that $Z$ induces an inverse to $R$.

    By (\ref{eq:tree_preservation_theorem}), the degree $n$ part of $Z$ induces a map $Z_n\colon L_n(m)\rightarrow\As^{\FI}_n(m)_{\Q}$, which is constant on $C_{n+1}$-equivalence classes. Thus it factors through
    $$Z^{GH}_n\colon\Ls_n L(m)\rightarrow\As^{\FI}_n(m)_{\Q},$$
    whose image lies in $\Prim(\As^{\FI}(m)_{\Q})_n$. Identifying the latter space with $\Ls_n^{\FI}(m)_{\Q}$ by Theorem \ref{thm:presentation_primitive_Lie_algebra}, direct summing all the degrees at once and tensoring with $\Q$, we finally obtain
    $$Z^{GH}\colon\Ls L(m)_{\Q}\rightarrow\Ls^{\FI}_n(m)_{\Q},$$ which is an inverse to $R$ by (\ref{eq:tree_preservation_theorem}).
\end{proof}

\section{Application to the rational Goussarov--Habiro conjecture}\label{sec:Application_to_rational_GHC}
\subsection{The rational Goussarov--Habiro conjecture}
\subsubsection{Vassiliev filtration and associated graded algebra.}
The monoid ring\footnote{If $M$ is a monoid and $R$ a ring, the \emph{monoid ring} $RM$ is the abelian group of formal sums $\sum_{i=1}^n r_i x_i$, where $r_i\in R$, $x_i\in M$, wih multiplication given by $rx\cdot r'x'=rr' xx'$ and extended linearly.} $\Z L(m)$ of string links is filtered by the \emph{Vassiliev filtration} (\ref{eq:Vassiliev_filtration}), with associated graded algebra $\As L(m)$ (\ref{eq:graded_algebra_associated_with_Vassiliev_filtration}):
\begin{align}
    \Z L(m)&\supset \del_1(\Z S_{1} L(m))\supset\del_2(\Z S_{2} L(m))\supset\dots,\label{eq:Vassiliev_filtration}\\
    \As L(m)&\coloneqq\bigoplus_{n\geq 0}\As_nL(m)\coloneqq\bigoplus_{n\geq 0}\frac{\del_n(\Z S_{n} L(m))}{\del_{n+1}(\Z S_{n+1} L(m))},\label{eq:graded_algebra_associated_with_Vassiliev_filtration}
\end{align}
where $S_{\bullet} L(m)$ denotes the graded monoid of (isotopy classes of) singular\footnote{Singularities are only allowed to be transversal double points, and isotopies must preserve the double points.} string links, graded by the number of double points. The $\Z$-linear maps $\del_n\colon\Z S_n L(m)\rightarrow \Z L(m)$ are defined by applying the following \emph{Skein relation} \cite{birman1993knot} to each double point
\begin{equation}\label{eq:skein_relation}
    \includegraphics[width=25pt,valign=c]{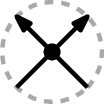}\longmapsto\includegraphics[width=25pt,valign=c]{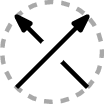}-\includegraphics[width=25pt,valign=c]{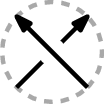},
\end{equation}
where it is understood that the terms are identical outside of the part shown. Thus, for $\gamma\in S_nL(m)$ a singular string link with $n$ double points, its image $\del_n(\gamma)\in\Z L(m)$ is an alternating sum of all possible resolutions of the $n$ double points. 

A string link invariant $V\colon L(m)\rightarrow A$ valued in an abelian group $A$ naturally extends to a $Z$-linear map $V\colon\Z L(m)\rightarrow A$. We say that $V$ is a \emph{Vassiliev invariant} or \emph{finite type invariant} of degree $n$ if it vanishes on $\del_{n+1}(\Z S_{n+1} L(m))$. In particular, a Vassiliev invariant of degree $n$ induces a functional on $\As_nL(m)$. Two string links $\gamma,\gamma'\in L(m)$ are \emph{$V_n$-equivalent} if they cannot be distinguished by Vassiliev invariants of degree $<n$ or, equivalently, if $\gamma-\gamma'$ belongs to $\del_{n}(\Z S_{n} L(m))$.

There is a realization map $R^v\colon\As^{\FI}(m)\epic\As L(m)$ (the $v$ refers to \emph{Vassiliev}), which is a morphism of $\Z$-algebras. Originally, $R^v$ was defined on the presentation of $\As^{\FI}(m)$ by chord diagrams, by sending a chord diagram to a singular string link realizing it \cite{bar1995vassiliev}. However, one can define this realization map on Feynman (forest) diagrams using claspers\footnote{This defines the same realization map since degree $1$ clasper surgeries are crossing changes} \cite[Section 6]{habiro2000claspers}. This was made precise by Conant and Teichner in the proof of \cite[Theorem 1.1]{conantteichner2004gropefeynmandiagrams}. Following their argument, we can construct a surjective\footnote{Surjectivity follows from Section 6 in \cite{habiro2000claspers}, where Habiro defines the Vassiliev filtration using forest claspers instead of singular string links.} morphism of graded $\Z$-algebras
\begin{equation}\label{eq:Vassiliev_realization_map}
    R^v\colon\As^{\FI}(m)\epic\As L(m)
\end{equation}
which in degree $n$ sends a forest diagram $F\in\Ds_n(m)$ to $$[\gamma_0;T_1,\dots,T_s]\coloneqq \sum_{\Ss'\subset\{T_1,\dots,T_s\}}(-1)^{s-\abs{\Ss'}}(\gamma_0^{\Ss'})\in\As_n L(m)$$ where $T_1\cup\dots\cup T_s$ is a forest clasper realizing $F$, $\gamma_0$ denotes the trivial string link on $m$ strands and $\gamma^{\Ss'}_0$ denotes the result of clasper surgery on $\gamma_0$ along the trees contained in $\Ss'$. That $R^v$ is well-defined follows from \cite[Section 6]{habiro2000claspers} and \cite{conantteichner2004gropefeynmandiagrams}.

\begin{remark}
    A forest clasper realizing a forest diagram $F$ is constructed as in the proof of Proposition \ref{prop:realizing_trees}, but with multiple trees at the same time and with an additional positive half-twist near a chosen leaf of each tree. Those half-twits are important in order for $\STU$ relations to be realized with correct signs. This explains the footnote on page 68 in \cite{habiro2000claspers}.
\end{remark}

Over $\Q$, the realization map $R^v\colon\As^{\FI}(m)_{\Q}\xrightarrow{\cong}\As L(m)_{\Q}$ is an isomorphism, with inverse induced by the Kontsevich integral \cite{kontsevich1993vassiliev,bar1995vassiliev}.

\subsubsection{The Goussarov--Habiro conjecture}\label{sec:GH_conjecture}
From Habiro's definition of the Vassiliev filtration using claspers\footnote{And from Goussarov's work as well.}, it follows that $C_n$-equivalence implies $V_n$-equivalence for string links in $L(m)$. Goussarov and Habiro independently conjectured the converse:
\begin{conjecture*}[Goussarov--Habiro conjecture for string links in the cylinder]
    For all $m\geq 1$ and $n\geq 1$, any $V_n$-equivalent string links on $m$ strands are $C_n$-equivalent.
\end{conjecture*}

Since $C_n$-equivalence implies $V_n$-equivalence, any $C_n$-trivial string link $\gamma\in L_n(m)$ is automatically $V_n$-equivalent to the trivial string link $1$, hence $\gamma-1$ belongs to $\del_n(\Z S_nL(m))$. This yields the \emph{comparison map} \cite[Section 8.2]{habiro2000claspers}
\begin{equation*}
    \chi\colon\Ls L(m)\rightarrow\As L(m)\colon [\gamma]\mapsto [\gamma-1],
\end{equation*}
which is injective if and only if the Goussarov--Habiro conjecture is true. In particular, it is injective when $m=1$ by the Goussarov--Habiro theorem for knots \cite[Theorem 6.18]{habiro2000claspers}.

Exploiting the \emph{dimension subgroup property} in characteristic zero, Massuyeau showed the following rational version of the Goussarov--Habiro conjecture:
\begin{theorem}[Massuyeau \cite{massuyeau2006finitetype}]\label{thm:Massuyeau}
    Over $\Q$, the map $\chi\colon\Ls L(m)_{\Q}\rightarrow\As L(m)_{\Q}$  is injective.
\end{theorem}

The realization maps $R,R^v$ and the comparison map $\chi$ combine into the following commutative square of graded (Lie or Hopf) $\Z$-algebras:
\begin{equation}\label{eq:integral_diagram}
    \begin{tikzcd}
	{\Ls^{\FI}(m)} & {\Prim(\As^{\FI}(m))} & {\As^{\FI}(m)} \\
	{\Ls L(m)} && {\As L(m)}
	\arrow["\chi", from=2-1, to=2-3]
	\arrow["{R}"', two heads, from=1-1, to=2-1]
	\arrow["{R^v}",two heads, from=1-3, to=2-3]
	\arrow[from=1-1, to=1-2]
	\arrow[hook, from=1-2, to=1-3]
\end{tikzcd}
\end{equation}
When tensoring everything with $\Q$, we obtain:
\[\begin{tikzcd}
	{\Ls^{\FI}(m)_{\Q}} & {\Prim(\As^{\FI}(m)_{\Q})} & {\As^{\FI}(m)_{\Q}} \\
	{\Ls L(m)_{\Q}} && {\As L(m)_{\Q}}
	\arrow["\chi", hook, from=2-1, to=2-3]
	\arrow["{R}"',"\cong", two heads, from=1-1, to=2-1]
	\arrow["{R^v}"',"\cong", bend right = 25,two heads, from=1-3, to=2-3]
	\arrow["\cong"',"\iota", from=1-1, to=1-2]
	\arrow[hook, from=1-2, to=1-3]
	\arrow["{Z^{GH}}"', from=2-1, to=1-2]
	\arrow["{Z^{v}}"', bend right=25, from=2-3, to=1-3]
\end{tikzcd}\]
where $Z^v$ is the inverse to $R^v$ by Kontsevich's theorem, $\iota$ is an isomorphism by Theorem \ref{thm:presentation_primitive_Lie_algebra}, the left triangle commutes and $R$ becomes an isomorphism by Theorem \ref{thm:tree_preservation_theorem}. Therefore, $\chi$ is injective, which gives an alternative proof the rational Goussarov--Habiro conjecture.

\printbibliography[title={References}]
\end{document}